\documentclass{jT}

\usepackage{amssymb,amsmath,latexsym,amsthm}
\usepackage[english]{babel}

\newtheorem{theorem}{Theorem}
\newtheorem{lemma}{Lemma}
\newtheorem{proposition}{Proposition}

\newtheorem{notation}{Notation}
\theoremstyle{definition}
\newtheorem*{definition}{Definition}
\newtheorem*{remark}{Remark}
\newtheorem*{example}{Example}
\numberwithin{equation}{section}

\numberwithin{equation}{section}

\begin{document}

\title[A C.F. algorithm with Lagrange \& Dirichlet properties in dim $2$]{ A
two-dimensional continued fraction algorithm with Lagrange and Dirichlet
properties}

\author{\textsc{Christian Drouin}}
\address{Christian Drouin\\
26 Avenue d'Yreye\\
40 510 SEIGNOSSE FRANCE}
\email{http://www.christian.drouin@wanadoo.fr}

\subjclass[2000]{11 J 70, 11 J 13, 11 H 06}

\maketitle

\begin{resume}
On  d\'{e}montre dans cet article un Th\'{e}or\`{e}me
de Lagrange,
 pour un certain algorithme de fraction continue en dimension 2, dont la
 d\'{e}finition g\'{e}om\'{e}trique est tr\`{e}s naturelle. Des propri\'{e}t\'{e}s type Dirichlet sont aussi obtenues pour la convergence de cet algorithme.
Ces propri\'{e}t\'{e}s proviennent de caract\'{e}ristiques g\'{e}om\'{e}triques de l'algorithme. Les relations entre ces diff\'{e}rentes propri\'{e}t\'{e}s sont \'{e}tudi\'{e}es. En lien avec l'algorithme pr\'{e}sent\'{e}, sont rapidement \'{e}voqu\'{e}s les travaux de divers auteurs dans le domaine des fractions continues multidimensionnelles.
\end{resume}

\begin{abstr}
A Lagrange Theorem in dimension 2 is proved in this paper, for a particular two
dimensional continued fraction algorithm, with a very natural geometrical definition. Dirichlet type properties for the
convergence of this algorithm are also proved. These properties proceed from
a geometrical quality of the algorithm. The links between all these
properties are studied. In relation with this algorithm, some references are given to the works of various authors, in the 
domain of multidimensional continued fractions algorithms. 
\end{abstr}

\bigskip

\section{Introduction and results}

\subsection{Quick presentation of the main results}

Since the beginning of the theory of Multidimensional Continued Fractions,
an extension of the well known \emph{Lagrange Theorem} in dimension one has
been searched for. Historical remarks on the multidimensional continued
fractions (the Jacobi-Perron algorithm and others) can be found in the
works by F. Schweiger: \cite{Schwei} and \cite{Schw WL}, and A.J. Brentjes: 
\cite{Brentj}.

The classical one-dimensional continued fraction algorithm applied on a real
number $x$ generates a sequence $\left( \xi _{s}\right) _{s\in 
\mathbb{N}
}$ in $%
\mathbb{R}
$, with $\xi _{0}=x$, named the "complete quotients", and Lagrange proved
that the following assertions are equivalent: 

$\left( \mathbf{1}\right) $ $x$
is a quadratic algebraic number.

$\left( \mathbf{2}\right) $ There exist natural numbers $s\geq 0$ and $p\geq
1$ such that $\xi _{s+p}=\xi _{s}$.

$\left( \mathbf{3}\right) $ There exist natural numbers $s_{0}\geq 0$ and $%
p\geq 1$ such that for every $s\geq s_{0}$, $\xi _{s+p}=\xi _{s}$ holds
(periodicity).

The property $\left( \mathbf{2}\right) $ will be called \emph{loop property }%
in this paper.

Here we define a very natural two-dimensional continued-fraction algorithm
for which the analogue in dimension two of the properties $\left( \mathbf{1}%
\right) $ and $\left( \mathbf{2}\right) $ are equivalent: This algorithm,
named \emph{Smallest Vector Algorithm }or "\emph{SVA}",\ \textbf{makes a loop%
} (property $\left( \mathbf{2}\right) $) if and only if the real numbers
which are its two initial values are in the same cubic field (property $%
\left( \mathbf{1}\right) $). The \emph{SVA }is defined at the beginning of
Subsection 1.3..

We have to notice that we do not have periodicity, i.e. the property $\left( 
\mathbf{3}\right) $, for initial values in the same cubic fields. The reason
why is that our algorithm, unlike a lot of known multidimensional continued
fraction algorithms, is not of the \emph{vectorial }kind. Therefore, the
loop property $\left( \mathbf{2}\right) $ does not imply periodicity $\left( 
\mathbf{3}\right) $. Nevertheless, the loop property $\left( \mathbf{2}%
\right) $ implies interesting algebraic properties and the fact that the
algorithm is not vectorial permits strong approximation properties.

Let's state our Lagrange-type theorem. From any initial value

$\mathbf{X}_{0}=\mathbf{X=~}^{\textsc{T}}\left( x_{0},x_{1},x_{2}\right)$,
with $0<x_{0}<x_{1}<x_{2}$, the \emph{Smallest Vector Algorithm }generates a
sequence $\left( \mathbf{X}_{s}\right) =\left( ^{\textsc{T}}\left(
x_{0,s}~,x_{1,s}~,x_{2,s}\right)\right) $ of triplets of real numbers, and we have
the following statement.

\begin{theorem}[Lagrange Loop Theorem]
\textbf{First Part: }Let $\rho $ be any real root of a third degree
irreducible polynomial $P\left( r\right) =r^{3}-ar^{2}-br-c$, with $a,b,c$ 
\emph{rationals}; let $\mathbf{X=~}^{\textsc{T}}\left( x_{0},x_{1},x_{2}\right) $ be
any rationally independent triplet of real numbers in the field $\mathbb{Q}%
\left[ \rho \right] $, with $0<x_{0}<x_{1}<x_{2}$. Then the \emph{Smallest
Vector Algorithm} applied on the triplet $\mathbf{X}$ "makes a
loop\textquotedblright: there exist integers $s$\ and $p$ with $p>0$ \ and
a real number $\lambda $ such that: 
\begin{equation*}
\mathbf{X}_{s+p}=\lambda \mathbf{X}_{s}\text{ or equivalently: }\mathbf{x}%
_{s+p}=\mathbf{x}_{s}\text{, with }\mathbf{x}_{s}=\left( \frac{x_{0,s}}{%
x_{2,s}},\frac{x_{1,s}}{x_{2,s}}\right) \text{.}
\end{equation*}%
Moreover, $\lambda $ is an algebraic \emph{integer }of degree 3, and a \emph{%
unit, }such that $\mathbb{Q}\left[ \rho \right] =\mathbb{Q}\left[ \lambda %
\right] $. The minimal polynomial of $\lambda $ can be easily deduced from
the relation $\mathbf{X}_{s+p}=\lambda \mathbf{X}_{s}$, as also the
expressions of $\ \frac{x_{0}}{x_{2}}$ and $\frac{x_{1}}{x_{2}}$ as rational
fractions of $\lambda $.\newline
\textbf{Second Part: Converse Statement:} Let $\mathbf{X=~}^{\textsc{T}}\left(
x_{0},x_{1},x_{2}\right) $ be any rationally independent triplet of real
numbers, with $0<x_{0}<x_{1}<x_{2}$. Let's suppose that the \emph{Smallest
Vector Algorithm} applied on the triplet $\mathbf{X}$ makes
\textquotedblright a loop\textquotedblright\ i.e. that $\mathbf{X}%
_{s+p}=\lambda \mathbf{X}_{s}$ with $p>0$. Then $\lambda $ is an algebraic 
\emph{integer }of degree 3, and a \emph{unit.} Again, the minimal polynomial
of $\lambda $ can be easily deduced from the relation $\mathbf{X}%
_{s+p}=\lambda \mathbf{X}_{s}$, as also the expressions of $\frac{x_{0}}{%
x_{2}}$ and $\frac{x_{1}}{x_{2}}$ as rational fractions of $\lambda $.
\end{theorem}

The objects in this theorem are more precisely described in the following
subsections. We also prove that the same algorithm provides rational
approximations with Dirichlet properties, id est, with an
optimal exponent.

Throughout this paper, we are going to use only the canonical euclidean norm and inner product 
in $%
\mathbb{R}
^{3}$\ for our approximations. 

Our \emph{Dirichlet} \ property is that for \emph{every }independent triplet
of positive real numbers $\mathbf{X=~}^{\textsc{T}}\left( x_{0},x_{1},x_{2}\right) $,
the algorithm generates a sequence $\left( \mathbf{g}_{0}^{(s)},\mathbf{g}%
_{1}^{(s)},\mathbf{g}_{2}^{(s)}\right) $ of triplets of three-dimensional 
\emph{integer} vectors, which realize integer approximation of the plane $%
\left( \mathbf{X}^{\perp }\right) $ with the following inequality, on an
infinite set $S$ of integers:%
\begin{equation*}
\underset{s\in S}{\sup }\left[ \left( \underset{i=0,1,2}{\min }\left\vert 
\mathbf{g}_{i}^{(s)}\bullet \mathbf{X}\right\vert \right) \left( \underset{%
i=0,1,2}{\max }\left\Vert \mathbf{g}_{i}^{(s)}\right\Vert \right) ^{2}\right]
<+\infty
\end{equation*}%
(the index $^{\left( s\right) }$ is above, in parentheses; the big point
denotes the scalar product), with of course:$\underset{s\rightarrow +\infty ,%
\text{ }s\in S}{\lim }\left( \underset{i=0,1,2}{\max }\left\Vert \mathbf{g}%
_{i}^{(s)}\right\Vert \right) =+\infty $. See Theorem 2 in subsection 1.3..
We prove additional Dirichlet properties, for the integer approximation of $%
\mathbb{R}\mathbf{X}$ as well as of $\mathbf{X}^{\perp }$, when $\mathbf{X}%
^{\perp }$ (or $\mathbf{X}$) has a bad approximation property. See again
subsection 1.3..

Let's notice that the approximation properties of our algorithm hold only
for a \emph{subsequence} of the integer vectors $\left( \mathbf{g}_{0}^{(s)},%
\mathbf{g}_{1}^{(s)},\mathbf{g}_{2}^{(s)}\right) _{s\in 
\mathbb{N}
};$ the algorithm has a very simple geometrical definition, and strong
geometrical, algebraic and approximation properties, but it is not designed
to provide \emph{only} best approximants, or \emph{only }approximants with
optimal exponent.

The goal of this paper is also to show the relations between different kinds
of properties of such an algorithm:

$\left( \text{a}\right) $ \emph{Lagrange property}; $\left( \text{b}\right) $
\emph{Dirichlet approximation properties};

$\left( \text{c}\right) $ \emph{Best approximation properties}.

$\left( \text{d}\right) $ \emph{Properties of the triplet }$X$ which is the
initial value of the algorithm (it may be badly approximable by integers, or
well approximable)

$\left( \text{e}\right) $ \emph{Geometrical properties of the tetrahedrons}
formed by the three integer vectors, generated at each step by the algorithm.

This study is a generalization of the well known continued fractions theory
in dimension 1, with a two-dimensional algorithm which has more properties
than most of the existing ones. Let's notice that all the mathematical
techniques used in this paper are elementary. The most sophisticated tool
appearing here is the Minkowski's Theorem on \emph{Successive Minima} of
symmetrical convex sets.

At the end of the paper, in Section 6., the reader shall find  a short review of the themes on Diophantine approximation which are related to this
paper, with some bibliographical references.

In subsection 1.3. are given the \emph{main theorems and definitions }of
this paper. In subsection 1.4. the reader shall find two numerical examples
of Lagrange loops. The \emph{plan }of our paper is in subsection 1.5.

But first, in order to understand better the two-dimensional case, we recall
some facts and notations about the classical continued fractions algorithm
in dimension one, from a particular point of view.

\subsection{The one dimensional example}

\subsubsection{A formalism with matrices}

A real number $x$ is chosen, with $0<x<1$. Here it is supposed to be
irrational. Let the vector $\mathbf{X}$ \ be: $\mathbf{X}=\ ^{\text{T}%
}\left( x,1\right) $, where the "$^\mathsf{T}$" denotes the transposition. The
classic one-dimensional algorithm provides integer points $\ ^{\textsc{T}}%
\left( p_{n},q_{n}\right) $ which are the nearest integer points to the
line $\mathbb{D}=\mathbb{R}\mathbf{X}$. These points are called "convergent"
points. We consider the matrices $\mathbf{\ B}_{n}=%
\begin{pmatrix}
p_{n-1} & p_{n} \\ 
q_{n-1} & q_{n}%
\end{pmatrix}%
$. We have the relation: $\mathbf{B}_{n+1}=\mathbf{B}_{n}\times 
\begin{pmatrix}
0 & 1 \\ 
1 & a_{n}%
\end{pmatrix}%
$, where $a_{n}$ is the $n$-th "partial quotient" of the continued fraction
and is a strictly positive integer. If we denote: $\mathbf{A}_{n}=%
\begin{pmatrix}
0 & 1 \\ 
1 & a_{n}%
\end{pmatrix}%
$, then the approximating matrices $\mathbf{B}_{n}$ appear as products of
matrices $\mathbf{A}_{k}$ ($1\leq k\leq n-1$). Let's notice that $\mathbf{B}%
_{1}$ is the Identity matrix.

In order to be closer to our two-dimensional algorithm, we may also split
the $n$-th step into more elementary steps, and consider the simple matrix $%
\ \mathbf{D}=%
\begin{pmatrix}
1 & 0 \\ 
1 & 1%
\end{pmatrix}%
$, then we have $\mathbf{B}_{n+1}=\mathbf{B}_{n}\mathbf{A}_{n}=\mathbf{B}_{n}%
\mathbf{D}^{a_{n}}%
\begin{pmatrix}
0 & 1 \\ 
1 & 0%
\end{pmatrix}%
$. The last matrix corresponds to an exchange of vectors, when the following
convergent$~^{\textsc{T}}\left( p_{n+1},q_{n+1}\right) $ is found. We may
notice that all the matrices involved have determinant $\pm 1$.

\subsubsection{Polar matrices, cofactors, periodicity}

We also introduce the polar matrices $\mathbf{G}_{n}$, each of them being
the transposed matrix of the inverse of $\mathbf{B}_{n}$. Let $\mathbf{g}%
_{0,n}$ and $\mathbf{g}_{1,n}$ be the column vectors of $\mathbf{G}_{n}$.
These vectors realize integer approximations of the line $\mathbf{\Delta }$
orthogonal to $\mathbb{D}$, and the regular continued fraction algorithm is
precisely designed to obtain both: $\mathbf{g}_{0,n}\bullet \mathbf{X}>0$
and $\mathbf{g}_{1,n}\bullet \mathbf{X}>0$ for the scalar products.

These quantities $\mathbf{g}_{0,n}\bullet \mathbf{X}$ and $\mathbf{g}%
_{1,n}\bullet \mathbf{X}$ are particularly important in the theory of
continued fractions. Let $\mathbf{b}_{0,n}$ and $\mathbf{b}_{1,n}$ be the
column vectors of $\mathbf{B}_{n}$. We have the vectorial relation: $\left( 
\mathbf{g}_{0,n}\bullet \mathbf{X}\right) \mathbf{b}_{0,n}+$ $\left( \mathbf{%
g}_{1,n}\bullet \mathbf{X}\right) \mathbf{b}_{1,n}=\mathbf{X}$. (To see
that, make the scalar product of the left-hand vector of the equality with $%
\mathbf{g}_{0,n}$, and then with $\mathbf{g}_{1,n}$). Because of this
relation, $\left( \mathbf{g}_{0,n}\bullet \mathbf{X}\right) $ and $\left( 
\mathbf{g}_{1,n}\bullet \mathbf{X}\right) $ are called the \emph{cofactors}
in the algorithm.

Let's form the \emph{cofactors vector}$:$ $\mathbf{X}_{n}=%
\begin{pmatrix}
\mathbf{g}_{0,n}\bullet \mathbf{X} \\ 
\mathbf{g}_{1,n}\bullet \mathbf{X}%
\end{pmatrix}%
.$ We have:

$~^{\textsc{T}}\mathbf{G}_{n}\mathbf{X}=\mathbf{X}_{n}$, id est $\mathbf{B}%
_{n}^{-1}\mathbf{X}=\mathbf{X}_{n}$. Now we may calculate $\mathbf{X}_{n}$.

We have $\mathbf{B}_{n}^{-1}=\left( -1\right) ^{\left( n-1\right) }%
\begin{pmatrix}
q_{n} & -p_{n} \\ 
-q_{n-1} & p_{n-1}%
\end{pmatrix}%
$ and then

$\mathbf{X}_{n}=~^{\textsc{T}}\left( \left( -1\right) ^{n}\left(
p_{n}-xq_{n}\right) ,~\left( -1\right) ^{\left( n-1\right) }\left(
p_{n-1}-xq_{n-1}\right) \right) $.

Then we define the sequence $\left( x_{n}\right) $ by $x_{n}=\left(
-1\right) ^{n}\left( p_{n}-xq_{n}\right) =\mathbf{g}_{0,n}\bullet \mathbf{X}$%
. At the first step, $x_{1}=x$.

We may now give the rule which defines the quantity $a_{n}$, in dimension 1.
Here the brackets denote the integer part: $a_{n}=\left[ \dfrac{\mathbf{g}%
_{1,n}\bullet \mathbf{X}}{\mathbf{g}_{0,n}\bullet \mathbf{X}}\right] =\left[ 
\dfrac{x_{n-1}}{x_{n}}\right].$

We now define another object, the inverse of this quotient: $\xi _{n}:=\frac{%
x_{n}}{x_{n-1}}$. Then the preceding relation writes: $a_{n}=\left[ \frac{1}{%
\xi _{n}}\right] $.

The quantities $x_{n}$ and the cofactors vectors $\mathbf{X}_{n}$ are of
highest interest in questions concerning Lagrange property. This algorithm
is $\emph{eventually}$ \emph{periodic}, from the range $n$, if and only if
there exist an integer $p\geq 1$ and a real number $\lambda $ such that $%
\mathbf{X}_{n+p}=\lambda \mathbf{X}_{n}$, or equivalently: $\xi _{n+p}=\xi
_{n}$. These are the conditions we shall use in our Lagrange Theorem.

\subsubsection{Recursive relations on the polar matrices}

Let's denote by $\mathbf{M}^{\ast }$ the polar matrix of $\mathbf{M}$, such
that $\mathbf{M}^{\ast }\mathbf{=}\left( ^{\textsc{T}}\mathbf{M}\right) ^{-1}$%
. We have $\mathbf{G}_{n}=\mathbf{B}_{n}^{\ast }$, and then, each matrix $%
\mathbf{B}_{n}$ is a product of matrices $\mathbf{A}_{k}^{\ast }$ ($1\leq
k\leq n-1$), with $\mathbf{A}_{k}^{\ast }=%
\begin{pmatrix}
-a_{k} & 1 \\ 
1 & 0%
\end{pmatrix}%
$. We may consider the simpler matrix $\ \mathbf{C}=%
\begin{pmatrix}
1 & 0 \\ 
-1 & 1%
\end{pmatrix}%
$, then we have $\mathbf{G}_{n+1}=\mathbf{G}_{n}\mathbf{A}_{n}^{\ast }=%
\mathbf{G}_{n}\mathbf{C}^{a_{n}}%
\begin{pmatrix}
0 & 1 \\ 
1 & 0%
\end{pmatrix}%
$.

This leads to the relations $^{\textsc{T}}\mathbf{G}_{n+1}=\mathbf{A}%
_{n}^{-1}~^{\textsc{T}}\mathbf{G}_{n}$ \ and then, by $~^{\textsc{T}}\mathbf{G}%
_{n}\mathbf{X}=\mathbf{X}_{n}$, to: $\mathbf{X}_{n+1}=\mathbf{A}_{n}^{-1}\mathbf{X}_{n}=%
\begin{pmatrix}
-a_{n} & 1 \\ 
1 & 0%
\end{pmatrix}%
\mathbf{X}_{n}$; therefore: $x_{n+1}=x_{n-1}-a_{n}x_{n}$. Because of this
formula, the continued fractions algorithms may be called \emph{subtractive}%
. This leads to the recursive relation: $\xi _{n+1}=\frac{1}{\xi _{n}}-a_{n}=%
\frac{1}{\xi _{n}}-\left[ \frac{1}{\xi _{n}}\right] $.

In particular, if $\mathbf{g}_{0}$ and $\mathbf{g}_{1}$ are the column
vectors of $\mathbf{G},$ then the column vectors of the following matrix $%
\mathbf{G}\times \mathbf{C}=\mathbf{G}\times 
\begin{pmatrix}
1 & 0 \\ 
-1 & 1%
\end{pmatrix}%
$ are $\left( \mathbf{g}_{0}-\mathbf{g}_{1}\right) $ and $\mathbf{g}_{1}.$
Our two-dimensional algorithm is built in a similar way, in the next
subsection.

\subsubsection{Use of the orthogonal projections}

From a more geometrical point of view, let's denote by $\mathbf{g}%
_{0,n}^{\prime \prime }$ and $\mathbf{g}_{1,n}^{\prime \prime }$ the
orthogonal projections of $\mathbf{g}_{0,n}$ and $\mathbf{g}_{1,n}$ on $%
\mathbb{D}$. Then we also have: $\mathbf{X}_{n}=\left\Vert \mathbf{X}%
\right\Vert 
\begin{pmatrix}
\left\Vert \mathbf{g}_{0,n}^{\prime \prime }\right\Vert \\ 
\left\Vert \mathbf{g}_{1,n}^{\prime \prime }\right\Vert%
\end{pmatrix}%
$.

Concerning a \emph{Dirichlet property} of the algorithm, it can be written:

For each $n$, $\ \left\Vert \mathbf{g}_{1,n}^{\prime \prime }\right\Vert
~q_{n}\leq 1$. Our matricial formalism would permit to prove easily the
Lagrange Theorem for a quadratic number $x$, using this Dirichlet property.

Now we are going to generalize all these properties and demonstrations in
dimension two.

\subsection{Results in this paper}

Now we are in dimension two. Throughout this paper we suppose that the
triplet $\mathbf{X}=\ ^{\textsc{T}}\left( x_{0},x_{1},x_{2}\right) $ of real
numbers is rationally independent, and verifies $0<x_{0}<x_{1}<x_{2}$. We
use the canonical euclidean norm in $%
%
\mathbb{R}
^{3}$ and the canonical scalar product. We denote by $\mathbb{D}$ the line $%
\mathbb{D}=\mathbb{R}\mathbf{x}$ \ and by $\mathbb{P}$ the plane orthogonal
to $\mathbb{D}$.

As it is usually done in this field, the index $^{\left( s\right) }$ of the
sequences will be above, in parentheses.

\begin{definition}
The \emph{Smallest Vector Algorithm }is described by the following sequence $%
\left( \mathbf{G}^{(s)}\right) _{s\in \mathbb{N}}$ of $3\times 3$ integer
matrices, which is inductively defined by: 

\textbf{a) } $\mathbf{G}^{(0)}=\mathbf{I}$, the Identity matrix.

\textbf{b) } Let's suppose $\mathbf{G}^{(s)}=\left( \mathbf{g}_{0}^{(s)},%
\mathbf{g}_{1}^{(s)},\mathbf{g}_{2}^{(s)}\right) $ has been defined. Let $%
\mathbf{g}_{i}^{\prime (s)}$ and $\mathbf{g}^{\prime \prime }{}_{i}^{(s)}$%
denote the respective orthogonal projections of $\mathbf{g}_{i}^{(s)}$ on $%
\mathbb{D}$ and $\mathbb{P}$. Let $\Delta _{\min }$ denote $\Delta _{\min
}:=\min \left( \left\Vert \mathbf{g}_{1}^{\prime (s)}-\mathbf{g}_{0}^{\prime
(s)}\right\Vert ,\left\Vert \mathbf{g}_{2}^{\prime (s)}-\mathbf{g}%
_{1}^{\prime (s)}\right\Vert ,\left\Vert \mathbf{g}_{2}^{\prime (s)}-\mathbf{%
g}_{0}^{\prime (s)}\right\Vert \right) $. We define first the three column
vectors $\left( \mathbf{f}_{0},\mathbf{f}_{1},\mathbf{f}_{2}\right) $ of $%
\mathbf{G}^{(s+1)}$, in disorder:

$\left( \mathbf{f}_{0},\mathbf{f}_{1},\mathbf{f}_{2}\right) =\left( \mathbf{g%
}_{0}^{(s)},\mathbf{g}_{1}^{(s)}-\mathbf{g}_{0}^{(s)},\mathbf{g}%
_{2}^{(s)}\right) $ if $\Delta _{\min }=\left\| \mathbf{g}_{1}^{\prime (s)}-%
\mathbf{g}_{0}^{\prime (s)}\right\| $;

$\left( \mathbf{f}_{0},\mathbf{f}_{1},\mathbf{f}_{2}\right) =\left( \mathbf{g%
}_{0}^{(s)},\mathbf{g}_{1}^{(s)},\mathbf{g}_{2}^{(s)}-\mathbf{g}%
_{1}^{(s)}\right) $ if $\Delta _{\min }=\left\| \mathbf{g}_{2}^{\prime (s)}-%
\mathbf{g}_{1}^{\prime (s)}\right\| $;

$\left( \mathbf{f}_{0},\mathbf{f}_{1},\mathbf{f}_{2}\right) =\left( \mathbf{g%
}_{0}^{(s)},\mathbf{g}_{1}^{(s)},\mathbf{g}_{2}^{(s)}-\mathbf{g}%
_{0}^{(s)}\right) $ if $\Delta _{\min }=\left\| \mathbf{g}_{2}^{\prime (s)}-%
\mathbf{g}_{0}^{\prime (s)}\right\| $.

$\mathbf{G}^{(s+1)}=\left( \mathbf{g}_{0}^{(s+1)},\mathbf{g}_{1}^{(s+1)},%
\mathbf{g}_{2}^{(s+1)}\right) $ is defined as any of the permutations of the
vectors $\left( \mathbf{f}_{0},\mathbf{f}_{1},\mathbf{f}_{2}\right) $ such
that we have: $\left\Vert \mathbf{g}^{\prime \prime
}{}_{0}^{(s+1)}\right\Vert \leq \left\Vert \mathbf{g}^{\prime \prime
}{}_{1}^{(s+1)}\right\Vert \leq \left\Vert \mathbf{g}^{\prime \prime
}{}_{2}^{(s+1)}\right\Vert $.
\end{definition}

The columns $\mathbf{g}_{0}^{(s)},\mathbf{g}_{1}^{(s)},\mathbf{g}_{2}^{(s)}$
of the matrices $\mathbf{G}^{(s)}$ realize integer approximations of the
plane $\mathbb{P}$. They play the same role in dimension 2 as, in dimension
1, the matrices $\mathbf{G}_{n}$ studied above.

As in dimension 1, the cofactors vector $\mathbf{X}^{\left( s\right) }$ is
fundamental. It's defined by: $\mathbf{X}^{\left( s\right) }=\left\Vert \mathbf{X}\right\Vert \cdot \ ^{%
\text{T}}\left( \left\Vert \mathbf{g}_{0}^{\prime \prime (s)}\right\Vert
,\left\Vert \mathbf{g}_{1}^{\prime \prime (s)}\right\Vert ,\left\Vert 
\mathbf{g}_{2}^{\prime \prime (s)}\right\Vert \right) ,$ and the vector $%
\mathbf{x}^{\left( s\right) }$, the "projective" version of $\ \mathbf{X}%
^{\left( s\right) }$, is defined by $\mathbf{x}^{\left( s\right) }=
\ ^{\textsc{T}}\left( \frac{\left\Vert \mathbf{g}_{0}^{\prime \prime (s)}\right\Vert }{%
\left\Vert \mathbf{g}_{2}^{\prime \prime (s)}\right\Vert },\frac{\left\Vert 
\mathbf{g}_{1}^{\prime \prime (s)}\right\Vert }{\left\Vert \mathbf{g}%
_{2}^{\prime \prime (s)}\right\Vert },1\right) $.

The \emph{Smallest Vector Algorithm} has a \emph{Lagrange} property, which
is expressed by Theorem 1 of subsection 1.1.
The demonstration of this theorem (see below) is essentially the same as in
dimension one, based on a \emph{Dirichlet }result (In dimension two, Theorem 2).

First we introduce the unimodular positive integer matrix $\mathbf{B}^{(s)}$%
, defined as the polar matrix of $\mathbf{G}^{(s)}$ (the transposed matrix
of the inverse of $\mathbf{G}^{(s)}$). We denote by $\mathbf{b}_{0}^{(s)},%
\mathbf{b}_{1}^{(s)},\mathbf{b}_{2}^{(s)}$ the column vectors of $\mathbf{B}%
^{(s)}$ which realize integer approximations of the line $\mathbb{D}$. The
vectors $\mathbf{b}_{0}^{\prime \prime }{}^{(s)}$, $\mathbf{b}_{1}^{\prime
\prime }{}^{(s)}$, $\mathbf{b}_{1}^{\prime \prime }{}^{(s)}$ will be the
orthogonal projections of these vectors on $\mathbb{D}$, and the $\ \mathbf{b%
}_{i}^{\prime (s)}$ their orthogonal projections on $\mathbb{P}$. \ In the
same way are defined the $\mathbf{g}_{1}^{\prime \prime }{}^{(s)}$ and the $%
\mathbf{g}_{i}^{\prime (s)},$ and, for any vector $\mathbf{h}$, the
projections $\mathbf{h}^{\prime }$ and $\mathbf{h}^{\prime \prime }$ of $%
\mathbf{h}$ on $\mathbb{P}$ and $\mathbb{D}$,

\begin{theorem}[Dirichlet Properties]
For each rationally independent triplet $\mathbf{X}$\ of real numbers, with $%
0<x_{0}<x_{1}<x_{2}$, the \emph{Smallest Vector Algorithm }has the following
properties

\textbf{a) }There exists an infinite set $S$ of natural integers such that:%
\textbf{\ }%
\begin{equation*}
\underset{s\in S}{\sup }\left[ \left( \underset{i=0,1,2}{\max }\left\Vert 
\mathbf{g}_{i}^{\prime (s)}\right\Vert \right) ^{2}\left( \underset{i=0,1,2}{%
\min }\left\Vert \mathbf{g}_{i}^{\prime \prime }{}^{(s)}\right\Vert \right) %
\right] <+\infty \text{, with also:}
\end{equation*}%
$\underset{s\rightarrow +\infty ,\text{ }s\in S}{\lim }\left( \underset{%
i=0,1,2}{\min }\left\Vert \mathbf{g}_{i}^{\prime \prime }{}^{(s)}\right\Vert
\right) =0.$ In other words, this algorithm provides a Diophantine
approximation of $\ \mathbb{P}$ which possesses a Dirichlet property, with
the optimal exponent, $2$, and with a form which is rather strong.

\textbf{b)} IF there exists $c>0$ such that for any integer point $\mathbf{%
h\neq 0},$\newline
$\left\Vert \mathbf{h}\right\Vert ^{2}\left\Vert \mathbf{h}^{\prime \prime
}\right\Vert >c$ holds (id est, if the couple $\left( \dfrac{x_{1}}{x_{0}},%
\dfrac{x_{2}}{x_{0}}\right) $ is badly approximable, as it is proved in
Section 3), then the following relation holds with the Smallest Vector
Algorithm, with the same infinite set $S$: 
\begin{equation*}
\underset{s\in S}{\sup }\left[ \left( \underset{i=0,1,2}{\max }\left\Vert 
\mathbf{g}_{i}^{\prime (s)}\right\Vert \right) ^{2}\left( \underset{i=0,1,2}{%
\max }\left\Vert \mathbf{g}_{i}^{\prime \prime }{}^{(s)}\right\Vert \right) %
\right] <+\infty \text{, with also:}
\end{equation*}%
$\underset{s\rightarrow +\infty ,\text{ }s\in S}{\lim }\left( \underset{%
i=0,1,2}{\max }\left\Vert \mathbf{g}_{i}^{\prime \prime }{}^{(s)}\right\Vert
\right) =0$.

\textbf{c)} If again there exists $c>0$ such that for any integer point $%
\mathbf{h\neq 0},$

$\left\Vert \mathbf{h}\right\Vert ^{2}\left\Vert \mathbf{h}^{\prime \prime
}\right\Vert >c$ holds, then%
\begin{equation*}
\underset{s\in S}{\sup }\left[ \left( \underset{i=0,1,2}{\max }\left\Vert 
\mathbf{b}_{i}^{\prime (s)}\right\Vert \right) ^{2}\left( \underset{i=0,1,2}{%
\max }\left\Vert \mathbf{b}_{i}^{\prime \prime }{}^{(s)}\right\Vert \right) %
\right] <+\infty \text{; with also:}
\end{equation*}%
$\underset{s\rightarrow +\infty ,\text{ }s\in S}{\lim }\left( \underset{%
i=0,1,2}{\max }\left\Vert \mathbf{b}_{i}^{\prime (s)}\right\Vert \right) =0$.
\end{theorem}

\begin{theorem}[Geometrical Theorem]
The \emph{Smallest Vector Algorithm} possesses the following property. Let $%
\mathbf{H}^{\prime (s)}$ be the convex hull

$\mathbf{H}^{\prime (s)}=\textsc{conv}\left( \left\{ \mathbf{g}_{0}^{\prime
(s)},\mathbf{g}_{1}^{\prime (s)},\mathbf{g}_{2}^{\prime (s)},-\mathbf{g}%
_{0}^{\prime (s)},-\mathbf{g}_{1}^{\prime (s)}-,\mathbf{g}_{2}^{\prime
(s)}\right\} \right) $ in $\mathbb{P}$. Let $\rho ^{(s)}$ be the radius of
the greatest disk in $\mathbb{P}$ with center $0$ contained in $\mathbf{H}%
^{\prime (s)}$. Then there exists an infinite set $S$ of natural integers
such that:\newline
$\underset{s\in S}{\sup }\dfrac{\underset{i=0,1,2}{\max }\left\Vert \mathbf{g%
}_{i}^{\prime (s)}\right\Vert }{\rho ^{(s)}}<+\infty .$
\end{theorem}

\begin{definition}
If $\ \underset{s\rightarrow +\infty }{\lim \inf }$ $\frac{\underset{i=0,1,2}%
{\max }\left\Vert \mathbf{g}_{i}^{\prime (s)}\right\Vert }{\rho ^{(s)}}%
<+\infty $, with the notations of the above Geometrical Theorem, it will be
said that the algorithm is \emph{balanced}.
\end{definition}

As a \emph{Best Approximation }result, we give our Prism Lemma, which is
proved in Subsection \textbf{3.1}.. It is a very easy result, but it is true
and important for any continued fraction algorithm, and the author has not
seen this statement anywhere in literature.

\textbf{Lemma ( Prism Lemma ).}\textit{\ Let }$\mathbf{g}_{0}^{(s)}$, $%
\mathbf{g}_{1}^{(s)}$, $\mathbf{g}_{2}^{(s)}$ \textit{be the column vectors
of the matrix }$\mathbf{G}^{(s)}$\textit{\ generated at the }$s$\textit{-th
step by the Smallest Vector Algorithm. Let the sets }$\mathbf{H}^{\prime (s)}
$\textit{\ be the convex hulls: }\newline
$\mathbf{H}^{\prime (s)}=\textsc{conv}\left( \mathbf{g}_{0}^{\prime (s)},%
\mathbf{g}_{1}^{\prime (s)},\mathbf{g}_{2}^{\prime (s)},-\mathbf{g}%
_{0}^{\prime (s)},-\mathbf{g}_{1}^{\prime (s)},-\mathbf{g}_{2}^{\prime
(s)}\right) $\textit{\ in }$\mathbb{P}$\textit{\ .}

\textit{We shall omit the indices }$^{(s)}$\textit{. Let }$\mathbf{H}$%
\textit{\ be the prism: }$\mathbf{H}=\mathbf{H}^{\prime }+\mathbb{D}$\textit{%
. Then, with the usual notation }$\mathbf{h}^{\prime \prime }$ \textit{for
the orthogonal projection of }$\mathbf{h}$\textit{\ on }$\mathbb{D}$\textit{:%
}

\textit{-For each non zero integer point }$h$\textit{\ in }$\mathbf{H}$%
\textit{\ , }$\left\Vert \mathbf{h}^{\prime \prime }\right\Vert \geq
\left\Vert \mathbf{g}_{0}^{\prime \prime }\right\Vert $\textit{\ holds.}

\textit{-For each integer point }$\mathbf{h}$\textit{\ in }$\mathbf{H}$%
\textit{\ which is not of the form }$n_{0}\mathbf{g}_{0}$\textit{, with }$%
n_{0}$\textit{\ integer, }$\left\Vert \mathbf{h}^{\prime \prime }\right\Vert
\geq \left\Vert \mathbf{g}_{1}^{\prime \prime }\right\Vert $\textit{\ holds.}

\textit{-For each integer point }$\mathbf{h}$\textit{\ in }$\mathbf{H}$%
\textit{\ \ which is not of the form }$n_{0}\mathbf{g}_{0}+n_{1}\mathbf{g}%
_{1}$\textit{, with }$n_{0}$\textit{\ and }$n_{1}$\textit{\ integers, }$%
\left\Vert \mathbf{h}^{\prime \prime }\right\Vert \geq \left\Vert \mathbf{g}%
_{2}^{\prime \prime }\right\Vert $\textit{\ holds.}

\textit{- That implies that, if }$\left( \mathbf{h}_{0},\mathbf{h}_{1},%
\mathbf{h}_{2}\right) $\textit{\ is a free triplet of integer points in }$%
\mathbf{H}$\textit{,\ then }$\underset{i=0,1,2}{\max }\left( \left\Vert 
\mathbf{h}_{i}^{\prime \prime }\right\Vert \right) \geq \underset{i=0,1,2}{%
\max }\left( \left\Vert \mathbf{g}_{i}^{\prime \prime }\right\Vert \right)
=\left\Vert \mathbf{g}_{2}^{\prime \prime }\right\Vert $\textit{.\smallskip }

The previous theorems suppose that the initial values $\left(
x_{0};x_{1};x_{2}\right) $ of our algorithm are rationally independent. The
reader may wonder what happens when they are not.

\begin{theorem}
The \emph{Smallest Vector Algorithm} finds rational dependence. If the
initial values $\left( x_{0},x_{1},x_{2}\right) $ of our algorithm are
rationally dependent, then the \emph{SVA }generates at some step an integer
triplet $\left( \mathbf{g}_{0}^{(s)},\mathbf{g}_{1}^{(s)},\mathbf{g}%
_{2}^{(s)}\right) $ such that: $\left\Vert \mathbf{g}_{0}^{\prime \prime
(s)}\right\Vert =\mathbf{g}_{0}^{(s)}\bullet \mathbf{X}=0$. This relation
gives the coefficients of the rational (integral) dependence of $\ \left(
x_{0},x_{1},x_{2}\right) $.
\end{theorem}

The proof uses Lemma 2 in the next subsection, the Geometrical Theorem and the Prism Lemma. The demonstrations of these three results are still valid if $\ \left(
x_{0},x_{1},x_{2}\right) $  is rationally dependent. With these three results, the proof of Theorem 4 is very short.

\begin{proof}
Let's suppose that we have an integral dependence relation, of the shape $%
\mathbf{h\bullet X}=0$, $\mathbf{h}$ being a non null integer vector. By the
Lemma 2 of the next subsection: $\underset{s\rightarrow +\infty }{\lim }%
\left( \underset{i=0,1,2}{\max }\left\Vert \mathbf{g}_{i}^{\prime
(s)}\right\Vert \right) =+\infty $. In addition, by the Geometrical Theorem
above, $\underset{s\rightarrow +\infty }{\lim \inf }$ $\frac{\underset{%
i=0,1,2}{\max }\left\Vert \mathbf{g}_{i}^{\prime (s)}\right\Vert }{\rho
^{(s)}}<+\infty $. Then $\underset{%
s\rightarrow +\infty }{\lim \sup }\, \rho ^{(s)}=+\infty $. Then $\mathbf{h}$
belongs to some hexagon $\mathbf{H}^{\prime (s)}$ defined as a convex hull
in the Prism Lemma just above. By this Prism Lemma, $0=\left\Vert \mathbf{h}%
^{\prime \prime }\right\Vert \geq \left\Vert \mathbf{g}_{0}^{\prime \prime
}\right\Vert =\dfrac{\mathbf{X\bullet g}_{0}^{(s)}}{\left\Vert \mathbf{X}%
\right\Vert }\geq 0$ holds. Then $\mathbf{X\bullet g}_{0}^{(s)}=0$.
\end{proof}

\subsection{Numerical examples} 
We give two examples of Lagrange Loops when the initial values are in
some cubic field.

\begin{example}
$\left( x_{0};x_{1};x_{2}\right) =\left( 1;2\cos \left( \pi /7\right) ;4\cos
^{2}\left( \pi /7\right) \right) $. The sequence of the\newline
$\left( \frac{x_{0}^{\left( s\right) }}{x_{2}^{\left( s\right) }};\frac{%
x_{1}^{\left( s\right) }}{x_{2}^{\left( s\right) }}\right) $ is the simplest
the author has met. Almost every couple in it repeats infinitely. We have: $%
x_{1}^{3}-x_{1}^{2}-2x_{1}+1=0$, and\newline
$\left( x_{2}-2\right) ^{3}+\left( x_{2}-2\right) ^{2}-2\left(
x_{2}-2\right) -1=0$; the algorithm provides the following $\left(
s+1,\left( \frac{x_{0}^{\left( s\right) }}{x_{2}^{\left( s\right) }};\frac{%
x_{1}^{\left( s\right) }}{x_{2}^{\left( s\right) }}\right) \right) $.
\end{example}

$\{1,\{1.80193773580484,\qquad 3.24697960371747\}\ \}$

$\{2,\{1.80193773580484,\qquad 2.24697960371747\}\ \}$

$\{3,\{1.24697960371747,\qquad 2.80193773580484\}\ \}$

$\{4,\{1.24697960371747,\qquad 1.80193773580484\}\ \}$

$\{5,\{1.80193773580484,\qquad 2.24697960371747\}\ \}$

$.../...$

$\{292,\{1.24697960371747,\qquad 1.80193773580484\}\}$

$\{293,\{1.80193773580484,\qquad 2.24697960371747\}\}$

$\{294,\{1.24697960371747,\qquad 1.80193773580484\}\}$

$\{295,\{1.24697960371747,\qquad 1.55495813208737\}\}$

$\{296,\{1.80193773580484,\qquad 2.24697960371747\}\}$

$\{297,\{1.24697960371747,\qquad 1.80193773580484\}\}$

$\{298,\{1.80193773580484,\qquad 2.24697960371747\}\}$

$\{299,\{1.24697960371747,\qquad 2.80193773580484\}\}$

$\{300,\{1.24697960371747,\qquad 1.80193773580484\}\}$

\begin{example}
$\left( x_{0};x_{1};x_{2}\right) =\left( 1;\sqrt[3]{13};\sqrt[3]{13^{2}}%
\right) $. The algorithm provides the following sequence $\left( s+1,\left( 
\frac{x_{0}^{\left( s\right) }}{x_{2}^{\left( s\right) }};\frac{%
x_{1}^{\left( s\right) }}{x_{2}^{\left( s\right) }}\right) \right) $. Here
there are very few repetitions (loops), but they exist.
\end{example}

$\{1,\{2.35133468772075748950001,\qquad 5.5287748136788721414723\}\}$

$\{2,\{2.35133468772075748950001,\qquad 4.5287748136788721414723\}\}$

$\{3,\{1.35133468772075748950001,\qquad 4.5287748136788721414723\}\}$

$.../...$

$\{104,\{2.35133468772075748950001,\qquad 5.528774813678872141472\}\}$

$\{105,\{2.35133468772075748950001,\qquad 4.528774813678872141472\}\}$

$\{106,\{1.35133468772075748950001,\qquad 4.528774813678872141472\}\}$

$../...$

$\{160,\{5.5057084068852398563646,\qquad 47.82563987973201144317\}\}$

$.../...$

$\{219,\{1.35133468772075748950001,\qquad 4.5287748136788721414723\}\}$

$.../...$

$\{308,\{2.35133468772075748950001,\qquad 5.5287748136788721414723\}\}$

$.../...$

$\{411,\{2.35133468772075748950001,\qquad 5.5287748136788721414723\}\}$

\subsection{PLAN of the paper}

We shall prove the Results above in the order in which they are written in
this first section, and, in fact, in the reverse of the logical order.

\begin{itemize}
\item In the\textbf{\ second section}, we admit that the Smallest Vector
Algorithm or \emph{SVA} has the Dirichlet Properties of Theorem 2, and we
prove that this implies the Lagrange properties. But this demonstration is
made \textbf{only in a particular case}, namely $\mathbf{X}=\left( 1,\sqrt[3]%
{N},\sqrt[3]{N^{2}}\right) $ with $N$ natural.\textbf{\ The complete proof}
of the Lagrange Property for three numbers in a cubic number field is given
in \textbf{Section 5}.\ But this general demonstration is intricate, and the
particular case gives all the main ideas involved. That's why we prefer to
begin with this particular case, for more clarity.

\item In the\textbf{\ third section}, we admit the geometrical property of
the SVA, namely that it is balanced. We prove that this property implies the
Dirichlet Properties.

\item In the \textbf{fourth section}, we prove the Geometrical Theorem,
namely that the SVA is balanced.

\item In the\textbf{\ fifth section}, we give the general demonstration of
the Lagrange Theorem. Then all theorems are proved.

\item In \textbf{Section 6}, we locate the results of this paper in relation
to the main themes in Multidimensional Dimensional Fractions Theory and
Homogeneous Diophantine Approximation, with some bibliographical references.
\end{itemize}

\section{Demonstration of Lagrange Properties (particular case)}

\subsection{Four basic Lemmas}

We're going to need the following lemmas.

\begin{lemma}[Basic Properties of Brentjes' Algorithms]
\textbf{a) }Inductively, by the way of building the Smallest Vector
Algorithm, we have the following properties: \newline
$0\leq \mathbf{X\bullet g}_{0}^{(s)}=\left\Vert \mathbf{X}\right\Vert
\left\Vert \mathbf{g}_{0}^{\prime \prime }{}^{(s)}\right\Vert \leq \mathbf{%
X\bullet g}_{1}^{(s)}=\left\Vert \mathbf{X}\right\Vert \left\Vert \mathbf{g}%
_{1}^{\prime \prime }{}^{(s)}\right\Vert \leq \mathbf{X\bullet g}%
_{2}^{(s)}=\left\Vert \mathbf{X}\right\Vert \left\Vert \mathbf{g}%
_{2}^{\prime \prime }{}^{(s)}\right\Vert $.\newline
\textbf{b) }The following equality holds for every $s\in 
\mathbb{N}
$: 
\begin{equation*}
\left\Vert \mathbf{X}\right\Vert \left\Vert \mathbf{g}_{0}^{\prime \prime
}{}^{(s)}\right\Vert \cdot \mathbf{b}_{0}^{(s)}+\left\Vert \mathbf{X}%
\right\Vert \left\Vert \mathbf{g}_{1}^{\prime \prime }{}^{(s)}\right\Vert
\cdot \mathbf{b}_{1}^{(s)}+\left\Vert \mathbf{X}\right\Vert \left\Vert 
\mathbf{g}_{2}^{\prime \prime }{}^{(s)}\right\Vert \cdot \mathbf{b}%
_{2}^{(s)}=\mathbf{X}
\end{equation*}%
That's the reason why $\ \left\Vert \mathbf{X}\right\Vert \left\Vert \mathbf{%
g}_{0}^{\prime \prime }{}^{(s)}\right\Vert ,$ $\left\Vert \mathbf{X}%
\right\Vert \left\Vert \mathbf{g}_{1}^{\prime \prime }{}^{(s)}\right\Vert ,$ 
$\left\Vert \mathbf{X}\right\Vert \left\Vert \mathbf{g}_{1}^{\prime \prime
}{}^{(s)}\right\Vert $ are called \emph{cofactors. }\newline
\textbf{c) }For every $s\in 
\mathbb{N}
$: $~^{\textsc{T}}\mathbf{G}^{(s)}\mathbf{X}=\mathbf{X}^{\left( s\right) }$
or, which is the same:\newline
 $\left( \mathbf{B}^{(s)}\right) ^{-1}\mathbf{X}=%
\mathbf{X}^{\left( s\right) }$.
\end{lemma}

\begin{proof}
First of all, the a) Property is true at the step $s=0$. Let's suppose it's
true at the step $s$. The construction of the new $\mathbf{g}_{j}^{(s+1)}$
by subtraction is always done in accordance with the order of \ the $\mathbf{%
X\bullet g}_{i}^{(s)}=\left\Vert \mathbf{X}\right\Vert \left\Vert \mathbf{g}%
_{i}^{\prime \prime }{}^{(s)}\right\Vert $. Then, at the next step, each of
the $\mathbf{X\bullet g}_{i}^{(s+1)}$ is positive, and then equals $%
\left\Vert \mathbf{X}\right\Vert \left\Vert \mathbf{g}_{i}^{\prime \prime
}{}^{(s+1)}\right\Vert $. The rule of the algorithm is to order these
numbers at step $\left( s+1\right) $. Then the property is obtained at step $%
\left( s+1\right) $, and \ a) is true by induction. For the b) property, the equality\newline
$\left( \mathbf{X\bullet g}_{0}^{(s)}\right) \cdot \mathbf{b}%
_{0}^{(s)}+\left( \mathbf{X\bullet g}_{1}^{(s)}\right) \cdot \mathbf{b}%
_{1}^{(s)}+\left( \mathbf{X\bullet g}_{2}^{(s)}\right) \cdot \mathbf{b}%
_{2}^{(s)}=\mathbf{X}$ \ holds. To see that, make the scalar product of both
the left-hand and the right-hand vector of the equality with each of the $%
\mathbf{g}_{i}^{(s)}$.\newline 
This equality is the same as the equality in b).

Because $\mathbf{X}^{\left( s\right) }=~^{\textsc{T}}\left( \mathbf{%
X\bullet g}_{0}^{(s)},\mathbf{X\bullet g}_{1}^{(s)},\mathbf{X\bullet g}%
_{2}^{(s)}\right) $, the equalities of c) are obvious.
\end{proof}

\begin{lemma}
For the Smallest Vector Algorithm, and more generally in any Brentjes'
algorithm, we have: $\underset{s\rightarrow +\infty }{\lim }\left( \underset{%
i=0,1,2}{\max }\left\Vert \mathbf{g}_{i}^{\prime (s)}\right\Vert \right)
=+\infty $.
\end{lemma}

\begin{proof}
If this limit does not hold, there exist a real number $M$ and an infinite
set $T$ such that for any $s\in T$, and for $i=0,1,2$, \ $\left\Vert \mathbf{%
g}_{i}^{\prime (s)}\right\Vert \leq M$ \ holds. In addition, by the
subtractive nature of these algorithms, the $\left\Vert \mathbf{g}%
_{i}^{\prime \prime }{}^{(s)}\right\Vert $, $i=0,1,2$ are also bounded. Then
the set of the triplets $\left( \mathbf{g}_{0}^{(s)},\mathbf{g}_{1}^{(s)},%
\mathbf{g}_{2}^{(s)}\right) $ with $s\in T$ is bounded. Then it is finite.
Then there exist two distinct natural numbers $s$ and $t$ such that $\left( 
\mathbf{g}_{0}^{(s)},\mathbf{g}_{1}^{(s)},\mathbf{g}_{2}^{(s)}\right)
=\left( \mathbf{g}_{0}^{(t)},\mathbf{g}_{1}^{(t)},\mathbf{g}%
_{2}^{(t)}\right) $ and particularly: $\left( \mathbf{g}_{0}^{\prime \prime
}{}^{(s)},\mathbf{g}_{1}^{\prime \prime }{}^{(s)},\mathbf{g}_{2}^{\prime
\prime }{}^{(s)}\right) =\left( \mathbf{g}_{0}^{\prime \prime }{}^{(t)},%
\mathbf{g}_{1}^{\prime \prime }{}^{(t)},\mathbf{g}_{2}^{\prime \prime
}{}^{(t)}\right) $. But that's impossible, again by the subtractive nature
of these algorithms. Our hypothesis was false, and then the conclusion of
the theorem holds.
\end{proof}

Now, we state two lemmas which will provide the Second Part of our Lagrange
Theorem.

\begin{lemma}[Degree in a Loop]
Let $\mathbf{X=~}^{\textsc{T}}\left( x_{0},x_{1},x_{2}\right) $ be any
rationally independent triplet of real numbers, with $0<x_{0}<x_{1}<x_{2}$.
Let's suppose that the \emph{Smallest Vector Algorithm} applied on the
triplet $\mathbf{X}$ "makes a loop"\ i.e. that $\mathbf{X}^{\left(
s+p\right) }=\lambda \mathbf{X}^{\left( s\right) }$ with $p>0$. Then $%
\lambda $ \emph{cannot }be a quadratic real number.
\end{lemma}

\begin{proof}
If the case of such a loop, we have $\mathbf{X}^{\left( s+p\right) }=\lambda 
\mathbf{X}_{s}^{\left( s\right) }$, or equivalently

$\left( \mathbf{B}^{(s+p)}\right) ^{-1}\mathbf{X}=\lambda \left( \mathbf{B}%
^{(s)}\right) ^{-1}\mathbf{X}$ \ or $\left( \mathbf{B}^{(s+p)}\right) ^{-1}%
\mathbf{B}^{(s)}\mathbf{X}^{\left( s\right) }=\lambda \mathbf{X}^{\left(
s\right) }$.

Let's denote: $\widetilde{\mathbf{B}}:=\left( \mathbf{B}^{(s+p)}\right) ^{-1}%
\mathbf{B}^{(s)}$. Then: $\widetilde{\mathbf{B}}$\ $\mathbf{X}^{\left(
s\right) }=\lambda \mathbf{X}^{\left( s\right) }$, with

$\mathbf{X}^{\left( s\right) }=\ ^{\textsc{T}}\left( x_{0}^{\left( s\right)
},x_{1}^{\left( s\right) },x_{2}^{\left( s\right) }\right) $.

Let's denote: $\mathbf{Y:}=\ ^{\textsc{T}}\left( \frac{x_{0}^{\left( s\right) }%
}{x_{2}^{\left( s\right) }},\frac{x_{1}^{\left( s\right) }}{x_{2}^{\left(
s\right) }},1\right):=\ ^{\textsc{T}}\left( y_{0},y_{1},1\right) $.

Then also: $\widetilde{\mathbf{B}}\mathbf{Y}=\lambda \mathbf{Y}$, which can
be written

$\left( 
\begin{array}{ccc}
\beta _{00} & \beta _{01} & \beta _{02} \\ 
\beta _{10} & \beta _{11} & \beta _{12} \\ 
\beta _{20} & \beta _{21} & \beta _{22}%
\end{array}%
\right) \left( 
\begin{array}{c}
y_{0} \\ 
y_{1} \\ 
1%
\end{array}%
\right) =\left( 
\begin{array}{c}
\lambda y_{0} \\ 
\lambda y_{1} \\ 
\lambda%
\end{array}%
\right) $, where each $\beta _{ij}$ is an integer.

Then we have: $\beta _{20}y_{0}+\beta _{21}y_{1}+\beta _{22}=\lambda $. From
now on, we suppose that $\lambda $ is a quadratic real number.

\textbf{First Case: \ }$\beta _{21}=0$ \textbf{. }In this case, $y_{0}$
belongs to $%
\mathbb{Q}
\left( \lambda \right) $. In addition, the same equality of matrices
provides $\beta _{10}y_{0}+\beta _{11}y_{1}+\beta _{12}=\lambda y_{1}$,
which can be written: $\beta _{10}y_{0}+\left( \beta _{11}-\lambda \right)
y_{1}+\beta _{12}=0$. Then $y_{1}=\dfrac{\beta _{10}y_{0}+\beta _{12}}{%
\lambda -\beta _{11}}$ and $y_{1}$ also belongs to $%
\mathbb{Q}
\left( \lambda \right) $.

\textbf{Second Case: }$\beta _{21}\neq 0$. Then: $y_{1}=\gamma
_{0}y_{0}+\gamma _{1}\lambda +\gamma _{2}$, with $\gamma _{0}$, $\gamma _{1}$%
, $\gamma _{2}$ rationals. The same equality of matrices provides $\beta
_{00}y_{0}+\beta _{01}y_{1}+\beta _{02}=\lambda y_{0}$, or $\left( \beta
_{00}+\beta _{01}\gamma _{0}-\lambda \right) y_{0}+\beta _{01}\gamma
_{1}\lambda +\gamma _{2}+\beta _{02}=0$. Then $y_{0}=\dfrac{\beta
_{01}\gamma _{1}\lambda +\gamma _{2}+\beta _{02}}{\lambda -\beta _{00}-\beta
_{01}\gamma _{0}}$.\newline
Then, in both cases, both $y_{0}$ and $y_{1}$ belong to $%
\mathbb{Q}
\left( \lambda \right) $, which is a vectorial space with dimension $2$ over 
$%
\mathbb{Q}
.$ Then the three numbers $y_{0}$, $y_{1}$ and $1$ are rationally dependent.
So are therefore the three numbers $x_{0}^{\left( s\right) },x_{1}^{\left(
s\right) },x_{2}^{\left( s\right) },$ and then the three numbers $%
x_{0}=x_{0}^{\left( 0\right) }$, $x_{1}=x_{1}^{\left( 0\right) }$, $%
x_{2}=x_{2}^{\left( 0\right) },$ because $\mathbf{X}^{\left( 0\right) }=%
\mathbf{B}^{(s)}\mathbf{X}^{\left( s\right) }$. This contradicts our
hypothesis. Then $\lambda $ \emph{cannot }be a quadratic real number.
\end{proof}

\begin{lemma}
\emph{(Converse Statement in Lagrange Theorem)}.\newline
Let $\mathbf{X=~}^{\textsc{T}}%
\left( x_{0},x_{1},x_{2}\right) $ be any rationally independent triplet of
real numbers, with $0<x_{0}<x_{1}<x_{2}$. Let's suppose that the \emph{%
Smallest Vector Algorithm} applied on the triplet $\mathbf{X}$ makes a
"loop\textquotedblright\ i.e. that $\mathbf{X}^{\left( s+p\right) }=\lambda 
\mathbf{X}^{\left( s\right) }$ with $p>0$. Then $\lambda $ is an algebraic 
\emph{integer }of degree 3, and a \emph{unit}. The minimal polynomial of $%
\lambda $ can be easily deduced from the relation $\mathbf{X}^{\left(
s+p\right) }=\lambda \mathbf{X}^{\left( s\right) }$ as also the expressions
of $\ \frac{x_{0}}{x_{2}}$ and $\frac{x_{1}}{x_{2}}$ as rational fractions
of $\lambda $.
\end{lemma}

\begin{proof}
By hypothesis,we have: $\mathbf{X}^{\left( s+p\right) }=\lambda \mathbf{X}%
^{\left( s\right) }$, or equivalently

$\left( \mathbf{B}^{(s+p)}\right) ^{-1}\mathbf{X}=\lambda \left( \mathbf{B}%
^{(s)}\right) ^{-1}\mathbf{X}$ \ or $\left( \mathbf{B}^{(s+p)}\right)
^{-1}\left( \mathbf{B}^{(s)}\right) \mathbf{X}^{\left( s\right) }=\lambda 
\mathbf{X}^{\left( s\right) }$.

Let's denote: $\widetilde{\mathbf{B}}:=\left( \mathbf{B}^{(s+p)}\right)
^{-1}\left( \mathbf{B}^{(s)}\right) $. Then: $\widetilde{\mathbf{B}}$\ $%
\mathbf{X}^{\left( s\right) }=\lambda \mathbf{X}^{\left( s\right) }$. Let $%
F\left( \xi \right) $ be the polynomial: $F\left( \xi \right) =\det \left( 
\widetilde{\mathbf{B}}-\xi \mathbf{I}\right) $, we have: $F\left( \lambda
\right) =0$. But $\widetilde{\mathbf{B}}$ \ is an integer matrix with
determinant $\pm 1$; then we have a relation of the shape: $\lambda
^{3}+m\lambda ^{2}+n\lambda \pm 1=0$, with $m$ and $n$ natural integers.

Then $\lambda $ is an algebraic integer and a unit. Then, if $\lambda $ is a
rational number, it has to be: $\lambda =\pm 1$. But that's impossible
because we should have $\mathbf{X}^{\left( s+p\right) }=\pm \mathbf{X}%
^{\left( s\right) }$, which is impossible by the subtractive form of the
recursive relation on the sequence $\left( \mathbf{X}^{\left( s\right)
}\right) $. By the previous Lemma $\lambda $ cannot be a root of a second
degree polynomial. Then $F\left( \xi \right) $ is the minimal polynomial of $%
\lambda $ over $\mathbb{Q}$, and $\lambda $ is a cubic algebraic integer.
Its norm is 1, and the relation $\lambda \left( \lambda ^{2}+m\lambda
+n\right) =\mp 1$ shows that $\lambda $ is a unit.

Now, we have to solve the equation $\left( \widetilde{\mathbf{B}}-\lambda 
\mathbf{I}\right) \mathbf{X}^{\left( s\right) }=0$, or $\left( \widetilde{%
\mathbf{B}}-\lambda \mathbf{I}\right) \mathbf{Y}=\mathbf{0}$, $\mathbf{Y}$
being the unknown vector, with the classic method of linear algebra. Because 
$\mathbf{Y}=\left( \mathbf{B}^{(s)}\right) ^{-1}\mathbf{X},$ the triplet $%
\mathbf{Y}$ is rationally independent, then $y_{3}\neq 0$. If we just want
to obtain one vector solution, we may even suppose that $y_{3}=1$. Let $%
\mathbf{A}$ be the matrix $\left( \widetilde{\mathbf{B}}-\lambda \mathbf{I}%
\right) $ without its third row, let $\mathbf{a}_{2}$ be the third column of
the matrix $\mathbf{A}$, and let $\mathbf{A}^{\prime }$ be the matrix $%
\mathbf{A}$ without its third column. Let also $\mathbf{Y}^{\prime }$ be the
vector $\mathbf{Y}$ without its third coordinate. With block submatrices, we
have to solve in $\mathbf{Y}^{\prime }$ the equation: $%
\begin{bmatrix}
\mathbf{A}^{\prime } & \text{;\ }\mathbf{a}_{2}\text{ }%
\end{bmatrix}%
\times 
\begin{bmatrix}
\mathbf{Y}^{\prime } \\ 
1%
\end{bmatrix}%
=\mathbf{0}$. This gives $\mathbf{A}^{\prime }\times \mathbf{Y}^{\prime }+%
\mathbf{a}_{2}\times 1=\mathbf{0}$, or

$\mathbf{Y}^{\prime }=-\left( \mathbf{A}^{\prime }\right) ^{-1}\mathbf{a}%
_{2} $. We know that $\det \left( \mathbf{A}^{\prime }\right) \neq 0$,
because $F\left( \xi \right) $ is the minimal polynomial of $\lambda $. Then
we have $\ \mathbf{X}^{\left( s\right) }=%
\begin{bmatrix}
-\alpha \left( \mathbf{A}^{\prime }\right) ^{-1}\mathbf{a}_{2} \\ 
\alpha%
\end{bmatrix}%
$, for some $\alpha $, and then:

$\mathbf{X}=\mathbf{B}^{\left( s\right) }\mathbf{X}^{\left( s\right)
}=\alpha \mathbf{B}^{\left( s\right) }\times 
\begin{bmatrix}
-\left( \mathbf{A}^{\prime }\right) ^{-1}\mathbf{a}_{2} \\ 
1%
\end{bmatrix}%
$. Of course, $\mathbf{X}$ is defined up to a multiplicative coefficient.
Note that $\mathbf{A}^{\prime }$ and $\mathbf{a}_{2}$ are rational fractions
of the unit $\lambda $. Then the vector $\mathbf{Z:=}%
\begin{bmatrix}
-\left( \mathbf{A}^{\prime }\right) ^{-1}\mathbf{a}_{2} \\ 
1%
\end{bmatrix}%
$ can be also obtained as an expression of $\lambda $. Let $\widehat{\mathbf{%
b}}_{0}$, $\widehat{\mathbf{b}}_{1}$, $\widehat{\mathbf{b}}_{2}$ be the
three rows of the matrix $\mathbf{B}^{\left( s\right) }$; then we have $%
\left( 
\begin{array}{c}
x_{0}/\alpha \\ 
x_{1}/\alpha \\ 
x_{2}/\alpha%
\end{array}%
\right) =%
\begin{pmatrix}
\widehat{\mathbf{b}}_{0} \\ 
\widehat{\mathbf{b}}_{1} \\ 
\widehat{\mathbf{b}}_{2}%
\end{pmatrix}%
\mathbf{Z=}%
\begin{pmatrix}
\widehat{\mathbf{b}}_{0}\mathbf{Z} \\ 
\widehat{\mathbf{b}}_{1}\mathbf{Z} \\ 
\widehat{\mathbf{b}}_{2}\mathbf{Z}%
\end{pmatrix}%
$. Then $\frac{x_{0}}{x_{2}}=\frac{\widehat{\mathbf{b}}_{0}\mathbf{Z}}{%
\widehat{\mathbf{b}}_{2}\mathbf{Z}}$ and $\frac{x_{1}}{x_{2}}=\frac{\widehat{%
\mathbf{b}}_{1}\mathbf{Z}}{\widehat{\mathbf{b}}_{2}\mathbf{Z}}$ and we can
express$\ \frac{x_{0}}{x_{2}}$ and $\frac{x_{1}}{x_{2}}$ as rational
fractions of $\lambda $.
\end{proof}

\subsection{Demonstration of the Lagrange Property (particular case)}

In this subsection, we prove the Theorem on the Lagrange Property, admitting
the Dirichlet Theorem which is proved in other section; we're doing this
proof \textbf{only in the special case }$\mathbf{X}=~^{\textsc{T}}\left(
1,\theta ,\theta ^{2}\right) $, with $\ \theta =\sqrt[3]{N}$, $N$ \ being an
natural number, but not $\theta $. \ \textbf{For the complete proof, see
Section 5.}. We're going to prove that the Smallest Vector Algorithm in this
case makes a loop.

It is a basic result in Diophantine Approximation that there exists a real
number $e>0$ such that for any non-null integer point $\mathbf{h=\ }^{\text{T%
}}\left( m,n,p\right) ,$ the inequality $\left( \max \left( \left\vert
m\right\vert ,\left\vert n\right\vert ,\left\vert p\right\vert \right)
\right) ^{2}\times \left\vert \mathbf{h}\bullet \mathbf{X}\right\vert >e$ \
holds. See for instance \cite{CasselsDA} (Cassels), Theorem III, page 79,
statement (2), which is much stronger. But in finite dimension, all the
norms are equivalent. Then there exists $d>0$ such that for any non-null
integer point $\mathbf{h,}$ the inequality $\left\Vert \mathbf{h}\right\Vert
^{2}\left\vert \mathbf{h}\bullet \mathbf{X}\right\vert >d$ holds, id est $%
\left\Vert \mathbf{h}\right\Vert ^{2}\left\Vert \mathbf{h}^{\prime \prime
}\right\Vert \times \left\Vert \mathbf{X}\right\Vert >d.$ 
\newline
Then, by our Dirichlet Property Theorem 2, which we admit temporarily, there exists an infinite set of integers $S$ such that: 
\newline
$\underset{s\in S}{\sup }\left[ \left( \underset{i=0,1,2}{\max }\left\Vert 
\mathbf{b}_{i}^{\prime (s)}\right\Vert \right) ^{2}\left( \underset{i=0,1,2}{%
\max }\left\Vert \mathbf{b}^{\prime \prime }{}_{i}^{(s)}\right\Vert \right) %
\right] =L<+\infty $, with in addition:

$\underset{s\rightarrow +\infty ,\text{ }s\in S}{\lim }\left( \underset{%
i=0,1,2}{\max }\left\Vert \mathbf{b}_{i}^{\prime }{}^{(s)}\right\Vert
\right) =0.$

Let $s$ be any element of $S$, and let $\left( \mathbf{b}_{0}^{(s)},\mathbf{b%
}_{1}^{(s)},\mathbf{b}_{2}^{(s)}\right) $ be the column vectors of $\left( 
\mathbf{B}^{(s)}\right)$. We choose one of those three vectors, say $%
\mathbf{b}_{0}^{(s)}$, which will be more simply denoted: $\mathbf{b}^{(s)}=%
\mathbf{b}_{0}^{(s)}$. Let's define its coordinates: $\mathbf{b}^{(s)}=~^{%
\text{T}}\left( b_{x}^{(s)},b_{y}^{(s)},b_{z}^{(s)}\right) $. From now on
and for a while, we may omit the indices $^{(s)}$.

\begin{notation}
The notation $\mathbf{M}_{\mathbf{b}}^{(s)}=\mathbf{M}_{\mathbf{b}}$ or $%
\mathbf{M}\left[ \mathbf{b}^{(s)}\right] $ will denote the integer matrix: $%
\mathbf{M}_{\mathbf{b}}=\mathbf{M}\left[ \mathbf{b}^{(s)}\right] =\left( 
\begin{array}{ccc}
b_{z} & b_{y} & b_{x} \\ 
N\cdot b_{x} & b_{z} & b_{y} \\ 
N\cdot b_{y} & N\cdot b_{x} & b_{z}%
\end{array}%
\right) $, which has the interesting property: $\mathbf{M}_{\mathbf{b}}%
\mathbf{X=M}_{\mathbf{b}}\left( 
\begin{array}{l}
1 \\ 
\theta \\ 
\theta ^{2}%
\end{array}%
\right) =\left( b_{z}+b_{y}\theta +b_{x}\theta ^{2}\right) \left( 
\begin{array}{l}
1 \\ 
\theta \\ 
\theta ^{2}%
\end{array}%
\right) $. Let's denote by $\lambda _{\mathbf{b}}$ or $\lambda \left[ 
\mathbf{b}^{(s)}\right] $ the following element of $\mathbb{Z}\left[ \theta %
\right] $:\newline
 $\lambda _{\mathbf{b}}:=\left( b_{z}+b_{y}\theta +b_{x}\theta
^{2}\right) $.
Then we have: $\mathbf{M}_{\mathbf{b}}\mathbf{X}=\lambda _{\mathbf{b}}%
\mathbf{X}$; i.e. $\lambda _{\mathbf{b}}$ is an eigenvalue of $\mathbf{M}_{%
\mathbf{b}}$ with eigenvector $\mathbf{X}$.
\end{notation}

We consider the sequence of the matrices $\mathbf{\Pi }^{(s)}=\mathbf{\Pi }=$
$^{\textsc{T}}\mathbf{M}_{\mathbf{b}}\mathbf{G}$. Let $\left( \mathbf{b}%
^{\#\#},\mathbf{b}^{\#},\mathbf{b}\right) $ be the three column vectors of $%
\mathbf{M}_{\mathbf{b}}$.

Then $\mathbf{\Pi }=$ $^{\textsc{T}}\mathbf{M}_{\mathbf{b}}\mathbf{G=}%
\left( 
\begin{array}{ccc}
\mathbf{b}^{\#\#}\bullet \mathbf{g}_{0} & \mathbf{b}^{\#\#}\bullet \mathbf{g}%
_{1} & \mathbf{b}^{\#\#}\bullet \mathbf{g}_{2} \\ 
\mathbf{b}^{\#}\bullet \mathbf{g}_{0} & \mathbf{b}^{\#}\bullet \mathbf{g}_{1}
& \mathbf{b}^{\#}\bullet \mathbf{g}_{2} \\ 
\mathbf{b}\bullet \mathbf{g}_{0} & \mathbf{b}\bullet \mathbf{g}_{1} & 
\mathbf{b}\bullet \mathbf{g}_{1}%
\end{array}%
\right):=\left( \pi _{i,j}\right) $, $i=0,1,2$; $j=0,1,2$.

\begin{lemma}[Main Lemma for Lagrange]
The matrices $\Pi ^{(s)}$are bounded independently from $s$.
\end{lemma}

\begin{proof}
We have to find an upper bound for each of the $\left\vert \pi
_{i,j}\right\vert $, $i=0,1,2$; $j=0,1,2$. Let's consider first the $\left\vert \mathbf{b}^{\#}\bullet \mathbf{g}%
_{i}\right\vert $.\newline
We have: $\mathbf{b}^{\#}=\mathbf{Q}^{\#}\left( 
\begin{array}{c}
b_{x} \\ 
b_{y} \\ 
b_{z}%
\end{array}%
\right) =\mathbf{Q}^{\#}\mathbf{b}$, with $\mathbf{Q}^{\#}=\left( 
\begin{array}{ccc}
0 & 1 & 0 \\ 
0 & 0 & 1 \\ 
N & 0 & 0%
\end{array}%
\right) $. We also have $\mathbf{Q}^{\#}\mathbf{X}=\theta \mathbf{X}$. Let $%
\mathbf{n}$ be: $\mathbf{n=}\dfrac{\mathbf{X}}{\left\Vert \mathbf{X}%
\right\Vert }$. Then also $\mathbf{Q}^{\#}\mathbf{n=}\theta \mathbf{n}$.

We have, with always the same kind of notations, $\mathbf{b}^{\#}=\mathbf{b}%
^{\#\prime }+\mathbf{b}^{\#\prime \prime }$.\newline
$\mathbf{b}^{\#}\bullet \mathbf{g}_{i}=\left( \mathbf{b}^{\#\prime }+\mathbf{%
b}^{\#\prime \prime }\right) \bullet \left( \mathbf{g}_{i}^{\prime }+\mathbf{%
g}_{i}^{\prime \prime }\right) =\mathbf{b}^{\#\prime }\bullet \mathbf{g}%
_{i}^{\prime }+\mathbf{b}^{\#\prime \prime }\bullet \mathbf{g}_{i}^{\prime
\prime }$,

with $\mathbf{g}_{i}^{\prime }=\pm \left( \mathbf{b}_{j}^{\prime \prime
}\wedge \mathbf{b}_{k}^{\prime }+\mathbf{b}_{j}^{\prime }\wedge \mathbf{b}%
_{k}^{\prime \prime }\right) $; $\mathbf{g}_{i}^{\prime \prime }=\pm \mathbf{%
b}_{j}^{\prime }\wedge \mathbf{b}_{k}^{\prime }$. Then:%
\begin{equation*}
\left\vert \mathbf{b}^{\#}\bullet \mathbf{g}_{i}\right\vert \leq \left\Vert 
\mathbf{b}^{\#\prime }\right\Vert \left\Vert \mathbf{b}^{\prime \prime
}{}_{j}\right\Vert \left\Vert \mathbf{b}_{k}^{\prime }\right\Vert
+\left\Vert \mathbf{b}^{\#\prime }\right\Vert \left\Vert \mathbf{b}%
_{j}^{\prime }\right\Vert \left\Vert \mathbf{b}^{\prime \prime
}{}_{k}\right\Vert +\left\Vert \mathbf{b}^{\#\prime \prime }\right\Vert
\left\Vert \mathbf{b}_{j}^{\prime }\right\Vert \left\Vert \mathbf{b}%
_{k}^{\prime }\right\Vert .
\end{equation*}%
Let's denote: $\beta ^{\prime }=\underset{i=0,1,2}{\max }\left\Vert \mathbf{b%
}_{i}^{\prime }\right\Vert $ and $\beta ^{\prime \prime }=$ $\underset{%
i=0,1,2}{\max }\left\Vert \mathbf{b}_{i}^{^{\prime \prime }}\right\Vert $,
so that: $\left( \beta ^{\prime }\right) ^{2}\beta ^{\prime \prime }\leq L$.
Then: 
\begin{equation}
\left\vert \mathbf{b}^{\#}\bullet \mathbf{g}_{i}\right\vert \leq \left\Vert 
\mathbf{b}^{\#\prime }\right\Vert \beta ^{\prime \prime }\beta ^{\prime
}+\left\Vert \mathbf{b}^{\#\prime }\right\Vert \beta ^{\prime }\beta
^{\prime \prime }+\left\Vert \mathbf{b}^{\#\prime \prime }\right\Vert \left(
\beta ^{\prime }\right) ^{2}.\label{IneqDiese}
\end{equation}%
But we also have: $\mathbf{b}^{\#}=\mathbf{Q}^{\#}\mathbf{b}=\mathbf{Q}%
^{\#}\left( \mathbf{b}_{0}^{\prime }+\left\Vert \mathbf{b}_{0}^{\prime
\prime }\right\Vert \mathbf{n}\right) =\mathbf{Q}^{\#}\mathbf{b}_{0}^{\prime
}+\left\Vert \mathbf{b}_{0}^{\prime \prime }\right\Vert \theta \mathbf{n}$. %
This proves first that the distance $\left\Vert \mathbf{b}^{\#\prime
}\right\Vert$ between $\mathbf{b}^{\#}$ and $\mathbb{D}$ is less than $%
\left\Vert \mathbf{Q}^{\#}\mathbf{b}_{0}^{\prime }\right\Vert $: 
\begin{equation}
\left\Vert \mathbf{b}^{\#\prime }\right\Vert \leq \left\Vert \mathbf{Q}^{\#}%
\mathbf{b}_{0}^{\prime }\right\Vert \leq \left\Vert \mathbf{Q}%
^{\#}\right\Vert \times \left\Vert \mathbf{b}_{0}^{\prime }\right\Vert \leq
\left\Vert \mathbf{Q}^{\#}\right\Vert \beta ^{\prime }  \label{InegPrim}
\end{equation}%
(this using the norm of the matrix). 
Furthermore, we have the following equality: $\mathbf{b}^{\#\prime }+\mathbf{b}^{\#\prime \prime }=%
\mathbf{b}^{\#}=\mathbf{Q}^{\#}\mathbf{b}_{0}^{\prime }+\left\Vert \mathbf{b}%
_{0}^{\prime \prime }\right\Vert ~\theta \mathbf{n}$, and we deduce:\newline
$\mathbf{b}^{\#\prime \prime }$ $=\left\Vert \mathbf{b}_{0}^{\prime \prime
}\right\Vert ~\theta \mathbf{n+\mathbf{Q}^{\#}\mathbf{b}_{0}^{\prime }-b}%
^{\#\prime }$. Then:\begin{equation}
\left\Vert \mathbf{b}^{\#\prime \prime }\right\Vert \leq \left\Vert \mathbf{b%
}_{0}^{\prime \prime }\right\Vert \theta +2\left\Vert \mathbf{Q}%
^{\#}\right\Vert ~\left\Vert \mathbf{b}_{0}^{\prime }\right\Vert \leq \beta
^{\prime \prime }\theta +2\left\Vert \mathbf{Q}^{\#}\right\Vert \beta
^{\prime }  \label{Ineq Sec}
\end{equation}%
Putting \ref{InegPrim} and \ref{Ineq Sec} in \ref{IneqDiese}, we obtain:

$\left\vert \mathbf{b}^{\#(s)}\bullet \mathbf{g}_{i}^{(s)}\right\vert \leq
\left( \beta ^{\prime (s)}\right) ^{2}\beta ^{\prime \prime (s)}\left(
2\left\Vert \mathbf{Q}^{\#}\right\Vert +\theta \right) +2\left\Vert \mathbf{Q%
}^{\#}\right\Vert \left( \beta ^{\prime (s)}\right) ^{3}$

and then: $\left\vert \mathbf{b}^{\#(s)}\bullet \mathbf{g}%
_{i}^{(s)}\right\vert \leq L\left( 2\left\Vert \mathbf{Q}^{\#}\right\Vert
+\theta \right) +2\left\Vert \mathbf{Q}^{\#}\right\Vert \left( \beta
^{\prime (s)}\right) ^{3}$.

The limit of the last term is $0$, by the \textbf{c)} of the Dirichlet
Properties Theorem. Then the set of the $\left\vert \mathbf{b}%
^{\#(s)}\bullet \mathbf{g}_{i}^{(s)}\right\vert $, with $s$ in $S$, is
bounded. By a similar demonstration, we prove that the numbers $\left\vert 
\mathbf{b}^{\#\#(s)}\bullet \mathbf{g}_{i}^{(s)}\right\vert $ are also
bounded and so are, in an obvious way, the numbers $\left\vert \mathbf{b}%
_{0}^{(s)}\bullet \mathbf{g}_{i}^{(s)}\right\vert $. Then the set of the $%
\mathbf{\Pi }^{(s)}$ is bounded, and the Lemma is proved.
\end{proof}

With this Lemma, the demonstration of the Lagrange result is easy.

\begin{proof}[Proof of the Lagrange Result]
The set of the $\mathbf{\Pi }^{(s)}$ with $s$ in $S$ is bounded; but these
matrices have integral coefficients. Then the number of all the $\mathbf{\Pi 
}^{(s)}$ is finite. Then there exist $s$ and $t$, with $s<t$, such that $%
\mathbf{\Pi }^{(t)}=\mathbf{\Pi }^{(s)}$. This means: $^{\textsc{T}}\mathbf{M}\left[ \mathbf{b}^{(t)}\right] \times \mathbf{G}%
^{(t)}=$~$^{\textsc{T}}\mathbf{M}\left[ \mathbf{b}^{(s)}\right] \times 
\mathbf{G}^{(s)}$. By transposition:

$\left( \mathbf{B}^{(t)}\right) ^{-1}\times \mathbf{M}\left[ \mathbf{b}^{(t)}%
\right] =\left( \mathbf{B}^{(s)}\right) ^{-1}\times \mathbf{M}\left[ \mathbf{%
b}^{(s)}\right] $. We apply that to the column vector $\mathbf{X}$; we
obtain $\left( \mathbf{B}^{(t)}\right) ^{-1}\times \mathbf{M}\left[ \mathbf{b%
}^{(t)}\right] \times \mathbf{X}=\left( \mathbf{B}^{(s)}\right) ^{-1}\times 
\mathbf{M}\left[ \mathbf{b}^{(s)}\right] \times \mathbf{X}$; then, by
\textquotedblright eigenvalue\textquotedblright , see Notation above:

$\left( \mathbf{B}^{(t)}\right) ^{-1}\times $ $\left( \lambda \left[ \mathbf{%
b}^{(t)}\right] \cdot \mathbf{X}\right) \mathbf{=}\left( \mathbf{B}%
^{(s)}\right) ^{-1}\times $ $\left( \lambda \left[ \mathbf{b}^{(s)}\right]
\cdot \mathbf{X}\right) $.\newline
We recall that $\lambda \left[ \mathbf{b}^{(s)}%
\right] $ is a real number in $\mathbb{Z}\left[ \theta \right] $.

Then $\left( \mathbf{B}^{(t)}\right) ^{-1}\mathbf{X=}\dfrac{\lambda \left[ 
\mathbf{b}^{(s)}\right] }{\lambda \left[ \mathbf{b}^{(t)}\right] }\cdot
\left( \mathbf{B}^{(s)}\right) ^{-1}\mathbf{X}$.

This reads $\mathbf{X}^{\left( t\right) }=\lambda \mathbf{X}^{\left(
s\right) }$, with $\lambda =\frac{\lambda \left[ \mathbf{b}^{(s)}\right] }{%
\lambda \left[ \mathbf{b}^{(t)}\right] }$, and the first part of the Theorem
is obtained.

We have already proved the Second Part of the Theorem, by Lemma 4 of the
previous subsection. With this Second Part, we also obtain the final
assertions of the First Part.

In particular, $x_{0},x_{1},x_{2},$ are rationally independent, then so are
also $\frac{x_{0}}{x_{2}}$ , $\frac{x_{1}}{x_{2}}$ and $1$. Then $\left( 
\frac{x_{0}}{x_{2}},\frac{x_{1}}{x_{2}},1\right) $ is a basis of $%
\mathbb{Q}
\left( \rho \right) $ over $%
\mathbb{Q}
$. But we know by the Lemma 4 that $\frac{x_{0}}{x_{2}}$ and $\frac{x_{1}}{%
x_{2}}$ are rational fractions of $\lambda $, therefore they belong to $%
\mathbb{Q}
\left( \lambda \right) $ and then $%
\mathbb{Q}
\left( \rho \right) =%
\mathbb{Q}
\left( \lambda \right) $.
\end{proof}

\section{From the Geometrical Property to the Dirichlet Properties}

In this section, the Geometrical Theorem is admitted, and we prove the 
\textbf{Theorem 2} on Dirichlet Properties.

\subsection{Proof of Part a) of Theorem 2}

\begin{lemma}[Prism Lemma]
Let $\mathbf{g}_{0}^{(s)},\mathbf{g}_{1}^{(s)},$ $\mathbf{g}_{2}^{(s)}$be
the column vectors of the matrix $\mathbf{G}^{(s)}$ generated at the $s$-th
stage by the Smallest Vector Algorithm. Let the sets $\mathbf{H}^{\prime
(s)} $ be the convex hulls:

$\mathbf{H}^{\prime (s)}=\textsc{conv}\left( \mathbf{g}_{0}^{\prime (s)},%
\mathbf{g}_{1}^{\prime (s)},\mathbf{g}_{2}^{\prime (s)},-\mathbf{g}%
_{0}^{\prime (s)},-\mathbf{g}_{1}^{\prime (s)},-\mathbf{g}_{2}^{\prime
(s)}\right) $ in $\mathbb{P}$ .

We shall omit the indices $^{(s)}$. Let $\mathbf{H}$ be the prism: $\mathbf{H%
}:=\mathbf{H}^{\prime }+\mathbb{D}$. Then, with the usual notation $\mathbf{h%
}^{\prime \prime }$ for the orthogonal projection of $\mathbf{h}$ on $%
\mathbb{D}$:

-For each non zero integer point $\mathbf{h}$ in $\mathbf{H}$ , $\left\| 
\mathbf{h}^{\prime \prime }\right\| \geq \left\| \mathbf{g}_{0}^{\prime
\prime }\right\| $ holds.

-For each integer point $\mathbf{h}$ in $\mathbf{H}$ which is not of the
form $n_{0}\mathbf{g}_{0}$, with $n_{0}$ integer, $\left\Vert \mathbf{h}%
^{\prime \prime }\right\Vert \geq \left\Vert \mathbf{g}_{1}^{\prime \prime
}\right\Vert $ holds.

-For each integer point $\mathbf{h}$ in $\mathbf{H}$ \ which is not of the
form $n_{0}\mathbf{g}_{0}+n_{1}\mathbf{g}_{1}$, with $n_{0}$ and $n_{1}$
integers, $\left\Vert \mathbf{h}^{\prime \prime }\right\Vert \geq \left\Vert 
\mathbf{g}_{2}^{\prime \prime }\right\Vert $ holds.

- This implies that, if $\left( \mathbf{h}_{0},\mathbf{h}_{1},\mathbf{h}%
_{2}\right) $ is a free triplet of integer points in $\mathbf{H}$, then $%
\underset{i=0,1,2}{\max }\left( \left\Vert \mathbf{h}_{i}^{\prime \prime
}\right\Vert \right) \geq \underset{i=0,1,2}{\max }\left( \left\Vert \mathbf{%
g}_{i}^{\prime \prime }\right\Vert \right) =\left\Vert \mathbf{g}%
_{2}^{\prime \prime }\right\Vert $.
\end{lemma}

\begin{proof}
Let $\mathbf{h}$ be any non-null integer vector in $\mathbf{H}$. Since $\det
\left( \mathbf{G}^{(s)} \right) =\pm 1$, there exist three relative integers $n_{0}$, $n_{1}$%
, $n_{2}$ such that $\mathbf{h=}n_{0}\mathbf{g}_{0}+n_{1}\mathbf{g}_{1}+n_{2}%
\mathbf{g}_{2}$. Let $\mathbf{h}^{\prime }$ be the orthogonal projection of $%
\mathbf{h}$ on $\mathbb{P}$. We have $\mathbf{h}^{\prime }\mathbf{=}n_{0}%
\mathbf{g}_{0}^{\prime }+n_{1}\mathbf{g}_{1}^{\prime }+n_{2}\mathbf{g}%
_{2}^{\prime }$ and also $\mathbf{h}^{\prime }\in \mathbf{H}^{\prime }$,
which is the convex hull of six points. But an easy geometrical study shows
that, in the plane, if $\mathbf{p}$ is in the convex hulls of six points, it
has to be in the convex hull of three among these six points. In this proof,
"positive" will mean: "$\geq 0$", and not "$>0".$

Then $\mathbf{h}^{\prime }$ is the positive barycenter of three points among 
$\mathbf{g}_{0}^{\prime }$, $\mathbf{g}_{1}^{\prime },$ $\mathbf{g}%
_{2}^{\prime },$ $-\mathbf{g}_{0}^{\prime }$, $-\mathbf{g}_{1}^{\prime }$, $-%
\mathbf{g}_{2}^{\prime }.$ This means that $\mathbf{h}^{\prime }$ is the
positive barycenter of three points of the shape $\varepsilon _{0}\mathbf{g}%
_{0}^{\prime }$, $\varepsilon _{1}\mathbf{g}_{1}^{\prime }$, $\varepsilon
_{2}\mathbf{g}_{2}^{\prime }$, with $\varepsilon _{i}\in \left\{
-1,1\right\} $. Then there exist positive real numbers $y_{0}$, $y_{1}$, $%
y_{2}$ such that $y_{0}+y_{1}+y_{2}=1$ and such that

$\mathbf{h}^{\prime }=y_{0}\varepsilon _{0}\mathbf{g}_{0}^{\prime
}+y_{1}\varepsilon _{1}\mathbf{g}_{1}^{\prime }+y_{2}\varepsilon _{2}\mathbf{%
g}_{2}^{\prime }=n_{0}\mathbf{g}_{0}^{\prime }+n\mathbf{g}_{1}^{\prime
}+n_{2}\mathbf{g}_{2}^{\prime }$.

Then $\left( n_{0}-y_{0}\varepsilon _{0}\right) \mathbf{g}_{0}^{\prime
}+\left( n_{1}-y_{1}\varepsilon _{1}\right) \mathbf{g}_{1}^{\prime }+\left(
n_{2}-y_{2}\varepsilon _{2}\right) \mathbf{g}_{2}^{\prime }=\mathbf{0}$. But
we have also $\left\Vert \mathbf{X}\right\Vert \left\Vert \mathbf{b}_{0}^{\prime \prime
}{}^{(s)}\right\Vert \mathbf{g}_{0}^{(s)}+\left\Vert \mathbf{X}\right\Vert
\left\Vert \mathbf{b}_{1}^{\prime \prime }{}^{(s)}\right\Vert \mathbf{g}%
_{1}^{(s)}+\left\Vert \mathbf{X}\right\Vert \left\Vert \mathbf{b}%
_{2}^{\prime \prime }{}^{(s)}\right\Vert \mathbf{g}_{2}^{(s)}=\mathbf{X,}$
like in the first Lemma of the Section 2, and then, with $z_{i}=\left\Vert 
\mathbf{X}\right\Vert \left\Vert \mathbf{b}_{0}^{\prime \prime
}{}^{(s)}\right\Vert $, we have:\newline$%
z_{0}\mathbf{g}_{0}^{\prime }+z_{1}\mathbf{g}_{1}^{\prime }+z_{2}\mathbf{g}%
_{2}^{\prime }=\mathbf{0}$, with $z_{i}>0$, by orthogonal projection on $\mathbb{P}$.

By the uniqueness of the barycentrical coordinates, up to a multiplicative
coefficient, we get: for some real number $\lambda $,

$n_{0}-y_{0}\varepsilon _{0}=\lambda z_{0}$; $n_{1}-y_{1}\varepsilon
_{1}=\lambda z_{1}$; $n_{2}-y_{2}\varepsilon _{2}=\lambda z_{2}$.

Then $n_{0}-y_{0}\varepsilon _{0}$, $n_{1}-y_{1}\varepsilon _{1\text{, }%
}n_{2}-y_{2}\varepsilon _{2}$ have the same sign ($0$ has both signs).
Without loss of generality, this sign may be supposed to be positive. 
Then we have $n_{0}\geq y_{0}\varepsilon _{0}\geq -1$, $n_{1}\geq y_{1}\varepsilon
_{1}\geq -1$, $n_{2}\geq y_{2}\varepsilon _{2}\geq -1$.

- If for every $i=1,2,3$, we have $n_{i}>-1$ h, then $n_{0}$, $n_{1}$, $n_{2}$
have the same sign.

- If now, for some $i$, we have $n_{i}=y_{i}\varepsilon _{i}=-1$, then $%
y_{i}=1$, and $y_{j}=y_{k}=0$, with $\left\{ i,j,k\right\} =\left\{
0,1,2\right\} $. Moreover, $\lambda =0$, and then $n_{i}=-1$, $%
n_{j}=y_{j}\varepsilon _{j}=0$, $n_{k}=y_{k}\varepsilon _{k}=0$. Then, in
every case, $n_{0}$, $n_{1}$, $n_{2}$ have the same sign. Then $\left\Vert \mathbf{h}^{\prime \prime }\right\Vert =\left\Vert n_{0}\mathbf{g%
}_{0}^{\prime \prime }+n_{1}\mathbf{g}_{1}^{\prime \prime }+n_{2}\mathbf{g}%
_{0}^{\prime \prime }\right\Vert =\left\vert n_{0}\right\vert \left\Vert 
\mathbf{g}_{0}^{\prime \prime }\right\Vert $+$\left\vert n_{1}\right\vert
\left\Vert \mathbf{g}_{1}^{\prime \prime }\right\Vert +\left\vert
n_{2}\right\vert \left\Vert \mathbf{g}_{2}^{\prime \prime }\right\Vert $,

and the conclusions of the theorem become obvious.
\end{proof}

Let's notice that the Prism Lemma shows that in some way, our algorithm
gives \emph{best integer approximations} of the plane on $\mathbb{P}$. In
fact, each $\mathbf{g}_{0}^{(s)}$ is a best approximation, and so are, after
the $n_{0}\mathbf{g}_{0}^{(s)}$, the vector $\mathbf{g}_{1}^{(s)}$ and,
after the vectors $n_{0}\mathbf{g}_{0}^{(s)}+n_{1}\mathbf{g}_{1}^{(s)}$, the
vector $\mathbf{g}_{2}^{(s)}$. The vectors $\mathbf{g}_{0}^{(s)},\mathbf{g}%
_{1}^{(s)},\mathbf{g}_{2}^{(s)}$ are not necessarily successive best
approximations with disks of the euclidean norm in $\mathbb{P}$, but so are
they for the "hexagon" (which can be a parallelogram) $\mathbf{H}^{\prime
(s)},$ which depends on $\left( \mathbf{g}_{0}^{(s)},\mathbf{g}_{1}^{(s)},%
\mathbf{g}_{2}^{(s)}\right) $ themselves.

\begin{proof}
Let's now prove the assertion a) of Theorem 2.

We recall that, by convention,
$\left\Vert \mathbf{g}^{\prime \prime }{}_{0}^{(s+1)}\right\Vert \leq
\left\Vert \mathbf{g}^{\prime \prime }{}_{1}^{(s+1)}\right\Vert \leq
\left\Vert \mathbf{g}^{\prime \prime }{}_{2}^{(s+1)}\right\Vert $ and we
denote by $\left( \mathbf{g}_{\mathtt{I}}^{(s)},\mathbf{g}_{\mathtt{II}%
}^{(s)},\mathbf{g}_{\mathtt{III}}^{(s)}\right) $ the permutation of $\left( 
\mathbf{g}_{0}^{(s)},\mathbf{g}_{1}^{(s)},\mathbf{g}_{2}^{(s)}\right) $ such
that $\left\Vert \mathbf{g}_{\mathtt{I}}^{\prime (s)}\right\Vert $ $\leq $ $%
\left\Vert \mathbf{g}_{\mathtt{II}}^{\prime (s)}\right\Vert $ $\leq
\left\Vert \mathbf{g}_{\mathtt{III}}^{\prime (s)}\right\Vert $.

We admit the Geometrical Theorem, namely that there exists an infinite set $%
S $ of natural integers such that: $\underset{s\in S}{\sup }\frac{\left\Vert 
\mathbf{g}_{\mathtt{III}}^{\prime (s)}\right\Vert }{\rho ^{(s)}}=L<+\infty .$
This will be proved in the next section. \newline
For any $s\in S$, we consider the cylinder with center at $\mathbf{0}$, with basis in $\mathbb{P}$, 
radius $\rho ^{(s)}$, and height $\frac{8}{\pi \left( \rho ^{(s)}\right) ^{2}} $,  namely: 
\begin{equation*}
\Gamma ^{(s)}=\text{Disk}~^{\prime }\left( \rho ^{(s)}\right) +\text{%
Disk}~^{\prime \prime }\left( \frac{4}{\pi \left( \rho ^{(s)}\right) ^{2}}%
\right)
\end{equation*}%
(the second disk being in $\mathbb{D}$, and being a segment). The volume of $%
\Gamma ^{(s)}$ is $8$. Then, by Minkowski's first Theorem, there is an
integer point $\mathbf{h\neq 0}$ in $\Gamma ^{(s)}$. For this theorem, see
for instance, \cite{CasselsDA}(Cassels), Theorem IV, page 154. Then we have, by the Prism Lemma for the second inequality,%
\begin{equation*}
\left\Vert \mathbf{g}_{\mathtt{III}}^{\prime (s)}\right\Vert ^{2}\left\Vert 
\mathbf{g}_{0}^{\prime \prime \left( s\right) }\right\Vert \leq L^{2}\left(
\rho ^{(s)}\right) ^{2}\left\Vert \mathbf{g}_{0}^{\prime \prime \left(
s\right) }\right\Vert \leq L^{2}\left( \rho ^{(s)}\right) ^{2}\left\Vert 
\mathbf{h}^{\prime \prime }\right\Vert \leq \dfrac{4L^{2}}{\pi }
\end{equation*}%
or $\left\Vert \mathbf{g}_{\mathtt{III}}^{\prime (s)}\right\Vert
^{2}\left\Vert \mathbf{g}_{0}^{\prime \prime \left( s\right) }\right\Vert
\leq M$ a constant, and the main statement of \textbf{a) }is proved.
This last result, with the help of the Lemma 2 of Section 2, namely $\underset{%
s\rightarrow +\infty ,\text{ }s\in S}{\lim }\left\Vert \mathbf{g}_{\mathtt{%
III}}^{\prime (s)}\right\Vert =+\infty $, implies the last statement of 
\textbf{a), }id est $\underset{s\rightarrow +\infty ,\text{ }s\in S}{\lim }%
\left\Vert \mathbf{g}_{0}^{\prime \prime \left( s\right) }\right\Vert =0.$
\end{proof}

\subsection{Demonstration of assertions b) and c) in Theorem 2}

We shall need some well known results in Diophantine Approximation or
Geometry of Numbers.

\begin{lemma}[Transference Theorem in dim 2]
Let $1,$ $\alpha ,$ $\beta $ be three rationally independent real numbers,
let $\mathbf{X}$ be: $\mathbf{X=}$\ $^{\textsc{T}}\left( 1,\alpha ,\beta
\right) $ and let $\mathbb{D}$ be $\mathbb{D}=\mathbb{R}\mathbf{X}$. Let $%
\mathbf{h:=}\left( m,n,p\right) $ be an integral point in $\mathbb{Z}%
^{3}\setminus \left\{ \mathbf{0}\right\} $. The following six assertions are
equivalent:

\textbf{(a) \ }$\underset{\mathbf{h}\text{ }=\left( m,n,p\right) \text{ with 
}\left( n,p\right) \neq \left( 0,0\right) }{\inf }\left\vert \mathbf{X}%
\bullet \mathbf{h}\right\vert \times \left[ \max \left( \left\vert
n\right\vert ,\left\vert p\right\vert \right) \right] ^{2}>0$

\textbf{(b) \ }$\underset{\mathbf{h}\text{ }=\left( m,n,p\right) \text{ with 
}m\neq 0}{\inf }\left\vert m\right\vert \times \left[ \max \left( \left\vert
\alpha m-n\right\vert ,\left\vert \beta m-p\right\vert \right) \right]
^{2}>0 $

$\left( \text{\textbf{a}}^{\prime }\right) $\textbf{~}$\underset{\mathbf{h}%
=\left( m,n,p\right) \in \mathbb{Z}^{3}\setminus \left\{ \mathbf{0}\right\} }%
{\inf }\left\vert \mathbf{X}\bullet \mathbf{h}\right\vert \times \left[ \max
\left( \left\vert m\right\vert ,\left\vert n\right\vert ,\left\vert
p\right\vert \right) \right] ^{2}>0$

$\left( \text{\textbf{b}}^{\prime }\right) ~\underset{\mathbf{h}=\left(
m,n,p\right) \in \mathbb{Z}^{3}\setminus \left\{ \mathbf{0}\right\} }{\inf }%
\max \left( \left\vert m\right\vert ,\left\vert n\right\vert ,\left\vert
p\right\vert \right) \times \left[ \max \left( \left\vert \alpha
m-n\right\vert ,\left\vert \beta m-p\right\vert \right) \right] ^{2}>0$

\textbf{(A) \ }$\underset{\mathbf{h}\in \mathbb{Z}^{3}\setminus \left\{ 
\mathbf{0}\right\} }{\inf }\left\vert \mathbf{X}\bullet \mathbf{h}%
\right\vert \times \left\Vert \mathbf{h}\right\Vert ^{2}>0$

\textbf{(B) \ }$\underset{\mathbf{h}=\left( m,n,p\right) \in \mathbb{Z}%
^{3}\setminus \left\{ \mathbf{0}\right\} }{\inf }\left\Vert \mathbf{h}%
\right\Vert \times \left( \left( \alpha m-n\right) ^{2}+\left( \beta
m-p\right) ^{2}\right) >0$
\end{lemma}

\begin{proof}
The equivalence of (a) with (b) is a classical result. See for instance: 
\cite{CasselsDA} (Cassels), Theorem II and Corollary. The other equivalences are also classical and easy.
\end{proof}

\begin{lemma}[Transference Theorem, geometrical point of view]
Let $1,$ $\alpha ,$ $\beta $ be three rationally independent real numbers,
let $\mathbf{X}$ be: $\mathbf{X=}$~$^{\textsc{T}}\left( 1,\alpha ,\beta
\right) $ and let $\mathbb{D}$ be $\mathbb{D}=\mathbb{R}\mathbf{X}$. As
usual, let $\mathbf{h}^{\prime }$ and $\mathbf{h}^{\prime \prime }$ be the
orthogonal projections of $\mathbf{h}$ on $\mathbb{D}$ and $\mathbb{P=D}%
^{\perp }$. The following three assertions are equivalent, and also are
equivalent to assertions (A) and (B) of the previous Lemma.

\textbf{(C) \ }$\underset{\mathbf{h}\in \mathbb{Z}^{3}\setminus \left\{ 
\mathbf{0}\right\} }{\inf }\left\Vert \mathbf{h}^{\prime \prime }\right\Vert
\times \left\Vert \mathbf{h}\right\Vert ^{2}>0$

\textbf{(D) \ }$\underset{\mathbf{h}\in \mathbb{Z}^{3}\setminus \left\{ 
\mathbf{0}\right\} }{\inf }\left\Vert \mathbf{h}\right\Vert \times
\left\Vert \mathbf{h}^{\prime }\right\Vert ^{2}>0$

\textbf{(E) \ }$\underset{\mathbf{h}\in \mathbb{Z}^{3}\setminus \left\{ 
\mathbf{0}\right\} }{\inf }\left\Vert \mathbf{h}^{\prime \prime }\right\Vert
\times \left\Vert \mathbf{h}^{\prime }\right\Vert ^{2}>0$
\end{lemma}
Again, the proofs are easy.
\begin{definition}
When one of the assertions (a), (b), (a'), (b'), (A), (B), (C), (D), (E) of
Lemmas 7 and 8 is true (id est, all of them), it will be said that \emph{the
couple }$\left( \mathbb{P},\mathbb{D}\right) $\emph{\ is badly approximable. 
}
\end{definition}

We shall also need famous Minkowski's theorem on successive minima.

\begin{lemma}[ Minkowski's Successive Minima Theorem]
For any convex set $E$ in $\mathbb{R}^{3}$ which is symmetric about $\mathbf{%
0}$, let $\Lambda _{i}\left( E\right) $, for $i=0,1$ or $2$, be the lower
bound of the numbers $\lambda $ such that $\lambda E$ contains $\left(
i+1\right) $ linearly independent integer vectors. Then, if the volume of $E$
is $8$, $\frac{1}{6}\leq \Lambda _{0}\left( E\right) \Lambda _{1}\left(
E\right) \Lambda _{2}\left( E\right) \leq 1$ holds. See Cassels \cite%
{Cassels} (Cassels) Ch. VIII, page 201 and following, especially assertions $%
\left\{ 12\right\} $ and $\left\{ 13\right\} $ page 203, or \cite{CasselsDA}
(Cassels), Theorem V page 156.
\end{lemma}

\begin{proof}[Proof of the\ b ) and c)\textbf{\ }of Th. 2 (Dirichlet
Properties)]
We suppose
\newline
 that the Hypothesis of the \textbf{b)} Property of Theorem 2
holds: there exists $c>0$ such that for any non-null integer point $\mathbf{h%
},$ $\left\Vert \mathbf{h}\right\Vert ^{2}\left\Vert \mathbf{h}^{\prime
\prime }\right\Vert >c$. Then, by our Geometrical Transference Lemma, there
exists $d>0$ such that for any non-null integer point $\mathbf{h},\left\Vert 
\mathbf{h}^{\prime }\right\Vert ^{2}\left\Vert \mathbf{h}^{\prime \prime
}\right\Vert >d$.

Like above, let $\Gamma _{R}$ be the cylinder: $\Gamma _{R}=\text{Disk}%
^{\prime }\left( R\right) +\text{Disk}^{\prime \prime }\left( \dfrac{4}{%
\pi R^{2}}\right) $.

Let's define $K_{0}=\left( \dfrac{\pi d}{4}\right) ^{\frac{1}{3}}$. \ The
inequality $\left\Vert \mathbf{h}^{\prime }\right\Vert ^{2}\left\Vert 
\mathbf{h}^{\prime \prime }\right\Vert >d$ implies that there's no non-null
integer point in $K_{0}\Gamma _{R}$. Then, for any $R$, $\Lambda _{0}\left(
\Gamma _{R}\right) \geq K_{0}$. In addition, we have $\Lambda _{0}\left(
\Gamma _{R}\right) \Lambda _{1}\left( \Gamma _{R}\right) \Lambda _{2}\left(
\Gamma _{R}\right) \leq 1$ and also $\Lambda _{1}\left( \Gamma _{R}\right)
\geq \Lambda _{0}\left( \Gamma _{R}\right) \geq K_{0}$ , so that we get $%
K_{0}^{2}\Lambda _{2}\left( \Gamma _{R}\right) \leq 1$ and then $\Lambda
_{2}\left( \Gamma _{R}\right) <K_{2}$, with $K_{2}=\dfrac{2}{K_{0}^{2}}.$
Then, for each $R$, the cylinder $K_{2}\Gamma _{R}$ contains a free triplet
of integer vectors $\left( \mathbf{h}_{0},\mathbf{h}_{1},\mathbf{h}%
_{2}\right)$. \newline
Let now be $s$ in $S$, verifying $\frac{\left\Vert \mathbf{g}_{\mathtt{III}%
}^{\prime (s)}\right\Vert }{\rho ^{(s)}}<L$. We may suppose $R=\dfrac{\rho
^{(s)}}{K_{2}}$, so that $K_{2}\Gamma _{R}=\text{Disk}^{\prime }\left(
\rho ^{(s)}\right) +\text{Disk}^{\prime \prime }\left( \dfrac{4K_{2}^{3}}{%
\pi \left( \rho ^{(s)}\right) ^{2}}\right) $, which contains a free triplet
of integer vectors $\left( \mathbf{h}_{0},\mathbf{h}_{1},\mathbf{h}%
_{2}\right) $, but the basis of which, $\text{Disk}^{\prime }\left( \rho
^{(s)}\right) $ is contained in the "hexagon" $\mathbf{H}^{\prime (s)}$
generated by our algorithm.  Then, by the Prism Lemma,\newline
$\left\Vert \mathbf{g}_{\mathtt{III}}^{\prime (s)}\right\Vert ^{2}\left\Vert 
\mathbf{g}_{2}^{\prime \prime \left( s\right) }\right\Vert \leq L^{2}\left(
\rho ^{(s)}\right) ^{2}\left\Vert \mathbf{g}_{2}^{\prime \prime \left(
s\right) }\right\Vert \ \leq L^{2}\left( \rho ^{(s)}\right) ^{2}\left( 
\underset{i=0,1,2}{\max }\left\Vert \mathbf{h}_{i}^{\prime \prime \left(
s\right) }\right\Vert \right) \leq$ \newline
$ \dfrac{4K_{2}^{3}L^{2}}{\pi } $
and the main statement of \textbf{b) }is proved. The second statement
follows from this very result and from $\underset{s\rightarrow +\infty ,%
\text{ }s\in S}{\lim }\left\Vert \mathbf{g}_{\mathtt{III}}^{\prime
(s)}\right\Vert =+\infty $ \ from Lemma 2 at the beginning of Section 2.

Let's now verify the \textbf{c)} Property. We've just proved that under the Geometrical Theorem, for some $M$, $\underset{s\in S}{\sup }\left( \left\Vert \mathbf{g}_{\mathtt{III}}^{\prime
(s)}\right\Vert ^{2}\left\Vert \mathbf{g}^{\prime \prime
}{}_{2}^{(s)}\right\Vert \right) \leq M$ \ holds$.$\newline
Now, if $\varepsilon _{s}=\det \left( \mathbf{B}^{(s)}\right) =\det \left( 
\mathbf{G}^{(s)}\right) =\pm 1$, and if $\left( i,j,k\right) $ is a direct circular permutation of $\left( 0,1,2\right) $, forgetting the indices, we have:\newline
$\mathbf{b}_{i}=\varepsilon \left( \mathbf{g}_{j}\wedge \mathbf{g}%
_{k}\right) $; $\mathbf{b}_{i}^{\prime }=\varepsilon \left( \mathbf{g}%
_{j}^{\prime \prime }\wedge \mathbf{g}_{k}^{\prime }+\mathbf{g}_{j}^{\prime
}\wedge \mathbf{g}_{k}^{\prime \prime }\right) $;$\ \mathbf{b}_{i}^{\prime
\prime }=\varepsilon \left( \mathbf{g}_{j}^{\prime }\wedge \mathbf{g}%
_{k}^{\prime }\right) $.

Then for each $i$, $\left\Vert \mathbf{b}_{i}^{\prime }\right\Vert \leq
2\left\Vert \mathbf{g}_{\mathtt{III}}^{\prime }\right\Vert \left\Vert 
\mathbf{g}^{\prime \prime }{}_{2}\right\Vert $, and $\left\Vert \mathbf{b}%
_{i}^{\prime \prime }\right\Vert \leq \left\Vert \mathbf{g}_{\mathtt{III}%
}^{\prime }\right\Vert ^{2}$. Then

$\left( \underset{i=0,1,2}{\max }\left\Vert \mathbf{b}_{i}^{\prime
(s)}\right\Vert \right) ^{2}\left( \underset{i=0,1,2}{\max }\left\Vert 
\mathbf{b}^{\prime \prime }{}_{i}^{(s)}\right\Vert \right) \leq 4\left(
\left\Vert \mathbf{g}_{\mathtt{III}}^{\prime (s)}\right\Vert ^{2}\left\Vert 
\mathbf{g}_{2}^{\prime \prime }{}^{(s)}\right\Vert \right) ^{2}\leq 4M^{2}$,
and the main conclusion of the Part \textbf{c)} is proved. Furthermore, we
have seen that for each $s$ and each $i$, $\left\Vert \mathbf{b}_{i}^{\prime
(s)}\right\Vert \leq 2\left\Vert \mathbf{g}_{\mathtt{III}}^{\prime
(s)}\right\Vert \left\Vert \mathbf{g}_{2}^{\prime \prime (s)}\right\Vert $;
in addition:

$\left\Vert \mathbf{g}_{\mathtt{III}}^{\prime (s)}\right\Vert ^{2}\left\Vert 
\mathbf{g}_{2}^{\prime \prime (s)}\right\Vert \leq M$ holds. Then, for each $%
s\in S$, $\left\Vert \mathbf{b}_{i}^{\prime (s)}\right\Vert \leq \frac{2M}{%
\left\Vert \mathbf{g}_{\mathtt{III}}^{\prime (s)}\right\Vert }$.

But, by the Lemma 2 of \S 2.1, $\underset{s\rightarrow +\infty }{\lim }%
\left\Vert \mathbf{g}_{\mathtt{III}}^{\prime (s)}\right\Vert =+\infty $.

Then $\underset{s\rightarrow +\infty ,s\in S}{\lim }\left( \underset{i=0,1,2}%
{\max }\left\Vert \mathbf{b}_{i}^{\prime (s)}\right\Vert \right) =0$ and the
second part of \textbf{c)} is proved.
\end{proof}

\section{Demonstration of the Geometrical Theorem of \S 1.3.}

The demonstration of the Geometrical Theorem, Theorem 3. in Subsection 1.3.,
involves only very elementary geometry, but is a little long. In order to
prove it, we first need some auxiliary sets and definitions.

\subsection{The area $A^{(s)}$ and the set $T$ (advances of $\left\Vert 
\mathbf{g}_{\mathtt{III}}^{\prime (s)}\right\Vert $)}

\begin{notation}
* We'll denote by $A^{(s)}$ \emph{twice }the area of the triangle 
\newline
$\left( 
\mathbf{g}_{0}^{\prime (s)}\mathbf{g}_{1}^{\prime (s)}\mathbf{g}_{2}^{\prime
(s)}\right) $.\newline
* We'll denote by $\left( \mathbf{g}_{\mathtt{I}}^{\prime (s)},\mathbf{g}_{%
\mathtt{II}}^{\prime (s)},\mathbf{g}_{\mathtt{III}}^{\prime (s)}\right) $
the permutation of $\left( \mathbf{g}_{0}^{\prime (s)},\mathbf{g}%
_{1}^{\prime (s)},\mathbf{g}_{2}^{\prime (s)}\right) $ such that $\left\Vert 
\mathbf{g}_{\mathtt{I}}^{\prime (s)}\right\Vert \leq \left\Vert \mathbf{g}_{%
\mathtt{II}}^{\prime (s)}\right\Vert \leq \left\Vert \mathbf{g}_{\mathtt{III}%
}^{\prime (s)}\right\Vert $.\newline
*In the Smallest Vector Algorithm, let's denote by $T$ the set of all
integers $s$ such that $\left\Vert \mathbf{g}_{\mathtt{III}}^{\prime
(s)}\right\Vert <\left\Vert \mathbf{g}_{\mathtt{III}}^{\prime
(s+1)}\right\Vert $.\newline
* The set $T$ is infinite, because by Lemma 2, we have: $\underset{%
s\rightarrow +\infty }{\lim }\left\Vert \mathbf{g}_{\mathtt{III}}^{\prime
(s)}\right\Vert =+\infty $.
\end{notation}

\begin{remark}
\textbf{a)} $A^{(s)}=\left\Vert \mathbf{g}_{\mathtt{I}}^{\prime (s)}\wedge 
\mathbf{g}_{\mathtt{II}}^{\prime (s)}+\mathbf{g}_{\mathtt{II}}^{\prime
(s)}\wedge \mathbf{g}_{\mathtt{III}}^{\prime (s)}+\mathbf{g}_{\mathtt{III}%
}^{\prime (s)}\wedge \mathbf{g}_{\mathtt{I}}^{\prime (s)}\right\Vert $;

\textbf{b)} It's easy to establish that for some $i$ and $j$ among $\left\{
0,1,2\right\} $, we have $A^{(s+1)}=A^{(s)}+\left\Vert \mathbf{g}%
_{i}^{\prime (s)}\wedge \mathbf{g}_{j}^{\prime (s)}\right\Vert $; \textbf{c) 
}Then the sequence $\left( A^{(s)}\right) _{s\in \mathbb{N}}$ is increasing.
\end{remark}

\begin{notation}
In this paper, for two non null vectors $\mathbf{a}$ and $\mathbf{b}$ in $%
\mathbb{R}
^{3}$, we shall consider the angle of these two vectors corresponding to the
canonical euclidean norm, and the measure of this angle which belongs to $%
\left] -\pi ;\pi \right] $. This measure will be denoted either by "$%
\measuredangle \left( \mathbf{a},\mathbf{b}\right) "$, or, more simply, when
no ambiguity can occur, by "$\left( \mathbf{a},\mathbf{b}\right) "$.
\end{notation}

\begin{lemma}[Geometry on $T$]
In the Smallest Vector Algorithm, for any $s\in T$, $\left\Vert \mathbf{g}_{%
\mathtt{II}}^{\prime (s)}\right\Vert >\dfrac{1}{2}\left\Vert \mathbf{g}_{%
\mathtt{III}}^{\prime (s)}.\right\Vert .$ Moreover, for $i,$ $j\in \left\{
0,1,2\right\} $ we have: $\dfrac{\pi }{3}<\left\vert \measuredangle \left( 
\mathbf{g}_{i}^{\prime (s)},\mathbf{g}_{j}^{\prime (s)}\right) \right\vert $.
\end{lemma}

\begin{proof}
Let $i,j$ $\in $ $\left\{ 0,1,2\right\} $, $i\neq j$. We have $\left\Vert 
\mathbf{g}_{\mathtt{III}}^{\prime (s+1)}\right\Vert \leq \left\Vert \mathbf{g%
}_{i}^{\prime (s)}-\mathbf{g}_{j}^{\prime (s)}\right\Vert $ and then $%
\left\Vert \mathbf{g}_{\mathtt{III}}^{\prime (s+1)}\right\Vert ^{2}\leq
\left\Vert \mathbf{g}_{i}^{\prime (s)}\right\Vert ^{2}+\left\Vert \mathbf{g}%
_{j}^{\prime (s)}\right\Vert ^{2}\left( 1-2\cos \left( \left\vert \left( 
\mathbf{g}_{i}^{\prime (s)},\mathbf{g}_{j}^{\prime (s)}\right) \right\vert
\right) \frac{\left\Vert \mathbf{g}_{i}^{\prime (s)}\right\Vert }{\left\Vert 
\mathbf{g}_{j}^{\prime (s)}\right\Vert }\right) $.

Without loss of generality, we may suppose $\left\Vert \mathbf{g}%
_{i}^{\prime (s)}\right\Vert \geq \left\Vert \mathbf{g}_{j}^{\prime
(s)}\right\Vert $. Then, by the cos formula above, if $\left\vert
\measuredangle \left( \mathbf{g}_{i}^{\prime (s)},\mathbf{g}_{j}^{\prime
(s)}\right) \right\vert \leq \frac{\pi }{3}$ would hold, we would have

$\left\Vert \mathbf{g}_{\mathtt{III}}^{\prime (s+1)}\right\Vert ^{2}\leq
\left\Vert \mathbf{g}_{i}^{\prime (s)}\right\Vert ^{2}\leq \left\Vert 
\mathbf{g}_{\mathtt{III}}^{\prime (s)}\right\Vert ^{2}$and $s$ wouldn't be
in $T$.

Then $\left\vert \measuredangle \left( \mathbf{g}_{i}^{\prime (s)},\mathbf{g}%
_{j}^{\prime (s)}\right) \right\vert >\frac{\pi }{3}$ \thinspace holds.

Moreover $\left\Vert \mathbf{g}_{\mathtt{III}}^{\prime (s)}\right\Vert
<\left\Vert \mathbf{g}_{\mathtt{III}}^{\prime (s+1)}\right\Vert \leq
\left\Vert \mathbf{g}_{\mathtt{I}}^{\prime (s)}-\mathbf{g}_{\mathtt{II}%
}^{\prime (s)}\right\Vert \leq 2\left\Vert \mathbf{g}_{\mathtt{II}}^{\prime
(s)}\right\Vert $.

Hence $\left\Vert \mathbf{g}_{\mathtt{II}}^{\prime (s)}\right\Vert >\frac{1}{%
2}\left\Vert \mathbf{g}_{\mathtt{III}}^{\prime (s)}\right\Vert $.
\end{proof}

\subsection{The notion of "Needling"}

As A.J. Brentjes has already noted in \cite{Brentj}, if we want our vectors
to have good approximation qualities, their projections $\mathbf{g}%
_{0}^{\prime }$, $\mathbf{g}_{1}^{\prime }$, $\mathbf{g}_{2}^{\prime }$ on $%
\mathbb{P}$ must avoid the \emph{needling}$,$ i.e. flattening phenomenon.

We're going to define and study this phenomenon, but first we need the
following Lemma in elementary geometry. It is very easy and its proof will
be omitted here.

\begin{lemma}
Let $\mathbf{a}$ and $\mathbf{b}$ be two vectors of the plane, with $%
0<\left\Vert \mathbf{b}\right\Vert \leq \left\Vert \mathbf{a}\right\Vert $.
Let's suppose that there exist real numbers $\varepsilon >0$ and $M>0$ such
that:\newline
$\varepsilon \leq \left\vert \measuredangle \left( \mathbf{a},\mathbf{b}%
\right) \right\vert \leq \pi -\varepsilon $ \ and $\frac{\left\Vert \mathbf{b%
}\right\Vert }{\left\Vert \mathbf{a}\right\Vert }\geq M$. Let $\rho $ be the
radius of the greatest disk centered at $\mathbf{0}$ and included in the
parallelogram $\left( \mathbf{a,b,}\left( -\mathbf{a}\right) ,\left( -%
\mathbf{b}\right) \right)$. Then there exists a real number $M^{\prime }>0$%
, depending only on $\varepsilon $ and $M$, such that $\dfrac{\rho }{%
\left\Vert \mathbf{a}\right\Vert }\geq M^{\prime }$.
\end{lemma}

As an immediate corollary, we have:

\begin{lemma}
\emph{(Needling Sequence of Parallelograms)}\textbf{\ }\newline
Let $\left( \mathbf{a}^{\left( s\right) }\mathbf{,}\text{ }\mathbf{b}%
^{\left( s\right) }\mathbf{,}\text{ }\left( -\mathbf{a}^{\left( s\right)
}\right) ,\text{ }\left( -\mathbf{b}^{\left( s\right) }\right) \right) $ be
a sequence of parallelograms, with:\newline
$0<\left\Vert \mathbf{b}^{\left( s\right) }\right\Vert \leq \left\Vert 
\mathbf{a}^{\left( s\right) }\right\Vert $. If these parallelograms are
needling, id est if \newline
 $\underset{s\rightarrow +\infty }{\lim }\dfrac{\rho
^{(s)}}{\left\Vert \mathbf{a}^{\left( s\right) }\right\Vert }=0$, where $%
\rho $ is defined like in the preceding Lemma, then $\underset{s\rightarrow
+\infty }{\lim }\sin \left( \mathbf{a}^{\left( s\right) }\mathbf{,b}^{\left(
s\right) }\right) \times \frac{\left\Vert \mathbf{b}^{\left( s\right)
}\right\Vert }{\left\Vert \mathbf{a}^{\left( s\right) }\right\Vert }=0$.
\end{lemma}

\begin{proof}
If the conclusion were false, then, for some  $\eta >0$ and for
any $s$ in some infinite set $U$, we should have: $\sin \left( \mathbf{a}%
^{\left( s\right) }\mathbf{,b}^{\left( s\right) }\right) \geq \eta $ and $%
\frac{\left\Vert \mathbf{b}^{\left( s\right) }\right\Vert }{\left\Vert 
\mathbf{a}^{\left( s\right) }\right\Vert }\geq \eta $, and then, by the
preceding Lemma, $\frac{\rho ^{(s)}}{\left\Vert \mathbf{a}^{\left( s\right)
}\right\Vert }\geq M^{\prime }$ for some $M^{\prime }$ and for $s\in U.$ But
this negates the hypothesis of our Lemma. Then the conclusion is true.
\end{proof}

This last Lemma leads to the more important Lemma, which describes the
needling phenomenon on the set $T$.

\begin{lemma}
\emph{(Needling Triangles)} The three following assumptions are
logically equivalent: \ \textbf{a )}$\underset{s\rightarrow +\infty ,s\in T}{%
\lim }\dfrac{A^{(s)}}{\left\Vert \mathbf{g}_{\mathtt{III}}^{\prime
(s)}\right\Vert ^{2}}=0$;\newline
\textbf{b )}$\underset{s\rightarrow +\infty ,s\in T}{\lim }\dfrac{\rho ^{(s)}%
}{\left\Vert \mathbf{g}_{\mathtt{III}}^{\prime (s)}\right\Vert }=0$, where $%
\rho ^{(s)}$ is the radius of the greatest disk centered at $\mathbf{0}$ and
included in the convex hull:\newline
 \ $\mathbf{H}^{\prime (s)}$\textbf{\ }$=\textsc{conv}\left( \mathbf{g}_{0}^{\prime },\mathbf{g}_{1}^{\prime },\mathbf{g}%
_{2}^{\prime },-\mathbf{g}_{0}^{\prime },-\mathbf{g}_{1}^{\prime },-\mathbf{g%
}_{2}^{\prime }\right) $ in $\mathbb{P}$;\newline
\textbf{c )}$\underset{s\rightarrow +\infty ,s\in T}{\lim }\left( \frac{%
\left\Vert \mathbf{g}_{\mathtt{I}}^{\prime (s)}\right\Vert }{\left\Vert 
\mathbf{g}_{\mathtt{II}}^{\prime (s)}\right\Vert }+\left( \pi -\left\vert
\measuredangle \left( \mathbf{g}_{\mathtt{II}}^{\prime (s)},\mathbf{g}_{%
\mathtt{III}}^{\prime (s)}\right) \right\vert \right) \right) =0$.
\end{lemma}

\begin{proof}

First, let's prove a)$\Rightarrow $b). Omitting the indices $^{\left(
s\right) },$ we have:

$\mathbf{H}^{\prime }$\textbf{\ }$=\textsc{conv}\left( \mathbf{g}_{0}^{\prime
},\mathbf{g}_{1}^{\prime },\mathbf{g}_{2}^{\prime },-\mathbf{g}_{0}^{\prime
},-\mathbf{g}_{1}^{\prime }-,\mathbf{g}_{2}^{\prime }\right) $ in $\mathbb{P}
$.

Let now $\mathbf{K}^{\prime }$ be\textbf{\ } the convex hull of all the
points $2\mathbf{g}_{i}^{\prime }-\mathbf{g}_{j}^{\prime }$, with $i\neq j$
and $\left\{ i,j\right\} \subset \left\{ 0,1,2\right\} .$ Each $\mathbf{g}%
_{i}^{\prime }$ belongs to $\mathbf{K}^{\prime }$, because

$\mathbf{g}%
_{i}^{\prime }=\frac{2}{3}\left( 2\mathbf{g}_{i}^{\prime }-\mathbf{g}%
_{j}^{\prime }\right) +\frac{1}{3}\left( 2\mathbf{g}_{j}^{\prime }-\mathbf{g}%
_{i}^{\prime }\right) $.

In addition, the projection of the cofactors relation on $\mathbb{P}$ leads
to:

$\left\Vert \mathbf{X}\right\Vert \left\Vert \mathbf{b}_{0}^{\prime \prime
}{}\right\Vert \cdot \mathbf{g}_{0}^{\prime }+\left\Vert \mathbf{X}%
\right\Vert \left\Vert \mathbf{b}_{1}^{\prime \prime }{}\right\Vert \cdot 
\mathbf{g}_{1}^{\prime }+\left\Vert \mathbf{X}\right\Vert \left\Vert \mathbf{%
b}_{2}^{\prime \prime }\right\Vert \cdot \mathbf{g}_{2}^{\prime }=\mathbf{0.}
$

Then, $\mathbf{0}$ is in the triangle $\mathbf{g}_{0}^{\prime }\mathbf{g}%
_{1}^{\prime }\mathbf{g}_{2}^{\prime }$.\ Let $F_{0}$ be the homothety with
center $\mathbf{g}_{0}^{\prime }$ and with scaling $2$.

Then $F_{0}\left( \mathbf{0}\right) =-\mathbf{g}_{0}^{\prime }$ is inside $%
F\left( \mathbf{g}_{0}^{\prime }\mathbf{g}_{1}^{\prime }\mathbf{g}%
_{2}^{\prime }\right) $, which is the triangle with summits $\mathbf{g}%
_{2}^{\prime }$; $2\mathbf{g}_{1}^{\prime }-\mathbf{g}_{0}^{\prime }$; $2%
\mathbf{g}_{2}^{\prime }-\mathbf{g}_{0}^{\prime }.$ Then $\left( -\mathbf{g}%
_{0}^{\prime }\right) $, and, in the same way, $\left( -\mathbf{g}%
_{1}^{\prime }\right) $ and $\left( -\mathbf{g}_{2}^{\prime }\right) $,
belong to $\mathbf{K}^{\prime }$. Then $\mathbf{H}^{\prime }\subset \mathbf{K%
}^{\prime }$; then $\pi \left( \rho ^{\left( s\right) }\right) ^{2}\leq 
\text{area}\left( \mathbf{K}^{\prime \left( s\right) }\right) $. But $%
\mathbf{K}^{\prime \left( s\right) }$ is formed with $13$ triangles, each of
them isometric to the triangle $\mathbf{g}_{0}^{\prime }\mathbf{g}%
_{1}^{\prime }\mathbf{g}_{2}^{\prime }$.  Then: $\pi \left( \rho ^{\left( s\right) }\right) ^{2}\leq \text{area}\left( 
\mathbf{K}^{\prime \left( s\right) }\right) \leq 13A^{(s)}$.

Then, using a), we have $\pi \left( \rho ^{\left( s\right) }\right) ^{2}\leq
\left\Vert \mathbf{g}_{\mathtt{III}}^{\prime (s)}\right\Vert ^{2}\varepsilon
\left( s\right) $, with $\underset{x\rightarrow +\infty ,x\in T}{\lim }%
\varepsilon \left( s\right) =0$. This implies b). Second, let's prove b)$\Rightarrow $c).

If b) is true, we also have $\underset{s\rightarrow +\infty ,s\in T}{\lim }%
\frac{\rho ^{(s)}}{\left\Vert \mathbf{g}_{\mathtt{III}}^{\prime
(s)}\right\Vert }=0$, if $\rho ^{(s)}$ is the radius of the greatest disk
centered at $\mathbf{0}$ and included in the \textbf{parallelogram }with
summits $\mathbf{g}_{\text{II}}^{\prime }$, $\mathbf{g}_{\text{III}}^{\prime
},$ $-\mathbf{g}_{\text{II}}^{\prime }$, $-\mathbf{g}_{\text{III}}^{\prime }$%
. Then, by the last Lemma: \newline
$\underset{s\rightarrow +\infty ,s\in T}{\lim }%
\sin \left( \mathbf{g}_{\text{II}}^{\prime (s)},\mathbf{g}_{\mathtt{III}%
}^{\prime (s)}\right) \times \frac{\left\Vert \mathbf{g}_{\text{II}}^{\prime
(s)}\right\Vert }{\left\Vert \mathbf{g}_{\text{III}}^{\prime (s)}\right\Vert 
}=0$. But, by the Lemma "Geometry on $T$" of the last subsection above, $%
\frac{\left\Vert \mathbf{g}_{\mathtt{II}}^{\prime (s)}\right\Vert }{%
\left\Vert \mathbf{g}_{\mathtt{III}}^{\prime (s)}\right\Vert }\geq \frac{1}{2%
}$ holds. \newline
 Then $\underset{s\rightarrow +\infty ,s\in T}{\lim }\sin \left( 
\mathbf{g}_{\text{II}}^{\prime (s)},\mathbf{g}_{\mathtt{III}}^{\prime
(s)}\right) =0$. By the same Lemma: \newline
 $\left\vert \measuredangle \left( 
\mathbf{g}_{\text{II}}^{\prime (s)},\mathbf{g}_{\mathtt{III}}^{\prime
(s)}\right) \right\vert >\frac{\pi }{3}.$ Then:$\underset{s\rightarrow
+\infty ,s\in T}{\lim }\left( \pi -\left\vert \measuredangle \left( \mathbf{g%
}_{\text{II}}^{\prime (s)},\mathbf{g}_{\mathtt{III}}^{\prime (s)}\right)
\right\vert \right) =0$.

This last result, with the help of $\left\vert \measuredangle \left( \mathbf{%
g}_{\text{I}}^{\prime (s)},\mathbf{g}_{\mathtt{III}}^{\prime (s)}\right)
\right\vert >\frac{\pi }{3},$ by the same Lemma, leads to: $\underset{%
s\rightarrow +\infty ,s\in T}{\lim \sup }\left\vert \measuredangle \left( 
\mathbf{g}_{\text{I}}^{\prime (s)},\mathbf{g}_{\mathtt{II}}^{\prime
(s)}\right) \right\vert <\frac{2\pi }{3}$. We also have by the same Lemma, $%
\left\vert \measuredangle \left( \mathbf{g}_{\text{I}}^{\prime (s)},\mathbf{g%
}_{\mathtt{II}}^{\prime (s)}\right) \right\vert >\frac{\pi }{3}$. Then: $%
\underset{s\rightarrow +\infty ,s\in T}{\lim \inf }\left\vert \sin \left( 
\mathbf{g}_{\text{I}}^{\prime (s)},\mathbf{g}_{\mathtt{II}}^{\prime
(s)}\right) \right\vert >0$.

But, if b) is true, $\underset{s\rightarrow +\infty ,s\in T}{\lim }\frac{%
\rho ^{(s)}}{\left\Vert \mathbf{g}_{\mathtt{III}}^{\prime (s)}\right\Vert }%
=0 $ holds, if $\rho ^{(s)}$ is the radius of the greatest disk centered at $%
\mathbf{0}$ and included in the other parallelogram, with summits $\mathbf{g}%
_{\text{I}}^{\prime }$, $\mathbf{g}_{\text{III}}^{\prime },$ $-\mathbf{g}_{%
\text{I}}^{\prime }$, $-\mathbf{g}_{\text{III}}^{\prime }$, and then, by $%
\frac{\left\Vert \mathbf{g}_{\mathtt{II}}^{\prime (s)}\right\Vert }{%
\left\Vert \mathbf{g}_{\mathtt{III}}^{\prime (s)}\right\Vert }\geq \frac{1}{2%
}$, we also have: $\underset{s\rightarrow +\infty ,s\in T}{\lim }\frac{\rho
^{(s)}}{\left\Vert \mathbf{g}_{\mathtt{II}}^{\prime (s)}\right\Vert }=0$.
Then, by the last Lemma on the needling parallelograms: $\underset{s\rightarrow +\infty ,s\in T}{\lim }\sin \left( \mathbf{g}_{\text{%
I}}^{\prime (s)},\mathbf{g}_{\mathtt{II}}^{\prime (s)}\right) \times \frac{%
\left\Vert \mathbf{g}_{\text{I}}^{\prime (s)}\right\Vert }{\left\Vert 
\mathbf{g}_{\text{II}}^{\prime (s)}\right\Vert }=0$, and using the $\lim
\inf $ above, we obtain: $\underset{s\rightarrow +\infty ,s\in T}{\lim }%
\frac{\left\Vert \mathbf{g}_{\mathtt{I}}^{\prime (s)}\right\Vert }{%
\left\Vert \mathbf{g}_{\mathtt{II}}^{\prime (s)}\right\Vert }=0$. The proof
of the part b)$\Rightarrow $c$)$ is done.\newline
Finally, the implication c)$%
\Rightarrow $a$)$ is obvious.
\end{proof}

In order to refute the needling phenomenon, we're going to study what
happens when the projections on $\mathbb{P}$\ are "almost flat". This will
allow us to prove that the needling CANNOT happen.

\subsection{Almost Flat Triangles. Set $T^{\ast }$ of indices.}

\begin{notation}
$T^{\ast }$ will denote the set of all integers $s\in T$
such that the
triangle $\mathbf{g}_{0}^{\prime (s)}\mathbf{g}_{1}^{\prime (s)}\mathbf{g}%
_{2}^{\prime (s)}$is "almost flat", namely such that:
$\frac{\left\Vert \mathbf{g}_{\mathtt{I}}^{\prime (s)}\right\Vert }{%
\left\Vert \mathbf{g}_{\mathtt{II}}^{\prime (s)}\right\Vert }\leq 0.1$ and $%
\left\vert \measuredangle \left( \mathbf{g}_{\mathtt{III}}^{\prime (s)},%
\mathbf{g}_{\mathtt{II}}^{\prime (s)}\right) \right\vert \geq \dfrac{30\pi }{%
31}$.
\end{notation}

\begin{lemma}[Flat Triangle Lemma]
If $s\in T^{\ast }$ (i.e. if the triangle \newline
$\mathbf{g}_{0}^{\prime (s)}%
\mathbf{g}_{1}^{\prime (s)}\mathbf{g}_{2}^{\prime (s)}$ is "almost flat"),
then we also have the following relations:\newline
$\dfrac{15\pi }{31}\leq \left\vert \measuredangle \left( \mathbf{g}_{\mathtt{%
III}}^{\prime (s)},\mathbf{g}_{\mathtt{I}}^{\prime (s)}\right) \right\vert
\leq \dfrac{17\pi }{31}$ and $\dfrac{15\pi }{31}\leq \left\vert
\measuredangle \left( \mathbf{g}_{\mathtt{II}}^{\prime (s)},\mathbf{g}_{%
\mathtt{I}}^{\prime (s)}\right) \right\vert \leq \dfrac{17\pi }{31}$,  \newline
and also:
$\frac{\left\Vert \mathbf{g}_{\mathtt{II}}^{\prime (s)}\right\Vert }{%
\left\Vert \mathbf{g}_{\mathtt{III}}^{\prime (s)}\right\Vert }\geq 0.979$, $%
\left\Vert \mathbf{g}_{\mathtt{II}}^{\prime (s)}+\mathbf{g}_{\mathtt{III}%
}^{\prime (s)}\right\Vert \leq 0.23\left\Vert \mathbf{g}_{\mathtt{III}%
}^{\prime (s)}\right\Vert $ and:\newline
$\left\Vert \mathbf{g}_{\mathtt{I}}^{\prime (s)}-\mathbf{g}_{\mathtt{II}%
}^{\prime (s)}-\mathbf{g}_{\mathtt{III}}^{\prime (s)}\right\Vert \leq
0.33\left\Vert \mathbf{g}_{\mathtt{III}}^{\prime (s)}\right\Vert $.
\end{lemma}

\begin{proof}
Because in $T$, $\left\Vert \mathbf{g}_{i}^{\prime (s)}-\mathbf{g}%
_{j}^{\prime (s)}\right\Vert \geq \left\Vert \mathbf{g}_{\mathtt{III}%
}^{\prime (s+1)}\right\Vert \geq \left\Vert \mathbf{g}_{\mathtt{III}%
}^{\prime (s)}\right\Vert $ holds, both following inequalities also hold:

$\left\Vert \mathbf{g}_{\mathtt{III}}^{\prime (s)}-\mathbf{g}_{\mathtt{I}%
}^{\prime (s)}\right\Vert \geq \left\Vert \mathbf{g}_{\mathtt{III}}^{\prime
(s)}\right\Vert $ and $\left\Vert \mathbf{g}_{\mathtt{II}}^{\prime (s)}-%
\mathbf{g}_{\mathtt{I}}^{\prime (s)}\right\Vert \geq \left\Vert \mathbf{g}_{%
\mathtt{III}}^{\prime (s)}\right\Vert $ $\geq \left\Vert \mathbf{g}_{\mathtt{%
II}}^{\prime (s)}\right\Vert $.

The first inequality leads to:
\begin{equation*}
\left\Vert \mathbf{g}_{\mathtt{III}}^{\prime (s)}\right\Vert ^{2}\leq
\left\Vert \mathbf{g}_{\mathtt{III}}^{\prime (s)}\right\Vert ^{2}+\left\Vert 
\mathbf{g}_{\mathtt{I}}^{\prime (s)}\right\Vert \left\Vert \mathbf{g}_{%
\mathtt{III}}^{\prime (s)}\right\Vert \left( \frac{\left\Vert \mathbf{g}_{%
\mathtt{I}}^{\prime (s)}\right\Vert }{\left\Vert \mathbf{g}_{\mathtt{III}%
}^{\prime (s)}\right\Vert }-2\cos \left( \mathbf{g}_{\mathtt{III}}^{\prime
(s)},\mathbf{g}_{\mathtt{I}}^{\prime (s)}\right) \right) 
\end{equation*}%

Then: $0.1-2\cos \left( \mathbf{g}_{\mathtt{III}}^{\prime (s)},\mathbf{g}_{%
\mathtt{I}}^{\prime (s)}\right) \geq 0$; then: $\cos \left( \mathbf{g}_{%
\mathtt{III}}^{\prime (s)},\mathbf{g}_{\mathtt{I}}^{\prime (s)}\right) \leq
0.05$.

Hence: $\left\vert \measuredangle \left( \mathbf{g}_{\mathtt{III}}^{\prime
(s)},\mathbf{g}_{\mathtt{I}}^{\prime (s)}\right) \right\vert \geq \frac{%
15\pi }{31}$. Hence $\left\vert \measuredangle \left( \mathbf{g}_{\mathtt{II}%
}^{\prime (s)},\mathbf{g}_{\mathtt{I}}^{\prime (s)}\right) \right\vert \leq 
\frac{17\pi }{31}$.

In the same way, we obtain:

$\left\vert \measuredangle \left( \mathbf{g}_{\mathtt{II}}^{\prime (s)},%
\mathbf{g}_{\mathtt{I}}^{\prime (s)}\right) \right\vert \geq \frac{15\pi }{31%
}$ and $\left\vert \measuredangle \left( \mathbf{g}_{\mathtt{III}}^{\prime
(s)},\mathbf{g}_{\mathtt{I}}^{\prime (s)}\right) \right\vert \leq \frac{%
17\pi }{31}$.

Again, let's use the inequality: $\left\Vert \mathbf{g}_{\mathtt{III}%
}^{\prime (s)}\right\Vert ^{2}\leq \left\Vert \mathbf{g}_{\mathtt{II}%
}^{\prime (s)}-\mathbf{g}_{\mathtt{I}}^{\prime (s)}\right\Vert ^{2}$. Then

$\left\Vert \mathbf{g}_{\mathtt{III}}^{\prime (s)}\right\Vert ^{2}\leq
\left\Vert \mathbf{g}_{\mathtt{II}}^{\prime (s)}\right\Vert ^{2}+\left\Vert 
\mathbf{g}_{\mathtt{I}}^{\prime (s)}\right\Vert ^{2}-2\cos \left( \mathbf{g}%
_{\mathtt{II}}^{\prime (s)},\mathbf{g}_{\mathtt{I}}^{\prime (s)}\right)
\left\Vert \mathbf{g}_{\mathtt{II}}^{\prime (s)}\right\Vert \left\Vert 
\mathbf{g}_{\mathtt{I}}^{\prime (s)}\right\Vert $.

But \ $-2\cos \left( \left( \mathbf{g}_{\mathtt{II}}^{\prime (s)},\mathbf{g}%
_{\mathtt{I}}^{\prime (s)}\right) \right) \leq -2\cos \left( \frac{17\pi }{31%
}\right) \leq 0.31$.

Then, dividing by $\left\Vert \mathbf{g}_{\mathtt{III}}^{\prime
(s)}\right\Vert ^{2}$, we obtain

$1\leq \frac{\left\Vert \mathbf{g}_{\mathtt{II}}^{\prime (s)}\right\Vert ^{2}%
}{\left\Vert \mathbf{g}_{\mathtt{III}}^{\prime (s)}\right\Vert ^{2}}%
+0.01+0.031\frac{\left\Vert \mathbf{g}_{\mathtt{II}}^{\prime (s)}\right\Vert 
}{\left\Vert \mathbf{g}_{\mathtt{III}}^{\prime (s)}\right\Vert }$, i.e.: $%
x^{2}+0.031x-0.99\geq 0$,

with $x=\frac{\left\Vert \mathbf{g}_{\mathtt{II}}^{\prime (s)}\right\Vert }{%
\left\Vert \mathbf{g}_{\mathtt{III}}^{\prime (s)}\right\Vert }$. This
implies $x=\frac{\left\Vert \mathbf{g}_{\mathtt{II}}^{\prime (s)}\right\Vert 
}{\left\Vert \mathbf{g}_{\mathtt{III}}^{\prime (s)}\right\Vert }\geq 0.979.$
We have

$\left\Vert \mathbf{g}_{\mathtt{II}}^{\prime (s)}+\mathbf{g}_{\mathtt{III}%
}^{\prime (s)}\right\Vert ^{2}\leq \left\Vert \mathbf{g}_{\mathtt{II}%
}^{\prime (s)}\right\Vert ^{2}+\left\Vert \mathbf{g}_{\mathtt{III}}^{\prime
(s)}\right\Vert ^{2}+2\cos \left( \frac{30\pi }{31}\right) \left\Vert 
\mathbf{g}_{\mathtt{III}}^{\prime (s)}\right\Vert \left\Vert \mathbf{g}_{%
\mathtt{II}}^{\prime (s)}\right\Vert $, then

$\left\Vert \mathbf{g}_{\mathtt{II}}^{\prime (s)}+\mathbf{g}_{\mathtt{III}%
}^{\prime (s)}\right\Vert ^{2}\leq 2\left\Vert \mathbf{g}_{\mathtt{III}%
}^{\prime (s)}\right\Vert ^{2}-1.9897\times 0.979\left\Vert \mathbf{g}_{%
\mathtt{III}}^{\prime (s)}\right\Vert ^{2},$ then \newline
$\left\Vert \mathbf{g}_{%
\mathtt{II}}^{\prime (s)}+\mathbf{g}_{\mathtt{III}}^{\prime (s)}\right\Vert
\leq \sqrt{0.0521}\left\Vert \mathbf{g}_{\mathtt{III}}^{\prime
(s)}\right\Vert \leq 0.23\left\Vert \mathbf{g}_{\mathtt{III}}^{\prime
(s)}\right\Vert $.\newline
Hence the last conclusion is obtained.
\end{proof}

\subsection{From a Lemma to the Geometrical Theorem (Theorem 3)}

\begin{notation}
The sequence $\left( \alpha ^{\left( s\right) }\right) $ is defined by: $%
\alpha ^{\left( s\right) }=\frac{A^{(s)}}{\left\Vert \mathbf{g}_{\mathtt{III}%
}^{\prime (s)}\right\Vert ^{2}}.$ \newline
If $\underset{s\rightarrow +\infty }{\lim
\sup ~}\alpha ^{\left( s\right) }>0$, then the algorithm is balanced$,$ i.e.
the triangles on $\mathbb{P}$ do not \emph{needle}.
\end{notation}

\begin{lemma}
\textbf{(Monotonic Subsequence Lemma) }Let $\left[ m;+\infty \right[ $ be an
interval of $\mathbb{N}$ such that $\left[ m;+\infty \right[ \cap T\subset
T^{\ast }$, which means that every \emph{advancing }triangle with its range
in $\left[ m;+\infty \right[ $ is \emph{almost flat}. Then the sequence $%
\left( \alpha ^{\left( s\right) }\right) $ is\emph{\ increasing }on $\left[
m;+\infty \right[ \cap T$, \ which means that for any $s,t\in \left[
m;+\infty \right[ \cap T$ \ with $s\leq t$, $\frac{A^{(s)}}{\left\Vert 
\mathbf{g}_{\mathtt{III}}^{\prime (s)}\right\Vert ^{2}}\leq \frac{A^{(t)}}{%
\left\Vert \mathbf{g}_{\mathtt{III}}^{\prime (t)}\right\Vert ^{2}}$ holds.
\end{lemma}

Let's admit this Lemma, which is proved in the next Subsection. Then we can
demonstrate the Geometrical Theorem, Theorem 3.

\begin{proof}[Proof of Theorem 3]
Let's suppose that the conclusion of the geometrical Lemma is FALSE, namely
that $\underset{s\rightarrow +\infty ,s\in T}{\lim }\frac{\rho ^{(s)}}{%
\left\Vert \mathbf{g}_{\mathtt{III}}^{\prime (s)}\right\Vert }=0$. Then, by
the Lemma on needling triangles of the second subsection above, $\underset{%
s\rightarrow +\infty ,s\in T}{\lim }\frac{A^{(s)}}{\left\Vert \mathbf{g}_{%
\mathtt{III}}^{\prime (s)}\right\Vert ^{2}}=0$. By the same Lemma, we have \newline $%
\underset{s\rightarrow +\infty ,s\in T}{\lim }\left( \frac{\left\Vert 
\mathbf{g}_{\mathtt{I}}^{\prime (s)}\right\Vert }{\left\Vert \mathbf{g}_{%
\mathtt{II}}^{\prime (s)}\right\Vert }+\left( \pi -\left\vert \measuredangle
\left( \mathbf{g}_{\mathtt{II}}^{\prime (s)},\mathbf{g}_{\mathtt{III}%
}^{\prime (s)}\right) \right\vert \right) \right) =0$. Then, there exists an
integer $m$ such that $\ \left[ m;+\infty \right[ \cap T\subset T^{\ast }$,
which means that with a range great enough, any advancing triangle is almost
flat. Then, by the Monotonic Subsequence Lemma, the sequence of the $\alpha
^{\left( s\right) }=\frac{A^{(s)}}{\left\Vert \mathbf{g}_{\mathtt{III}%
}^{\prime (s)}\right\Vert ^{2}}$ with $s\geq m$ is \emph{increasing }on $T$,
which is infinite. This is contradictory with $\underset{s\rightarrow
+\infty ,s\in T}{\lim }\frac{A^{(s)}}{\left\Vert \mathbf{g}_{\mathtt{III}%
}^{\prime (s)}\right\Vert ^{2}}=0$. Then, we have $\underset{s\rightarrow
+\infty ,s\in T}{\lim \sup }\frac{\rho ^{(s)}}{\left\Vert \mathbf{g}_{%
\mathtt{III}}^{\prime (s)}\right\Vert }>0$\ and the \textbf{Geometrical
Theorem is proved.}
\end{proof}

This Monotonic Subsequence Lemma has now to be proved.

\subsection{Proof of the Monotonic Subsequence Lemma}

Let $s$ be an integer in the interval $\left[ m;+\infty \right[ \cap
T\subset T^{\ast }$ as in the Hypothesis. Let's denote $s^{\prime }$ the
successor of $s$ in $T$, i.e. the smallest integer $t$ in $T$ such that $s<t$

In order to establish our Lemma, it suffices to show that 
\begin{equation}
\dfrac{A^{(s)}}{\left\Vert \mathbf{g}_{\mathtt{III}}^{\prime (s)}\right\Vert
^{2}}\leq \dfrac{A^{(s^{\prime })}}{\left\Vert \mathbf{g}_{\mathtt{III}%
}^{\prime (s^{\prime })}\right\Vert ^{2}}.  \label{Ineq MSL}
\end{equation}

We have four cases:

Either $\mathbf{g}_{\mathtt{III}}^{\prime (s+1)}=\mathbf{g}_{\mathtt{I}%
}^{\prime (s)}-\mathbf{g}_{\mathtt{II}}^{\prime (s)}$ (Case $\left( \mathtt{I%
}\ \smallsetminus \ \mathtt{II}\right) $), or $\mathbf{g}_{\mathtt{III}%
}^{\prime (s+1)}=\mathbf{g}_{\mathtt{I}}^{\prime (s)}-\mathbf{g}_{\mathtt{III%
}}^{\prime (s)}$ $\left( \mathtt{I}\ \smallsetminus \ \mathtt{III}\right) $,
or $\mathbf{g}_{\mathtt{III}}^{\prime (s+1)}=\mathbf{g}_{\mathtt{II}%
}^{\prime (s)}-\mathbf{g}_{\mathtt{I}}^{\prime (s)}$ $\left( \mathtt{II}\
\smallsetminus \ \mathtt{I}\right) $, or $\mathbf{g}_{\mathtt{III}}^{\prime
(s+1)}=\mathbf{g}_{\mathtt{III}}^{\prime (s)}-\mathbf{g}_{\mathtt{I}%
}^{\prime (s)}$ $\left( \mathtt{III}\ \smallsetminus \ \mathtt{I}\right) .$

We're going to prove the inequality (Ineq~\ref{Ineq MSL})
only in the case $\left( \mathtt{II}\ \smallsetminus \ \mathtt{I}\right) $.
The demonstration is similar, or easier, in the three other cases.

\textbf{Case }$\left( \mathtt{II}\ \smallsetminus \ \mathtt{I}\right) $%
\textbf{:}\texttt{\ }$\mathbf{g}_{\mathtt{III}}^{\prime (s+1)}=\mathbf{g}_{%
\mathtt{II}}^{\prime (s)}-\mathbf{g}_{\mathtt{I}}^{\prime (s)}$.

We have to obtain first \newline 
$\dfrac{A^{(s)}}{\left\Vert \mathbf{g}_{\mathtt{III}%
}^{\prime (s)}\right\Vert ^{2}}\leq \dfrac{A^{(s+1)}}{\left\Vert \mathbf{g}_{%
\mathtt{III}}^{\prime (s+1)}\right\Vert ^{2}}$. i.e. $\ \dfrac{A^{(s+1)}}{%
A^{(s)}}\geq \dfrac{\left\Vert \mathbf{g}_{\mathtt{II}}^{\prime (s)}-\mathbf{%
g}_{\mathtt{I}}^{\prime (s)}\right\Vert ^{2}}{\left\Vert \mathbf{g}_{\mathtt{%
III}}^{\prime (s)}\right\Vert ^{2}}$.

It's sufficient to show: $\frac{A^{(s+1)}}{A^{(s)}}\geq \left( 1+\frac{%
\left\Vert \mathbf{g}_{\mathtt{I}}^{\prime (s)}\right\Vert }{\left\Vert 
\mathbf{g}_{\mathtt{III}}^{\prime (s)}\right\Vert }\right) ^{2}$.

But: $\frac{A^{(s+1)}}{A^{(s)}}=\frac{A^{(s)}+\left\Vert \mathbf{g}_{%
\mathtt{III}}^{\prime (s)}\wedge \mathbf{g}_{\mathtt{I}}^{\prime
(s)}\right\Vert }{A^{(s)}}=1+\frac{\left\Vert \mathbf{g}_{\mathtt{III}%
}^{\prime (s)}\wedge \mathbf{g}_{\mathtt{I}}^{\prime (s)}\right\Vert }{%
\left\Vert \mathbf{g}_{\mathtt{I}}^{\prime (s)}\wedge \mathbf{g}_{\mathtt{II}%
}^{\prime (s)}+\mathbf{g}_{\mathtt{II}}^{\prime (s)}\wedge \mathbf{g}_{%
\mathtt{III}}^{\prime (s)}+\mathbf{g}_{\mathtt{III}}^{\prime (s)}\wedge 
\mathbf{g}_{\mathtt{I}}^{\prime (s)}\right\Vert }$. It's sufficient to show:

$1+\frac{\left\Vert \mathbf{g}_{\mathtt{III}}^{\prime (s)}\wedge \mathbf{g}_{%
\mathtt{I}}^{\prime (s)}\right\Vert }{\left\Vert \mathbf{g}_{\mathtt{I}%
}^{\prime (s)}\wedge \mathbf{g}_{\mathtt{II}}^{\prime (s)}+\mathbf{g}_{%
\mathtt{II}}^{\prime (s)}\wedge \mathbf{g}_{\mathtt{III}}^{\prime (s)}+%
\mathbf{g}_{\mathtt{III}}^{\prime (s)}\wedge \mathbf{g}_{\mathtt{I}}^{\prime
(s)}\right\Vert }$ $\geq \left( 1+\frac{\left\Vert \mathbf{g}_{\mathtt{I}%
}^{\prime (s)}\right\Vert }{\left\Vert \mathbf{g}_{\mathtt{III}}^{\prime
(s)}\right\Vert }\right) ^{2}$ i.e.:

$\frac{\left\Vert \mathbf{g}_{\mathtt{III}}^{\prime (s)}\wedge \mathbf{g}_{%
\mathtt{I}}^{\prime (s)}\right\Vert }{\left\Vert \mathbf{g}_{\mathtt{I}%
}^{\prime (s)}\wedge \mathbf{g}_{\mathtt{II}}^{\prime (s)}+\mathbf{g}_{%
\mathtt{II}}^{\prime (s)}\wedge \mathbf{g}_{\mathtt{III}}^{\prime (s)}+%
\mathbf{g}_{\mathtt{III}}^{\prime (s)}\wedge \mathbf{g}_{\mathtt{I}}^{\prime
(s)}\right\Vert }\geq \frac{\left\Vert \mathbf{g}_{\mathtt{I}}^{\prime
(s)}\right\Vert }{\left\Vert \mathbf{g}_{\mathtt{III}}^{\prime
(s)}\right\Vert }\left( \frac{\left\Vert \mathbf{g}_{\mathtt{I}}^{\prime
(s)}\right\Vert }{\left\Vert \mathbf{g}_{\mathtt{III}}^{\prime
(s)}\right\Vert }+2\right) $; i.e.:

$1+\frac{\left\Vert \mathbf{g}_{\mathtt{I}}^{\prime (s)}\wedge \mathbf{g}_{%
\mathtt{II}}^{\prime (s)}\right\Vert }{\left\Vert \mathbf{g}_{\mathtt{III}%
}^{\prime (s)}\wedge \mathbf{g}_{\mathtt{I}}^{\prime (s)}\right\Vert }+\frac{%
\left\Vert \mathbf{g}_{\mathtt{II}}^{\prime (s)}\wedge \mathbf{g}_{\mathtt{%
III}}^{\prime (s)}\right\Vert }{\left\Vert \mathbf{g}_{\mathtt{III}}^{\prime
(s)}\wedge \mathbf{g}_{\mathtt{I}}^{\prime (s)}\right\Vert }\leq \frac{%
\left\Vert \mathbf{g}_{\mathtt{III}}^{\prime (s)}\right\Vert }{\left\Vert 
\mathbf{g}_{\mathtt{I}}^{\prime (s)}\right\Vert \left( \frac{\left\Vert 
\mathbf{g}_{\mathtt{I}}^{\prime (s)}\right\Vert }{\left\Vert \mathbf{g}_{%
\mathtt{III}}^{\prime (s)}\right\Vert }+2\right) }$. To obtain that, it's
enough to show: $1+\frac{\left\Vert \mathbf{g}_{\mathtt{I}}^{\prime
(s)}\wedge \mathbf{g}_{\mathtt{II}}^{\prime (s)}\right\Vert }{\left\Vert 
\mathbf{g}_{\mathtt{III}}^{\prime (s)}\wedge \mathbf{g}_{\mathtt{I}}^{\prime
(s)}\right\Vert }+\frac{\left\Vert \mathbf{g}_{\mathtt{II}}^{\prime
(s)}\wedge \mathbf{g}_{\mathtt{III}}^{\prime (s)}\right\Vert }{\left\Vert 
\mathbf{g}_{\mathtt{III}}^{\prime (s)}\wedge \mathbf{g}_{\mathtt{I}}^{\prime
(s)}\right\Vert }\leq \frac{\left\Vert \mathbf{g}_{\mathtt{III}}^{\prime
(s)}\right\Vert }{2.1\left\Vert \mathbf{g}_{\mathtt{I}}^{\prime
(s)}\right\Vert }$, i.e.:%
\begin{equation}
\frac{2.1\left\Vert \mathbf{g}_{\mathtt{I}}^{\prime (s)}\right\Vert }{%
\left\Vert \mathbf{g}_{\mathtt{III}}^{\prime (s)}\right\Vert }+\frac{%
2.1\left\Vert \mathbf{g}_{\mathtt{I}}^{\prime (s)}\right\Vert \left\Vert 
\mathbf{g}_{\mathtt{I}}^{\prime (s)}\wedge \mathbf{g}_{\mathtt{II}}^{\prime
(s)}\right\Vert }{\left\Vert \mathbf{g}_{\mathtt{III}}^{\prime
(s)}\right\Vert \left\Vert \mathbf{g}_{\mathtt{III}}^{\prime (s)}\wedge 
\mathbf{g}_{\mathtt{I}}^{\prime (s)}\right\Vert }+\frac{2.1\left\Vert 
\mathbf{g}_{\mathtt{I}}^{\prime (s)}\right\Vert \left\Vert \mathbf{g}_{%
\mathtt{II}}^{\prime (s)}\wedge \mathbf{g}_{\mathtt{III}}^{\prime
(s)}\right\Vert }{\left\Vert \mathbf{g}_{\mathtt{III}}^{\prime
(s)}\right\Vert \left\Vert \mathbf{g}_{\mathtt{III}}^{\prime (s)}\wedge 
\mathbf{g}_{\mathtt{I}}^{\prime (s)}\right\Vert }\leq 1  \label{Ineg3T}
\end{equation}%
We have: \newline
$\dfrac{2.1\left\Vert \mathbf{g}_{\mathtt{I}}^{\prime
(s)}\right\Vert }{\left\Vert \mathbf{g}_{\mathtt{III}}^{\prime
(s)}\right\Vert }\dfrac{\left\Vert \mathbf{g}_{\mathtt{I}}^{\prime
(s)}\wedge \mathbf{g}_{\mathtt{II}}^{\prime (s)}\right\Vert }{\left\Vert 
\mathbf{g}_{\mathtt{III}}^{\prime (s)}\wedge \mathbf{g}_{\mathtt{I}}^{\prime
(s)}\right\Vert }=\dfrac{2.1\left\Vert \mathbf{g}_{\mathtt{I}}^{\prime
(s)}\right\Vert }{\left\Vert \mathbf{g}_{\mathtt{III}}^{\prime
(s)}\right\Vert }\dfrac{\sin \left( \left\vert \left( \mathbf{g}_{\mathtt{I}%
}^{\prime (s)},\mathbf{g}_{\mathtt{II}}^{\prime (s)}\right) \right\vert
\right) \left\Vert \mathbf{g}_{\mathtt{II}}^{\prime (s)}\right\Vert }{\sin
\left( \left\vert \left( \mathbf{g}_{\mathtt{I}}^{\prime (s)},\mathbf{g}_{%
\mathtt{III}}^{\prime (s)}\right) \right\vert \right) \left\Vert \mathbf{g}_{%
\mathtt{III}}^{\prime (s)}\right\Vert }$

$\leq \dfrac{2.1\times 0.1}{\sin \left( \frac{17\pi }{31}\right) }\leq 0.22$

In addition: $\dfrac{2.1\left\Vert \mathbf{g}_{\mathtt{I}}^{\prime
(s)}\right\Vert }{\left\Vert \mathbf{g}_{\mathtt{III}}^{\prime
(s)}\right\Vert }\cdot \dfrac{\left\Vert \mathbf{g}_{\mathtt{II}}^{\prime
(s)}\wedge \mathbf{g}_{\mathtt{III}}^{\prime (s)}\right\Vert }{\left\Vert 
\mathbf{g}_{\mathtt{III}}^{\prime (s)}\wedge \mathbf{g}_{\mathtt{I}}^{\prime
(s)}\right\Vert }=\dfrac{2.1\sin \left( \left\vert \left( \mathbf{g}_{%
\mathtt{III}}^{\prime (s)},\mathbf{g}_{\mathtt{II}}^{\prime (s)}\right)
\right\vert \right) \left\Vert \mathbf{g}_{\mathtt{II}}^{\prime
(s)}\right\Vert }{\sin \left( \left\vert \left( \mathbf{g}_{\mathtt{III}%
}^{\prime (s)},\mathbf{g}_{\mathtt{I}}^{\prime (s)}\right) \right\vert
\right) \left\Vert \mathbf{g}_{\mathtt{III}}^{\prime (s)}\right\Vert }$

$\leq \dfrac{2.1\sin \left( \frac{30\pi }{31}\right) }{\sin \left( \frac{%
17\pi }{31}\right) }\leq 0.22.$ 

Finally: $\dfrac{2.1\times \left\Vert \mathbf{g}_{\mathtt{I}}^{\prime
(s)}\right\Vert }{\left\Vert \mathbf{g}_{\mathtt{III}}^{\prime
(s)}\right\Vert }$ $\leq 2.1\times 0.1\leq 0.21$.

The three last inequalities lead to the sufficient condition: (Ineq \ref%
{Ineg3T}). Then we have proved: $\frac{A^{(s)}}{\left\Vert \mathbf{g}_{%
\mathtt{III}}^{\prime (s)}\right\Vert ^{2}}\leq \frac{A^{(s+1)}}{\left\Vert 
\mathbf{g}_{\mathtt{III}}^{\prime (s+1)}\right\Vert ^{2}}$.

If $\left( s+1\right) \in T,$ the proof of  (Ineq~\ref{Ineq MSL}) is
finished. If now $\left( s+1\right) \notin T$, then we have both $\left\Vert 
\mathbf{g}_{\mathtt{III}}^{\prime (s+1)}\right\Vert \leq \left\Vert \mathbf{g%
}_{\mathtt{III}}^{\prime (s+2)}\right\Vert $, and $A^{(s)}\leq A^{(s+2)}$
then$\frac{A^{(s)}}{\left\Vert \mathbf{g}_{\mathtt{III}}^{\prime
(s)}\right\Vert ^{2}}\leq \frac{A^{(s+2)}}{\left\Vert \mathbf{g}_{\mathtt{III%
}}^{\prime (s+2)}\right\Vert ^{2}}$ and again  the proof of (Ineq~\ref{Ineq
MSL}) is finished.

In the same way, as long as $\left( s+i\right) \notin T$, for $i=1,2,...$,
we have:

$\frac{A^{(s)}}{\left\Vert \mathbf{g}_{\mathtt{III}}^{\prime (s)}\right\Vert
^{2}}\leq \frac{A^{(s+2)}}{\left\Vert \mathbf{g}_{\mathtt{III}}^{\prime
(s+2)}\right\Vert ^{2}}\leq \frac{A^{(s+3)}}{\left\Vert \mathbf{g}_{\mathtt{%
III}}^{\prime (s+3)}\right\Vert ^{2}}\leq ...\leq \frac{A^{(s+i)}}{%
\left\Vert \mathbf{g}_{\mathtt{III}}^{\prime (s+i)}\right\Vert ^{2}}\leq
...\leq \frac{A^{(s^{\prime })}}{\left\Vert \mathbf{g}_{\mathtt{III}%
}^{\prime (s^{\prime })}\right\Vert ^{2}},$ $s^{\prime }$ being the
successor of $s$ in $T$.

Then $\dfrac{A^{(s)}}{\left\Vert \mathbf{g}_{\mathtt{III}}^{\prime
(s)}\right\Vert ^{2}}\leq \dfrac{A^{(s^{\prime })}}{\left\Vert \mathbf{g}_{%
\mathtt{III}}^{\prime (s^{\prime })}\right\Vert ^{2}}$ and the conclusion 
(Ineq~\ref{Ineq MSL})  is reached in the case $\left( \mathtt{II}\
\smallsetminus \ \mathtt{I}\right) $.

The reasoning is similar in all the four cases. 

\textbf{Then the conclusion of the Monotonic Sequence Lemma is established.
So is the Geometrical Theorem, and also the Dirichlet Theorem and the
Lagrange Theorem}, but the last one only in a special case. We have to prove
it generally.

\section{Lagrange Theorem from Dirichlet properties: complete demonstration}

Now, using the Theorem on Dirichlet Properties, we prove the Lagrange
Theorem with the help of some Definitions, Lemma, Propositions. First we
give the statements, then the proofs.

\subsection{Definition and Statements of \S 5}

\begin{definition}[Max-Dirichlet Property]
\ It will be said that a sequence $\left( \mathbf{P}^{(s)}\right) =\left( 
\mathbf{p}_{0}^{(s)},\mathbf{p}_{1}^{(s)},\mathbf{p}_{2}^{(s)}\right) $ of
triplets of integer vectors has the max-Dirichlet Property concerning $%
\mathbb{D}=\mathbb{R}\mathbf{X}$ (resp: $\mathbb{P}=\mathbf{X}^{\perp }$) if
there exists an infinite subset $S$ of $\mathbb{N}$ such that: $\underset{%
s\in S}{\sup }\left[ \left( \underset{i=0,1,2}{\max }\left\Vert \mathbf{p}%
_{i}^{\prime (s)}\right\Vert \right) ^{2}\left( \underset{i=0,1,2}{\max }%
\left\Vert \mathbf{p}^{\prime \prime }{}_{i}^{(s)}\right\Vert \right) \right]
<+\infty $ , with

$\underset{s\rightarrow +\infty ,s\in S}{\lim }\left( \underset{i=0,1,2}{%
\max }\left\Vert \mathbf{p}_{i}^{\prime (s)}\right\Vert \right) =0$ (resp: $%
\underset{s\rightarrow +\infty ,s\in S}{\lim }\left( \underset{i=0,1,2}{\max 
}\left\Vert \mathbf{p}^{\prime \prime }{}_{i}^{(s)}\right\Vert \right) =0$).
\end{definition}

\begin{lemma}[Polarity and Dirichlet Property]
Let $\left( \mathbf{P}^{(s)}\right) $ be a sequence of integer matrices, 
\emph{all with the same determinant }$D>0$\emph{, up to the sign}, id est,
for each $s\in \mathbb{N}\boldsymbol{,}$ $\det \left( \mathbf{P}%
^{(s)}\right) =\varepsilon ^{\left( s\right) }D$, with $\varepsilon ^{\left(
s\right) }\in \left\{ -1;1\right\}$. Let $\left( \mathbf{P}^{(s)}\right)
^{\ast }$ be the polar matrix of $\mathbf{P}^{(s)}$. Let $\mathbf{X}$ be a
triplet of rationally independent real numbers.

If the sequence $\left( \mathbf{P}^{(s)}\right) $ has the max-Dirichlet
Property for the plane $\mathbb{P}=\mathbf{X}^{\perp }$ then the sequence $%
\left( D.\left( \mathbf{P}^{(s)}\right) ^{\ast }\right) $ has the
max-Dirichlet Property for the line $\mathbb{D}=\mathbb{R}\mathbf{X}$.
\end{lemma}

\begin{proposition}
Let $\theta $ be a real root of a third degree irreducible polynomial $%
P\left( t\right) =t^{3}-mt-n$, with $m$ and $n$\ \emph{rationals.} Let the
vector $\mathbf{\Theta }$ be $\mathbf{\Theta =~}^{\textsc{T}}\left(
1,\theta ,\theta ^{2}\right) $. Let $\mathbf{R}$ be any\emph{\ rational}
matrix, with $\det \left( \mathbf{R}\right) \neq 0$, such that $C\mathbf{R}$
has integral coefficients, with $C$ an integer. Let $\mathbf{X}$\ be: $%
\mathbf{X}=\mathbf{R\Theta }$ = $\mathbf{R\ }^{\textsc{T}}\left( 1,\theta
,\theta ^{2}\right) $. Let's suppose that: $\mathbf{X=~}^{\textsc{T}}\left( x_{0},x_{1},x_{2}\right)$, with $0<x_{0}<x_{1}<x_{2}$. Let $\left( \mathbf{G}^{(s)}\right) $ be the sequence
of integer matrices generated by the Smallest Vector Algorithm with initial
value $\mathbf{X}=\mathbf{R\Theta }$.

Then the sequence $\left( \mathbf{G}^{(s)}\right) $ has the max-Dirichlet
Property for the approximation of $\mathbb{P=}\mathbf{X}^{\perp }$ and $%
\left( C.^{\textsc{T}}\mathbf{RG}^{(s)}\right) $ has the max-Dirichlet
Property for the approximation of $\mathbf{\Pi }=\mathbf{\Theta }^{\perp }$. Moreover, there exists an integer $A$ such that the matrices $\mathbf{A}%
^{(s)}=A\mathbf{R}^{-1}\left( \mathbf{G}^{(s)}\right) ^{\ast }$ are integer,
and such that the sequence $\left( \mathbf{A}^{(s)}\right) $ has also the
max-Dirichlet Property for the approximation of $\mathbf{\Delta }=\mathbb{R}%
\mathbf{\Theta }$.
\end{proposition}

\begin{proposition}
Let $\theta $ and $\mathbf{\Theta }$ be like in the previous Proposition.
Let $\mathbf{\Delta }$ be: $\mathbf{\Delta }=\mathbb{R}\mathbf{\Theta }$.
Let $\left( \mathbf{A}^{(s)}\right) $ be any sequence of integer matrices
having the max-Dirichlet Property for $\mathbf{\Delta }$, and all having the
same determinant $D>0$, up to the sign. Let $\left( \mathbf{J}^{(s)}\right) $
be the polar matrices of the $\left( \mathbf{A}^{(s)}\right) $. Then there
exists a sequence of \emph{rational} matrices $\left( \mathbf{M}%
^{(s)}\right) $ and an integer $Q$, which depends only on $m$ and $n$, such
that:

\begin{itemize}
\item For each $s$, $\mathbf{\Theta }$ is an eigenvector for $\mathbf{M}%
^{(s)}$;

\item $\underset{s\rightarrow +\infty }{\lim \inf }\left( \left\Vert DQ\ ^{%
\textsc{T}}\mathbf{M}^{(s)}\mathbf{J}^{(s)}\right\Vert \right) <+\infty $,
the matrices $\left( DQ\ ^{\textsc{T}}\mathbf{M}^{(s)}\mathbf{J}%
^{(s)}\right) $ having integral coefficients.
\end{itemize}
\end{proposition}

\begin{proposition}
Let $\theta $ be, like in the two previous propositions, a real root of a
third degree irreducible polynomial $P\left( t\right) =t^{3}-mt-n$, with $m$
and $n$\ \emph{rationals.} Let $\mathbf{X}=\ ^{\textsc{T}}\left(
x_{0},x_{1},x_{2}\right) $ be a free triplet of three positive real numbers from the
ring $\mathbb{Q}\left[ \theta \right] $. Then the Smallest Vector Algorithm
applied on $\mathbf{X}$ makes a loop: there exist integers $s$ and $t$, $%
s\neq t$, and a real number $\lambda $ such that:

$\left( \left\Vert \mathbf{g}^{\prime \prime }{}_{0}^{(s)}\right\Vert
,\left\Vert \mathbf{g}^{\prime \prime }{}_{1}^{(s)}\right\Vert ,\left\Vert 
\mathbf{g}^{\prime \prime }{}_{2}^{(s)}\right\Vert \right) =\lambda \left(
\left\Vert \mathbf{g}^{\prime \prime }{}_{0}^{(t)}\right\Vert ,\left\Vert 
\mathbf{g}^{\prime \prime }{}_{1}^{(t)}\right\Vert ,\left\Vert \mathbf{g}%
^{\prime \prime }{}_{2}^{(t)}\right\Vert \right) $.
\end{proposition}

\subsection{Demonstration of the Lemma}

Let's denote:

$\left( \mathbf{P}^{(s)}\right) ^{\ast }=\mathbf{Q}^{(s)}=$ $\left( \mathbf{q%
}_{0}^{(s)},\mathbf{q}_{1}^{(s)},\mathbf{q}_{2}^{(s)}\right) $; then, for
any direct circular permutation $\left( i,j,k\right) $ of $\left(
0,1,2\right) $, forgetting the indices $^{(s)}$ we have:

$\mathbf{q}_{i}=\frac{\varepsilon }{D}\left( \mathbf{p}_{j}\wedge \mathbf{p}%
_{k}\right) $; $\mathbf{q}_{i}^{\prime }=\frac{\varepsilon }{D}\left( 
\mathbf{p}_{j}^{\prime \prime }\wedge \mathbf{p}_{k}^{\prime }+\mathbf{p}%
_{j}^{\prime }\wedge \mathbf{p}_{k}^{\prime \prime }\right) $;$\ \mathbf{q}%
_{i}^{\prime \prime }=\frac{\varepsilon }{D}\left( \mathbf{p}_{j}^{\prime
}\wedge \mathbf{p}_{k}^{\prime }\right) $.

Let's denote: $\underset{i=0,1,2}{\max }\left\Vert \mathbf{p}_{i}^{\prime
(s)}\right\Vert =\mu ^{\prime (s)}$; $\underset{i=0,1,2}{\max }\left\Vert 
\mathbf{p}^{\prime \prime }{}_{i}^{(s)}\right\Vert =\mu ^{\prime \prime (s)}$%
.

By hypothesis, $\left( \mu ^{\prime (s)}\right) ^{2}\mu ^{\prime \prime
(s)}<L$ holds for some $L$ and for every $s$ in the infinite set $S$. For
each $i$, and any $s\in S$, we have: $\left\Vert \mathbf{q}_{i}^{\prime
(s)}\right\Vert \leq \tfrac{2}{D}\mu ^{\prime (s)}\mu ^{\prime \prime (s)}$,
and $\left\Vert \mathbf{q}^{\prime \prime }{}_{i}^{(s)}\right\Vert \leq 
\frac{1}{D}\left( \mu ^{\prime (s)}\right) ^{2}$. Then

$\left( \underset{i=0,1,2}{\max }\left\Vert \mathbf{q}_{i}^{\prime
(s)}\right\Vert \right) ^{2}\left( \underset{i=0,1,2}{\max }\left\Vert 
\mathbf{q}^{\prime \prime }{}_{i}^{(s)}\right\Vert \right) \leq \dfrac{4}{%
D^{3}}\left( \left( \mu ^{\prime (s)}\right) ^{2}\mu ^{\prime \prime
(s)}\right) ^{2}\leq \dfrac{4L^{2}}{D^{3}}$, and then: $\left( \underset{i=0,1,2}{\max }\left\Vert D\mathbf{q}_{i}^{\prime
(s)}\right\Vert \right) ^{2}\left( \underset{i=0,1,2}{\max }\left\Vert D%
\mathbf{q}^{\prime \prime }{}_{i}^{(s)}\right\Vert \right) \leq 4L^{2}$.

We have to prove in addition that the limit of the first factor is null.

By hypothesis: $\underset{s\rightarrow +\infty ,s\in S}{\lim }\left( 
\underset{i=0,1,2}{\max }\left\Vert \mathbf{p}^{\prime \prime
}{}_{i}^{(s)}\right\Vert \right) =0.$ This implies:

$\underset{s\rightarrow +\infty ,s\in S}{\lim }\left( \underset{i=0,1,2}{%
\max }\left\Vert \mathbf{p}_{i}^{\prime (s)}\right\Vert \right) =+\infty $.
Otherwise, the set of all the integer vectors $\mathbf{p}_{i}^{(s)}$, with $%
s $ in some infinite set $T\subset S,$ would be bounded, and then finite.
Then the sequence $\left( \underset{i=0,1,2}{\max }\left\Vert \mathbf{p}%
^{\prime \prime }{}_{i}^{(s)}\right\Vert \right) _{s\in T}$ would have a
non-null minimum. Contradiction!

Then we have: $\underset{s\rightarrow +\infty ,s\in S}{\lim }\mu ^{\prime
(s)}=+\infty $. But we also have, for every $s\in S$,

$\left\Vert \mathbf{q}_{i}^{\prime (s)}\right\Vert \leq \tfrac{2}{D}\mu
^{\prime (s)}\mu ^{\prime \prime (s)}=\tfrac{2}{D}\dfrac{\left( \mu ^{\prime
(s)}\right) ^{2}\mu ^{\prime \prime (s)}}{\mu ^{\prime (s)}}\leq \tfrac{2L}{D%
}\dfrac{1}{\mu ^{\prime (s)}}.$

Then $\underset{s\rightarrow +\infty ,s\in S}{\lim }\left( \underset{i=0,1,2}%
{\max }\left\Vert D\mathbf{q}_{i}^{\prime (s)}\right\Vert \right) =0$.

We have established that the sequence $\left( D.\left( \mathbf{P}%
^{(s)}\right) ^{\ast }\right) $ has the max Dirichlet Property for the line $%
\mathbb{D}=\mathbb{R}\mathbf{X}$.

\subsection{Demonstration of Proposition 1}

Let $\theta $ be a real root of a third degree irreducible polynomial $%
P\left( t\right) =t^{3}-mt-n$, where $m$ and $n$\ are \emph{rationals}.\newline
For the initial value $\mathbf{X}=\mathbf{R\Theta }$ = $\mathbf{R}$~$^{%
\text{T}}\left( 1,\theta ,\theta ^{2}\right) $, let $\left( \mathbf{B}%
^{(s)}\right) $ and  
$\left( \mathbf{G}^{(s)}\right) $ be the sequences of
integral matrices generated by the Smallest Vector Algorithm.

First, we establish that the couple $\left( \mathbb{P},\mathbb{D}\right)
=\left( \mathbf{X}^{\perp },\mathbb{R}\mathbf{X}\right) $ is badly
approximable, in the sense of the Lemma 8 and the following Definition in
Subsection 3.2.

By a classical theorem that we have already cited, (see \cite{CasselsDA}
(Cassels), Theorem III, page 79, statement (2)) the couple $\left( \mathbf{%
\Theta }^{\perp },\mathbb{R}\mathbf{\Theta }\right) $ is badly approximable.
Then: $\underset{\mathbf{k}\text{ integer }\neq \mathbf{0}}{\inf }\left[
\left\vert \mathbf{k}\bullet \mathbf{\Theta }\right\vert .\left\Vert \mathbf{%
k}\right\Vert ^{2}\right] >0$.

Let's suppose that the$\mathbb{\ }$couple $\left( \mathbb{P},\mathbb{D}%
\right) =\left( \mathbf{X}^{\perp },\mathbb{R}\mathbf{X}\right) $ is NOT
badly approximable. Then we would have: $\underset{\mathbf{h}\text{ integer }%
\neq \mathbf{0}}{\inf }\left[ \left\vert \mathbf{h}\bullet \mathbf{R\Theta }%
\right\vert .\left\Vert \mathbf{h}\right\Vert ^{2}\right] =0$;

then $\underset{\mathbf{h}\text{ integer }\neq \mathbf{0}}{\inf }\left[
\left\vert \left( ^{\textsc{T}}\mathbf{Rh}\right) \bullet \mathbf{\Theta }%
\right\vert .\left\Vert ^{\textsc{T}}\mathbf{Rh}\right\Vert ^{2}\right] =0$%
. Let $q$ be an integer number such that $q\mathbf{R}$ has integral
coefficient; then

$\underset{\mathbf{h}\text{ integer }\neq \mathbf{0}}{\inf }\left[
\left\vert \left( q\ ^{\textsc{T}}\mathbf{Rh}\right) \bullet \mathbf{\Theta 
}\right\vert .\left\Vert q\ ^{\textsc{T}}\mathbf{Rh}\right\Vert ^{2}\right]
=0$, with $\left( q\ ^{\textsc{T}}\mathbf{Rh}\right) $ non-null integer. Then $\underset{\mathbf{k}\text{ integer }\neq \mathbf{0}}{\inf }\left[
\left\vert \mathbf{k}\bullet \mathbf{\Theta }\right\vert .\left\Vert \mathbf{%
k}\right\Vert ^{2}\right] =0$. Contradiction. Then the couple $\left( 
\mathbb{P},\mathbb{D}\right) =\left( \mathbf{X}^{\perp },\mathbb{R}\mathbf{X}%
\right) $ is badly approximable.

By the Dirichlet Properties Theorem of Subsection 1.3., part b), the
sequence $\left( \mathbf{G}^{(s)}\right) $, generated from $\mathbf{X}$ by
the Smallest Vector Algorithm, have the max-Dirichlet Property for the
approximation of $\mathbb{P=}\mathbf{X}^{\perp }$. There exists an infinite
subset $S$ of $\mathbb{N}$ such that \newline
 $\underset{s\in S}{\sup }\left[ \left( 
\underset{i=0,1,2}{\max }\left\vert \mathbf{g}_{i}^{(s)}\bullet \mathbf{%
R\Theta }\right\vert \right) \left( \underset{i=0,1,2}{\max }\left\Vert 
\mathbf{g}_{i}^{(s)}\right\Vert \right) ^{2}\right] <+\infty $,

with $\underset{s\rightarrow +\infty ,s\in S}{\lim }\left( \underset{i=0,1,2}%
{\max }\left\vert \mathbf{g}_{i}^{(s)}\bullet \mathbf{R\Theta }\right\vert
\right) =0$. Then:

$\underset{s\in S}{\sup }\left[ \left( \underset{i=0,1,2}{\max }\left\vert
\left( C.^{\textsc{T}}\mathbf{Rg}_{i}^{(s)}\right) \bullet \mathbf{\Theta }%
\right\vert \right) \left( \underset{i=0,1,2}{\max }\left\Vert C.^{\text{T%
}}\mathbf{Rg}_{i}^{(s)}\right\Vert \right) ^{2}\right] <+\infty $,

with $\underset{s\rightarrow +\infty ,s\in S}{\lim }\left( \underset{i=0,1,2}%
{\max }\left\vert \left( C.^{\textsc{T}}\mathbf{Rg}_{i}^{(s)}\right)
\bullet \mathbf{\Theta }\right\vert \right) =0$.

That means that the sequence $\left( C.^{\textsc{T}}\mathbf{RG}%
^{(s)}\right) $ has the max-Dirichlet Property for the approximation of $%
\mathbf{\Pi }=\mathbf{\Theta }^{\perp }$.

But the $\left( C.^{\textsc{T}}\mathbf{RG}^{(s)}\right) $ have all the same
determinant $\left( C^{3}\det \left( \mathbf{R}\right) \right) ,$ up to the
sign. Let $A$ be $A=C^{3}\det \left( \mathbf{R}\right) .$ Then, by the
previous Lemma, the sequence $\left( A\left( ^{\textsc{T}}\mathbf{RG}%
^{(s)}\right) ^{\ast }\right) =$ $\left( A\mathbf{R}^{-1}\left( \mathbf{G}%
^{(s)}\right) ^{\ast }\right) $ of integer matrices has the max-Dirichlet
Property for the approximation of $\mathbf{\Delta }=\mathbb{R}\mathbf{\Theta 
}$.

\subsection{Demonstration of Proposition 2}

Let $\left( \mathbf{A}^{(s)}\right) $ be a sequence of integer matrices
having the max-Dirichlet Property for $\mathbf{\Delta }=\mathbb{R}\mathbf{%
\Theta }=$

$\mathbb{R\ }^{\textsc{T}}\left( 1,\theta ,\theta ^{2}\right) $, with $\theta ^{3}=m\theta +n$.

We suppose that the matrices $\mathbf{A}^{(s)}$ have all the same
determinant $D>0$\emph{,} up to the sign, which means that for each $s\in 
\mathbb{N}\boldsymbol{,}$ $\det \left( \mathbf{A}^{(s)}\right) =\varepsilon
^{\left( s\right) }D$, with $\varepsilon ^{\left( s\right) }\in \left\{
-1;1\right\} .$

Let $\left( \mathbf{a}_{0}^{(s)},\mathbf{a}_{1}^{(s)},\mathbf{a}%
_{2}^{(s)}\right) $\ be the column vectors of $\mathbf{A}^{(s)}$. We choose
one of these three vectors, say $\mathbf{a}_{0}^{(s)}$, which will be more
simply denoted: $\mathbf{a}^{(s)}:=\mathbf{a}_{0}^{(s)}$. Let's define its
coordinates by: $\mathbf{a}^{(s)}=\ ^{\textsc{T}}\left(
a_{x}^{(s)},a_{y}^{(s)},a_{z}^{(s)}\right) $.

Let's denote: $\mu ^{\prime (s)}=\underset{i=0,1,2}{\max }\left\Vert 
\mathbf{a}_{i}^{\prime (s)}\right\Vert $ and $\mu ^{\prime \prime (s)}=$ $%
\underset{i=0,1,2}{\max }\left\Vert \mathbf{a}^{\prime \prime
}{}_{i}^{(s)}\right\Vert $. We suppose that there exist an infinite set $S$
of integers and a real number $L$ such that for each $s\in S,$ the
inequality $\left( \mu ^{\prime (s)}\right) ^{2}\mu ^{\prime \prime (s)}<L$
holds. Let $s$ be any element of $S$. From now on, we may omit the indices $%
^{(s)}$.

The notation $\mathbf{M}=\mathbf{M}^{(s)}$ will denote the rational matrix

$\mathbf{M}^{(s)}=\left( 
\begin{array}{ccc}
-m.a_{x}+a_{z} & a_{y} & a_{x} \\ 
n\cdot a_{x} & a_{z} & a_{y} \\ 
n\cdot a_{y} & n.a_{x}+m.a_{y} & a_{z}%
\end{array}%
\right) $. If $Q$ is a natural such that $Qm$ and $Qn$ are integers, then $Q%
\mathbf{M}^{(s)}$ has integral coefficients.

We have: $\mathbf{M\Theta =M}\left( 
\begin{array}{l}
1 \\ 
\theta \\ 
\theta ^{2}%
\end{array}%
\right) =\left( -m.a_{x}+a_{z}+a_{y}\theta +a_{x}\theta ^{2}\right) \left( 
\begin{array}{l}
1 \\ 
\theta \\ 
\theta ^{2}%
\end{array}%
\right) $.

Let's denote by $\lambda $ the following element of $\mathbb{Z}\left[ \theta %
\right] $:

$\lambda:=\left( -m.a_{x}+a_{z}+a_{y}\theta +a_{x}\theta ^{2}\right) $.
Then we have $\mathbf{M\Theta }=\lambda \mathbf{\Theta }$; $\lambda $ is an
eigenvalue of $\mathbf{M}$ with eigenvector $\mathbf{\Theta }$ .

Let $\left( \mathbf{J}^{(s)}\right) $ the polar matrices of the $\left( 
\mathbf{A}^{(s)}\right) $. We consider the sequence of the matrices $\mathbf{%
\Pi }^{(s)}=\mathbf{\Pi }=$ $^{\textsc{T}}\mathbf{MJ}$.

Let $\left( \mathbf{a}^{\#\#},\mathbf{a}^{\#},\mathbf{a}\right) $ be the
three column vectors of $\mathbf{M}$. Then, with scalar products: $\mathbf{%
\Pi}=\ ^{\textsc{T}}\mathbf{MJ=}\left( 
\begin{array}{ccc}
\mathbf{a}^{\#\#}\bullet \mathbf{j}_{0} & \mathbf{a}^{\#\#}\bullet \mathbf{j}%
_{1} & \mathbf{a}^{\#\#}\bullet \mathbf{j}_{2} \\ 
\mathbf{a}^{\#}\bullet \mathbf{j}_{0} & \mathbf{a}^{\#}\bullet \mathbf{j}_{1}
& \mathbf{a}^{\#}\bullet \mathbf{j}_{2} \\ 
\mathbf{a}\bullet \mathbf{j}_{0} & \mathbf{a}\bullet \mathbf{j}_{1} & 
\mathbf{a}\bullet \mathbf{j}_{1}%
\end{array}%
\right) =\left( \pi _{i,j}\right) $, say$,$ with $i=1,2,3;$ $j=1,2,3$.

We now have to find an upper bound for each of the $\left| \pi _{i,j}\right| 
$.

We have: $\mathbf{a}^{\#}=\mathbf{Q}^{\#}\left( 
\begin{array}{c}
a_{x} \\ 
a_{y} \\ 
a_{z}%
\end{array}%
\right) =\mathbf{Q}^{\#}\mathbf{a}$, with $\mathbf{Q}^{\#}=\left( 
\begin{array}{ccc}
0 & 1 & 0 \\ 
0 & 0 & 1 \\ 
n & m & 0%
\end{array}%
\right) $ and:

$\mathbf{a}^{\#\#}=\mathbf{Q}^{\#\#}\left( 
\begin{array}{c}
a_{x} \\ 
a_{y} \\ 
a_{z}%
\end{array}%
\right) =\mathbf{Q}^{\#\#}\mathbf{a}$, with $\mathbf{Q}^{\#\#}=\left( 
\begin{array}{ccc}
-m & 0 & 1 \\ 
n & 0 & 0 \\ 
0 & n & 0%
\end{array}%
\right) $.

We have: $\mathbf{Q}^{\#}\mathbf{\Theta }=\theta \mathbf{\Theta }$, and $%
\mathbf{Q}^{\#\#}\mathbf{\Theta }=\left( -m+\theta ^{2}\right) \mathbf{%
\Theta }$.

First let's consider the $\left| \mathbf{a}^{\#}\bullet \mathbf{j}%
_{i}\right| $.

Let $\mathbf{\nu }$ be: $\mathbf{\nu:=}\dfrac{\mathbf{\Theta }}{\left\Vert 
\mathbf{\Theta }\right\Vert }$. Then also $\mathbf{Q}^{\#}\mathbf{\nu =}%
\theta \mathbf{\nu }$.

With always the same kind of notations, we have $\mathbf{a}^{\#}=\mathbf{a}%
^{\#\prime }+\mathbf{a}^{\#\prime \prime }$, and:

$\mathbf{a}^{\#}\bullet \mathbf{j}_{i}=\left( \mathbf{a}^{\#\prime }+\mathbf{%
a}^{\#\prime \prime }\right) \bullet \left( \mathbf{j}_{i}^{\prime }+\mathbf{%
j}_{i}^{\prime \prime }\right) =\mathbf{a}^{\#\prime }\bullet \mathbf{j}%
_{i}^{\prime }+\mathbf{a}^{\#\prime \prime }\bullet \mathbf{j}_{i}^{\prime
\prime }=\mathbf{a}^{\#\prime }\bullet \mathbf{j}_{i}^{\prime }+\mathbf{a}%
^{\#\prime \prime }\bullet \mathbf{j}_{i}^{\prime \prime }$, with
$\mathbf{j}_{i}=\frac{\varepsilon }{D}\left( \mathbf{a}_{j}\wedge \mathbf{a}%
_{k}\right) $; $\mathbf{j}_{i}^{\prime }=\frac{\varepsilon }{D}\left( 
\mathbf{a}_{j}^{\prime \prime }\wedge \mathbf{a}_{k}^{\prime }+\mathbf{a}%
_{j}^{\prime }\wedge \mathbf{a}_{k}^{\prime \prime }\right) $; $\mathbf{j}%
_{i}^{\prime \prime }=\frac{\varepsilon }{D}\left( \mathbf{a}_{j}^{\prime
}\wedge \mathbf{a}_{k}^{\prime }.\right) $. Then: 
\begin{equation}
\left\vert \mathbf{a}^{\#}\bullet \mathbf{j}_{i}\right\vert \leq \frac{1}{D}%
\left( \left\Vert \mathbf{a}^{\#\prime }\right\Vert ~\mu ^{\prime \prime
}\mu ^{\prime }++\left\Vert \mathbf{a}^{\#\prime }\right\Vert ~\mu ^{\prime
}\mu ^{\prime \prime }+\left\Vert \mathbf{a}^{\#\prime \prime }{}\right\Vert
~\left( \mu ^{\prime }\right) ^{2}\right)  \label{IneqDiese5A}
\end{equation}%
and we have also $\mathbf{a}^{\#}=\mathbf{Q}^{\#}\mathbf{a}=\mathbf{Q}%
^{\#}\left( \mathbf{a}_{0}^{\prime }+\left\Vert \mathbf{a}_{0}^{\prime
\prime }\right\Vert \mathbf{\nu }\right) =\mathbf{Q}^{\#}\mathbf{a}%
_{0}^{\prime }+\left\Vert \mathbf{a}_{0}^{\prime \prime }\right\Vert \theta 
\mathbf{\nu }$.

This proves first that the distance $\left\Vert \mathbf{a}^{\#\prime
}\right\Vert $ between $\mathbf{a}^{\#}$ and $\mathbb{D}$ is less than $%
\left\Vert \mathbf{Q}^{\#}\mathbf{a}_{0}^{\prime }\right\Vert$: 
\begin{equation}
\left\Vert \mathbf{a}^{\#\prime }\right\Vert \leq \left\Vert \mathbf{Q}^{\#}%
\mathbf{a}_{0}^{\prime }\right\Vert \leq \left\Vert \mathbf{Q}%
^{\#}\right\Vert \times \left\Vert \mathbf{a}_{0}^{\prime }\right\Vert \leq
\left\Vert \mathbf{Q}^{\#}\right\Vert \mu ^{\prime }  \label{InegPrim5}
\end{equation}%
(this using the norm of the matrix).

Moreover, we have $\mathbf{a}^{\#\prime }+\mathbf{a}^{\#\prime \prime }$ $=%
\mathbf{a}^{\#}=\mathbf{Q}^{\#}\mathbf{a}_{0}^{\prime }+\left\Vert \mathbf{a}%
_{0}^{\prime \prime }\right\Vert \theta \mathbf{\nu }$

then: $\mathbf{a}^{\#\prime \prime }=\left\Vert \mathbf{a}_{0}^{\prime
\prime }\right\Vert \theta \mathbf{\nu +\mathbf{Q}^{\#}\mathbf{a}%
_{0}^{\prime }-}$ $\mathbf{a}^{\#\prime }$. Then: 
\begin{equation}
\left\Vert \mathbf{a}^{\#\prime \prime }\right\Vert \leq \left\Vert \mathbf{a%
}_{0}^{\prime \prime }\right\Vert \theta +2\left\Vert \mathbf{Q}%
^{\#}\right\Vert \times \left\Vert \mathbf{a}_{0}^{\prime }\right\Vert \leq
\mu ^{\prime \prime }\theta +2\left\Vert \mathbf{Q}^{\#}\right\Vert \mu
^{\prime }  \label{Ineq Sec5}
\end{equation}%
Putting \ref{InegPrim5} and \ref{Ineq Sec5} in \ref{IneqDiese5A}, we obtain:

$\left\vert \mathbf{a}^{\#(s)}\bullet \mathbf{j}_{i}^{(s)}\right\vert \leq 
\frac{1}{D}\left( \mu ^{\prime (s)}\right) ^{2}\mu ^{\prime \prime
(s)}\left( 2\left\Vert \mathbf{Q}^{\#}\right\Vert +\theta \right) +\frac{2}{D%
}\left\Vert \mathbf{Q}^{\#}\right\Vert \left( \mu ^{\prime (s)}\right) ^{3}$,

and then: 
\begin{equation*}
\left\vert \mathbf{a}^{\#(s)}\bullet \mathbf{j}_{i}^{(s)}\right\vert \leq 
\dfrac{L}{D}\left( 2\left\Vert \mathbf{Q}^{\#}\right\Vert +\theta \right) +%
\frac{2}{D}\left\Vert \mathbf{Q}^{\#}\right\Vert \left( \mu ^{\prime
(s)}\right) ^{3}
\end{equation*}%
The limit of the last term is $0$.

Then the set of the $\left\vert \mathbf{a}^{\#(s)}\bullet \mathbf{j}%
_{i}^{(s)}\right\vert $, with $s$ in $S$, is bounded. A similar
demonstration shows that the $\left\vert \mathbf{a}^{\#\#(s)}\bullet \mathbf{%
j}_{i}^{(s)}\right\vert $ are also bounded and so are in an obvious way the $%
\left\vert \mathbf{a}_{0}^{(s)}\bullet \mathbf{j}_{i}^{(s)}\right\vert $.
Then the set of the $\mathbf{\Pi }^{(s)}=$ $^{\textsc{T}}\mathbf{M}^{(s)}%
\mathbf{J}^{(s)}$ is bounded. But $\mathbf{J}^{(s)}=\left( \mathbf{A}%
^{(s)}\right) ^{\ast }$, with $\det \left( \mathbf{A}^{(s)}\right) =\pm D$.
Then the matrices $D\mathbf{J}^{(s)}$ have integral coefficients. We have
seen that the matrices $Q\mathbf{M}^{(s)}$ have also integral coefficients.
In addition, the sequence $\left( DQ~^{\textsc{T}}\mathbf{M}^{(s)}\mathbf{J}%
^{(s)}\right) $ is bounded, and the proof is done.

\subsection{Demonstration of Proposition 3}

Let $\theta $ be a real root of a third degree irreducible polynomial $%
P\left( t\right) =t^{3}-mt-n$, with $m$ and $n$\ \emph{rationals.} Let $%
\mathbf{X}=\ ^{\textsc{T}}\left( x_{0},x_{1},x_{2}\right) $ be a free
triple of three positive real numbers from the ring $\mathbb{Q}\left[ \theta \right] $%
. Let $\mathbf{\Theta }$ be $\mathbf{\Theta =\ }^{\textsc{T}}\left(
1,\theta ,\theta ^{2}\right) $. Then there exists a rational matrix $\mathbf{%
R}$ , with $\det \left( \mathbf{R}\right) \neq 0$, such that $\mathbf{X}=%
\mathbf{R\Theta }$.

Then, by Proposition 1, there exists an integer $A$ such that the matrices $%
\mathbf{A}^{(s)}=A\mathbf{R}^{-1}\left( \mathbf{G}^{(s)}\right) ^{\ast }$
are integer, and such that the sequence $\left( \mathbf{A}^{(s)}\right) $
has the max-Dirichlet Property for the approximations of $\mathbf{\Delta }=%
\mathbb{R}\mathbf{\Theta }$. All the matrices $\mathbf{A}^{(s)}$ have the
same determinant, say $D>0$, up to the sign; then, by Proposition 2, there
exists a sequence of integer matrices $\left( \mathbf{M}^{(s)}\right) $ such
that, for each $s$, $\mathbf{\Theta }$ is an eigenvector for $\mathbf{M}%
^{(s)}$ and $\underset{s\rightarrow +\infty }{\lim \inf }\left( \left\Vert
DQ.^{\textsc{T}}\mathbf{M}^{(s)}\mathbf{J}^{(s)}\right\Vert \right)
<+\infty $, the matrices $\left( DQ.^{\textsc{T}}\mathbf{M}^{(s)}\mathbf{J}%
^{(s)}\right) $ having integral coefficients, with

$\mathbf{J}^{(s)}=\left( \mathbf{A}^{(s)}\right) ^{\ast }=\left( A\mathbf{R}%
^{-1}\left( \mathbf{G}^{(s)}\right) ^{\ast }\right) ^{\ast }=A^{-1}\ ^{%
\text{T}}\mathbf{RG}^{(s)}$.

There exists an infinite subset $S$ of $\mathbb{N}$, such that the set of
all the integer matrices $\left( DQ\ ^{\textsc{T}}\mathbf{M}^{(s)}\mathbf{J}%
^{(s)}\right) $ with $s$ in $S$ is bounded; then it is finite. Then there
exist $s$ and $t$, $t\neq s$, such that $^{\textsc{T}}\mathbf{M}^{(s)}{}^{%
\text{T}}\mathbf{R~G}^{(s)}$ $=\ ^{\textsc{T}}\mathbf{M}^{(t)}{}^{%
\text{T}}\mathbf{R~G}^{(t)}$. By transposition: $\left( \mathbf{B}^{(s)}\right) ^{-1}\mathbf{RM}^{(s)}=\left( \mathbf{B}%
^{(t)}\right) ^{-1}\mathbf{RM}^{(t)}$. We apply that to the column vector $%
\mathbf{\Theta }$:
$\left( \mathbf{B}^{(s)}\right) ^{-1}\mathbf{RM}^{(s)}\mathbf{\Theta }%
=\left( \mathbf{B}^{(t)}\right) ^{-1}\mathbf{RM}^{(t)}\mathbf{\Theta }$;
then, by \textquotedblright eigenvector\textquotedblright , and because $%
\mathbf{R\Theta }=\mathbf{X}$, we have $\lambda ^{(s)}\left( \mathbf{B}%
^{(s)}\right) ^{-1}\mathbf{X}=\lambda ^{(t)}\left( \mathbf{B}^{(t)}\right)
^{-1}\mathbf{X}$, then $\left( \mathbf{B}^{(s)}\right) ^{-1}\mathbf{X}=\dfrac{\lambda ^{(t)}}{%
\lambda ^{(s)}}\left( \mathbf{B}^{(t)}\right) ^{-1}\mathbf{X}$.

This reads: $\mathbf{X}^{\left( s\right) }=\lambda \mathbf{X}^{\left(
t\right) }$, with $\lambda =\frac{\lambda ^{(t)}}{\lambda ^{(s)}}$, and our
Lagrange Theorem is proved if $\mathbf{X}=\ ^{\textsc{T}}\left(
x_{0},x_{1},x_{2}\right) $ is a free triplet of three positive real numbers from the
ring $\mathbb{Q}\left[ \theta \right] $, $\theta $ being a real root of a
third degree irreducible polynomial $P\left( t\right) =t^{3}-mt-n$, with $m$
and $n$\ rationals\emph{.} Of course this case is general, as we're going to verify it.

\subsection{From Proposition 3 to the Lagrange Theorem}

This part is very quick. Let $\rho $ be a real root of a third degree
irreducible polynomial $S\left( t\right) =t^{3}-at^{2}-bt-c$, with with $%
a,b,c$ rationals. Then $\theta =\rho -\frac{a}{3}$ is a real root of a third
degree irreducible polynomial $ P\left( t\right) =t^{3}-mt-n$, with $m$ and $%
n $\ rationals. If $x_{0},x_{1},x_{2}$ are elements of $\mathbb{Q}\left[
\rho \right] $, they also belong to $\mathbb{Q}\left[ \theta \right] =%
\mathbb{Q}\left[ \rho \right] $. Then Proposition 3 implies the conclusion
of the main Lagrange Theorem (first part). The second part of the theorem
has been established in Section 2.

\section{Bibliography and Themes related to this Paper}

The work nearest to the present paper is the book by A.J. Brentjes \cite%
{Brentj}. A lot of themes are in common: the approach of the continued
fractions with matrices and linear algebra, the fact that \emph{non vectorial%
} algorithms are used, the study of angular properties and of the needling
phenomenon...Brentjes' book is mainly concerned with algebraic results, best
approximation, and (strong) convergence, rather than with "Dirichlet"
approximation, with the optimal exponent, or "Lagrange" results. However, it
contains a Lagrange-type statement, in the Corollary, page 106, but with a
lattice which is not $%
\mathbb{Z}
^{3}.$

Most multidimensional continued fractions algorithms, among those \newline
which are
additive (or subtractive, or multiplicative), are of the \emph{vectorial}
type. This means that in such an algorithm, the vector $\mathbf{X}^{(s+1)}$
depends only on $\mathbf{X}^{(s)}=\left\Vert \mathbf{X}\right\Vert \left( \mathbf{g}%
_{0}^{\prime \prime (s)},\mathbf{g}_{1}^{\prime \prime (s)},\mathbf{g}%
_{2}^{\prime \prime (s)}\right)$, in a simple way, and not on \newline
 $\left( 
\mathbf{g}_{0}^{\prime (s)},\mathbf{g}_{1}^{\prime (s)},\mathbf{g}%
_{2}^{\prime (s)}\right) $ in the plane $\mathbb{P}$. In this case, the
algorithm defines clearly a discrete dynamical system, the orbits of which
are the sequences $\left( \mathbf{X}^{(s)}\right) $. There are a lot of
interesting studies of these dynamical systems, by Fritz Schweiger, J.C.
Lagarias and many others, but our algorithm, like Brentjes' one, is \emph{%
non vectorial}, and different techniques are used. With such non vectorial
algorithms, results \emph{everywhere }may be obtained. With vectorial ones,
most of the results are obtained \emph{almost everywhere}.

Apart from Brentjes' book, there is another treatise by Fritz Schweiger on
multidimensional continued fractions \cite{Schwei}. It is very complete and
presents the general Brentjes' algorithms, but deals mainly with vectorial
algorithms and dynamical systems.

The continued fractions are only a tool in the theory of Diophantine
Approximation. Here we use two theorems by Minkowski in Geometry of Numbers.
The references in these fields are for instance: \cite{Cassels}, \cite%
{CasselsDA}, \cite{Koksma}, and \cite{Schmidt}.

In the area of \emph{algorithms }which aim to \emph{best approximation},
apart from the specific Brentjes' algorithm, we may cite the Furtw\"{a}%
ngler's algorithm \cite{Furtw} (an error was pointed out by K.M. Briggs, see
his paper), which inspired Keith Briggs \cite{BriggsF} and Vaughan Clarkson
\ \cite{ClarksonPM}; see also the Ph. D. thesis of V. Clarkson: \cite%
{Clarkson}.

There are some studies of the matrices of best approximations, which could
be connected to our work: By J.C. Lagarias: \cite{Lagarias}, and \cite%
{LagariasOne}, and a review by N.G. Moshchevitin: \cite{Moshchevitin}.

In the present paper, the result on best approximations is the Prism Lemma,
at the beginning of Section 3. It is an easy result, but perhaps it
clarifies the problem. It is more efficient if the hexagon it involves is
balanced, and we have some results in this direction in this paper.

J.C. Lagarias has also build in \cite{LagariasALG} a very interesting
algorithm, which is additive but not positive, and which provides best
approximations. See
also the very complete paper by N. Chevallier: \cite{Chevallier}.

The LLL algorithm (named after A.K. Lenstra, H.W. Lenstra, L. Lov\`{a}sz) is
very efficient in Number Theory. It provides good approximations, and even
Dirichlet approximations, with the optimal exponent: see \cite{SmeetsBos},
by W. Bosma and I. Smeets. But maybe it is not designed to possess
approximation properties with \emph{triplets }of integer vectors, nor \emph{%
Lagrange} properties, as the Smallest Vector Algorithm does.

There is an another kind of \emph{Multidimensional Continued Fractions},
very different from the additive (i.e. subtractive) ones we have considered
until now. These other constructions use \emph{stars of sails}$,$ obtained
from hyperplanes and pyramids in $\mathbb{R}^{k}$. The original idea is due
to K. Klein, H. Minkowski, and G. F. Voronoi. V.\ I. Arnold renewed the
interest toward this theory: \cite{Arnold}. "Lagrange" results seem to have
been obtained, by G. Lachaud, \cite{Lachaud}, E. Korkina \cite{Korkina}, or
O.N. German and E.L Lakshtanov: \cite{German}. But their statements don't
seem as simple as the Theorem 1\ of the present work. In the cited paper,
V.I. Arnold has written:

"The attempts to generalize to higher dimensions the \emph{algorithm }%
(emphasized by V.I. Arnold) of continued fractions lead to complicate and
ugly theories. For instance the sail corresponding to a cubical irrational
number is a double-periodic surface. However the algorithms define instead
of this surface a \emph{path }on it. [...] the path is not periodic at all
and looks like a rather chaotic object; it is unclear how to describe the
cubic irrationals in terms of the combinatorics of this path".

We can just hope that Arnold was only partly right. It would be interesting
to study the relation between the regularities we have pointed out in the "chaotic" paths generated
by our algorithm for cubic numbers, and the symmetries of the corresponding
sails.

\section{Acknowledgments}

I want to thank Professor Fritz Schweiger for his generous mathematical
help, his encouragements, his numerous and accurate readings of my papers
and his remarks. I also owe my gratitude to Professor Eug\`{e}ne Dubois, who
has given a lot of his time to read a previous version of this paper and to
Professor Michel Mend\`{e}s-France, in Bordeaux I University. I also thank
the anonymous referee, for having pointed out several gaps and mistakes in
previous versions of some demonstrations. A friendly thank to my Dax
colleague Philippe Paya, for all the fruitful talks and interesting remarks.

\end{document}